%% file: main.tex
\documentclass[sn-mathphys-num,pdflatex]{sn-jnl}%

\input{preamble}

\newcommand{\cost}{F}
\newcommand{\NMV}{R}

\newcommand{\Ninfty}{\mathbb{N}_{\infty}^{\#}}
\newcommand{\constraint}{C}

\newcommand{\smoothingDomain}{D}
\newcommand{\ratio}{\theta}

\begin{document}

\title{\Large Asymptotic Analysis of Gradient Mapping-type Stationarity Measure for the Sum of Nonconvex Nonsmooth Functions and Applications to Proximal Gradient-type Algorithms}

\date{\today}

\author*[1]{\fnm{Keita} \sur{Kume}}\email{kume@sp.ict.e.titech.ac.jp}

\author[1]{\fnm{Isao} \sur{Yamada}}\email{isao@sp.ict.e.titech.ac.jp}

\affil[1]{\centering\orgdiv{Dept. of Information and Communications Engineering},\\ \orgname{Institute of Science Tokyo}, \orgaddress{\street{2-12-1}, \city{Ookayama}, \postcode{152-8550}, \state{Tokyo}, \country{Japan}}}

\abstract{
  We propose a gradient mapping-type stationarity measure for the sum of two possibly nonsmooth nonconvex functions.
  Under suitable regularity assumptions, we show that Fr\'echet or proximal stationarity can be characterized through asymptotic vanishing of the proposed measure along some convergent sequence.
  We also establish analogous asymptotic results for a smoothing-based variant of the measure, which enables us to combine the proposed framework with smoothing techniques developed for nonsmooth optimization.
  As an application of our analysis of the stationarity measure, we provide an affirmative answer to an open question raised by [Olikier-Waldspurger, SIAM J.
      Optim., 2025] on whether
  every cluster point of a sequence generated by a proximal gradient method is a proximal stationary point under local Lipschitz smoothness of one component of the cost function.
  As a second application, we propose a proximal variable smoothing algorithm with a nonmonotone linesearch for minimizing the sum of two nonsmooth nonconvex functions under lower regularity of one component and prox-regularity of the other.
  For the proposed algorithm, we show that every cluster point of a subsequence such that the stationarity measure vanishes is a Fr\'echet stationary point.
}
\keywords{
  stationarity measure, nonsmooth nonconvex optimization, proximal gradient method, variable smoothing, prox-regularity, lower regularity
}
\pacs[MSC Classification]{
  49J52, 49J53,
  65K10, 90C26, 90C30
}

\maketitle

\mathtoolsset{showonlyrefs=true}

\section{Introduction} \label{sec:introduction}
Stationarity measures play a central role in the design and analysis of nonconvex optimization algorithms, where a stationarity measure quantifies the achievement level of a certain necessary condition, called a stationarity condition, for the local optimality of a cost function.
In smooth nonconvex optimization where the cost function
$J:\mathcal{X}\to \exR$
is continuously differentiable over the Euclidean space
$\mathcal{X}$,
the quantity
$\norm{\nabla J(\widebar{\bm{x}})}$
with the gradient
$\nabla J:\mathcal{X}\to\mathcal{X}$
is the most standard stationarity measure at
$\widebar{\bm{x}} \in \mathcal{X}$
because
the condition
$\norm{\nabla J(\widebar{\bm{x}})} = 0$
is a well-known necessary condition for the local optimality of
$J$
at
$\widebar{\bm{x}}$~\cite{Nocedal-Wright06,Nesterov14}.
Indeed, many iterative algorithms are designed so that this quantity
$\norm{\nabla J(\bm{x}_{n})}$
decreases along a generated sequence
$(\bm{x}_{n})_{n=1}^{\infty}\subset\mathcal{X}$~\cite{Zhang-Hager04,Nesterov14,Nocedal-Wright06},
in particular are designed to achieve (numerically or theoretically)
a faster convergence of
$\norm{\nabla J(\bm{x}_{n})}$
to zero~\cite{Nesterov14,Ghadimi-Lan16,Marumo-Takeda24}.
In this sense, the choice of an appropriate stationarity measure is not merely a technical issue in convergence analysis; rather, it provides a basic guideline for algorithm design, stopping criteria, and worst-case complexity analysis.

For nonsmooth nonconvex optimization, a natural extension of the quantity
$\norm{\nabla J(\widebar{\bm{x}})}$
may seem to be the distance
$\dist(\bm{0},\partial_{*} J(\widebar{\bm{x}}))$
between
$\bm{0}$
and
$\partial_{*} J(\widebar{\bm{x}})\subset \mathcal{X}$,
where
$\partial_{*}J:\mathcal{X}\rightrightarrows\mathcal{X}$
denotes a suitable subdifferential of
$J$
as a generalization of
$\nabla J$.
Here we use the generic notation $\partial_{*}J$, because several notions of subdifferentials have been proposed in the nonconvex setting~\cite{Rockafellar-Wets98},
e.g., the {\em proximal},
  {\em Fr\'echet}
and
  {\em limiting subdifferentials}
denoted respectively by
$\Psubdiff J$,
$\Fsubdiff J$,
and
$\Lsubdiff J$
(see Definition~\ref{definition:subdifferential}).
However, the quantity
$\dist(\bm{0},\partial_{*} J(\widebar{\bm{x}}))$
is not well suited to algorithmic design and analysis, because this quantity is in general highly discontinuous~\cite{Davis-Drusvyatskiy19}, in particular is not stable\footnote{
  Consider a case where a local minimizer
  $\bm{x}^{\star}\in\dom{J}$
  is located on the boundary of
  $\dom{J}$.
  In this case, although we have
  $\dist(\bm{0},\partial_{*} J(\bm{x}^{\star})) = 0$
  for every
  $\partial_{*} \in \{\Psubdiff,\Fsubdiff,\Lsubdiff\}$
  by Fermat's rule, e.g.,~\cite[Cor. 2.7 (a) in Ch.
    2]{Clarke-Ledyaev-Stern-Wolenski98},~\cite[Thm. 10.1]{Rockafellar-Wets98},
  there exists
  $(\bm{x}_{n})_{n=1}^{\infty} \subset \mathcal{X}\setminus \dom{J}$
  converging to
  $\bm{x}^{\star}$
  such that
  $\dist(\bm{0},\partial_{*}
    J(\bm{x}_{n}))=+\infty\ (n\in\mathbb{N})$
  due to
  $\partial_{*} J(\widebar{\bm{x}}) = \emptyset\ (\widebar{\bm{x}} \notin \dom{J})$;
  hence
  $\lim_{n\to\infty}\dist(\bm{0},\partial_{*} J(\bm{x}_{n}))=+\infty \neq 0 = \dist(\bm{0},\partial_{*} J(\bm{x}^{\star}))$.
} around the boundary of
$\dom{J}\coloneqq \{\bm{x} \in \mathcal{X} \mid J(\bm{x})<+\infty\}$.
Hence, alternative stationarity measures have been developed by exploiting special structures of a given cost function
$J$~\cite{Beck17,Themelis-Stella-Patrinos18,Davis-Drusvyatskiy19,Chen-Garcia-Shahrampour22,Pougkakiotis-Kalogerias23,Liu-Xia24,Kume-Yamada25B,Davis-Drusvyatskiy-Shi25,Yagishita-Ito25B}.

For an
$\eta(>0)$-weakly convex function
$J$,
i.e.,
$J+\frac{\eta}{2}\norm{\cdot}^{2}$
is convex, the Moreau envelope-based stationarity measure
\begin{equation}
  (\widebar{\bm{x}} \in \mathcal{X})\quad
  \norm{\nabla \moreau{J}{\gamma}(\widebar{\bm{x}})}
  \ \mathrm{with}\
  \nabla \moreau{J}{\gamma}(\widebar{\bm{x}}) = \frac{\widebar{\bm{x}} - \prox{\gamma J}(\widebar{\bm{x}})}{\gamma}
  \label{eq:gradient_Moreau}
\end{equation}
with some
$\gamma \in (0,\eta^{-1})$
has been used in, e.g., subgradient methods~\cite{Davis-Drusvyatskiy19,Chen-Garcia-Shahrampour22,Pougkakiotis-Kalogerias23,Davis-Drusvyatskiy-Shi25},
because the {\em Moreau envelope}
$\moreau{J}{\gamma}:\mathcal{X}\to\mathbb{R}$
serves as a smooth approximation of
$J$,
i.e.,
$\moreau{J}{\gamma}$
converges pointwise to
$J$
as
$\gamma \searrow 0$,
and is continuously differentiable.
Here,
the {\em proximity operator}
$\prox{\gamma J}$
of
$J$
is defined in general as a set-valued mapping:
\begin{equation}
  (\gamma \in \mathbb{R}_{++})\quad
  \prox{\gamma J}:\mathcal{X}\rightrightarrows \mathcal{X}:\widebar{\bm{x}}  \mapsto \argmin_{\bm{x} \in \mathcal{X}} \left(J(\bm{x}) + \frac{1}{2\gamma}\|\bm{x}-\widebar{\bm{x}}\|^{2}\right), \label{eq:prox_definition}
\end{equation}
where
$\prox{\gamma J}(\widebar{\bm{x}})\in\mathcal{X}$
is single-valued for
$\gamma \in (0,\eta^{-1})$
under the $\eta$-weak convexity of
$J$.
In another important case where
the cost function
$J$
admits the structure
$J=\cost+\phi$
with a differentiable function
$\cost:\mathcal{X}\to\mathbb{R}$
such that
$\nabla \cost:\mathcal{X}\to\mathcal{X}$
is globally Lipschitz continuous, and
a nonsmooth nonconvex function
$\phi:\mathcal{X}\to\exR$,
the quantity
\begin{equation}
  (\widebar{\bm{x}} \in \mathcal{X}) \quad
  \dist\left(\bm{0},\mathcal{G}_{\gamma}^{\cost,\phi}(\widebar{\bm{x}})\right)\ \mathrm{with}\
  \mathcal{G}_{\gamma}^{\cost,\phi}:\mathcal{X}\rightrightarrows \mathcal{X}: \widebar{\bm{x}} \mapsto \frac{\widebar{\bm{x}} - \prox{\gamma\phi}\left(\widebar{\bm{x}}-\gamma \nabla\cost\left(\widebar{\bm{x}}\right)\right)}{\gamma} \label{eq:gradient_mapping}
\end{equation}
with some
$\gamma \in \mathbb{R}_{++}$
has been utilized as a stationarity measure in, e.g.,~\cite{Beck17,Themelis-Stella-Patrinos18,Yagishita-Ito25B} for proximal gradient methods.
Here,
$\mathcal{G}_{\gamma}^{\cost,\phi}$
is called
the {\em gradient mapping}~\cite{Beck17} under the convexity of
$\phi$,
or the {\em forward-backward residual}~\cite{Themelis-Stella-Patrinos18} in the absence of the convexity of
$\phi$.
These stationarity measures in~\eqref{eq:gradient_Moreau} and~\eqref{eq:gradient_mapping} serve as mathematical tools therein for establishing convergence analyses, including rate analyses, of iterative algorithms in nonsmooth nonconvex optimization.
However, beyond the above cases,
stationarity measures for nonsmooth nonconvex optimization have not been sufficiently investigated.
In particular, it is unclear how to construct a stationarity measure that is both mathematically meaningful and useful for designing and analyzing iterative algorithms in a case where both
$\cost$
and
$\phi$
in
$J=\cost+\phi$
are nonsmooth and nonconvex.

In this paper, we address this issue for the nonsmooth optimization problem
\begin{equation}
  \mathop{\mathrm{minimize}}\limits_{\bm{x}\in\mathcal{X}} J(\bm{x}) \coloneqq (\cost+\phi)(\bm{x})
  \label{eq:problem}
\end{equation}
under the following standing assumption:

\begin{assumption}[Standing assumption for~\eqref{eq:problem}]
  \label{assumption:basic}
  Throughout this paper, we assume:
  \begin{enumerate}[label=(\alph*)]
    \item
          \label{enum:problem:origin:minimizer}
          $J=\cost + \phi$
          is bounded below, i.e.,
          $\inf_{\bm{x}\in\mathcal{X}}(\cost+\phi)(\bm{x}) > -\infty$;
    \item
          \label{enum:problem:origin:cost}
          $\cost:\mathcal{X} \to \exR$
          is proper lower semicontinuous and
          locally Lipschitz continuous at every
          $\widebar{\bm{x}} \in \mathrm{int}(\dom{\cost})$,
          and
          $\dom{\cost}$
          is convex with
          $\dom{\phi} \subset \mathrm{int}(\dom{\cost})$
          (hence
          $\dom{J}=\dom{\phi}$;
          and
          $\mathrm{int}(\dom{\cost})$
          is convex by~\cite[Prop. 3.45 (ii)]{Bauschke-Combettes17});
    \item
          \label{enum:problem:origin:phi}
          $\phi:\mathcal{X}\to \exR$
          is proper lower semicontinuous and prox-bounded with threshold
          $\gamma_{\phi} \in \exRp$,
          i.e.,
          $\phi(\cdot)+\frac{1}{2\gamma} \norm{\cdot}^{2}$
          is bounded below for all
          $\gamma \in (0,\gamma_{\phi})$,
          (Note:
          for
          $\gamma \in (0,\gamma_{\phi})$
          and
          $\widebar{\bm{x}} \in \mathcal{X}$,
          $\prox{\gamma\phi}(\widebar{\bm{x}})\neq \emptyset$
          is guaranteed; see Fact~\ref{fact:prox_set_mapping}~\ref{enum:fact:prox_set_mapping:nonempty_resolvent}).
  \end{enumerate}
\end{assumption}
In this setting, we introduce a {\em gradient mapping-type stationarity measure}
\begin{align}
  (\gamma \in \mathbb{R}_{++}) \quad
   & \mathcal{M}_{\gamma}^{\cost,\phi}
  : \mathcal{X} \to [0,+\infty]
  :\widebar{\bm{x}}
  \mapsto
  \dist\left(\bm{0},\frac{\widebar{\bm{x}} - \prox{\gamma\phi}\left(\widebar{\bm{x}}-\gamma \Lsubdiff \cost\left(\widebar{\bm{x}}\right)\right)}{\gamma}\right)                                                                                              \nonumber \\
   & = \inf\left\{\norm{\frac{\widebar{\bm{x}} - \bm{p}}{\gamma}} \mid \bm{p}\in \prox{\gamma\phi}(\widebar{\bm{x}}-\gamma \bm{v}), \ \bm{v} \in \Lsubdiff \cost(\widebar{\bm{x}}) \right\}, \label{eq:measure}
\end{align}
where
$\Lsubdiff \cost:\mathcal{X}\rightrightarrows\mathcal{X}$
is
the {\em limiting  subdifferential} of
$\cost$
(see Definition~\ref{definition:subdifferential}~\ref{enum:definition:subdifferential:limiting}).
The measure
$\mathcal{M}_{\gamma}^{\cost,\phi}$
can be viewed as a generalization of the stationarity measure in~\eqref{eq:gradient_mapping}
by replacing
$\nabla \cost$
with
$\Lsubdiff \cost$
in the gradient mapping
$\mathcal{G}_{\gamma}^{\cost,\phi}$.
Moreover, the authors recently studied in~\cite{Kume-Yamada25B} the measure
$\mathcal{M}_{\gamma}^{\cost,\phi}$
in a special case where
$\phi$
is convex,
and
technical discussions for analysis of
$\mathcal{M}_{\gamma}^{\cost,\phi}$
rely heavily on the convexity of
$\phi$.
Since
$\phi$
is not assumed to be convex in this paper, we cannot rely on such standard facts
under the convexity of
$\phi$.
Instead, we leverage tools from nonsmooth and set-valued analyses in order to
derive useful properties of
$\mathcal{M}_{\gamma}^{\cost,\phi}$
in Section~\ref{sec:stationarity},
especially under Assumption~\ref{assumption:regular} or~\ref{assumption:smooth} below, considered independently.

\begin{assumption}[Regularity assumption; see Definition~\ref{definition:regular}]
  \label{assumption:regular}
  Together with Assumption~\ref{assumption:basic}, we assume the following:
  \begin{enumerate}[label=(\alph*)]
    \item
          \label{enum:assumption:regular:cost}
          $\cost$
          is {\em lower regular} at every
          $\widebar{\bm{x}} \in \mathrm{int}(\dom{\cost})$;
    \item
          \label{enum:assumption:regular:phi}
          $\phi$
          is {\em prox-regular} at every
          $\widebar{\bm{x}} \in \dom{\phi}$.
  \end{enumerate}
\end{assumption}

\begin{assumption}[Local Lipschitz smoothness of $\cost$]
  \label{assumption:smooth}
  Together with Assumption~\ref{assumption:basic}, we assume that
  $\cost$
  is differentiable over
  $\mathrm{int}(\dom{\cost})$
  such that its gradient
  $\nabla \cost:\mathrm{int}(\dom{\cost})\to\mathcal{X}$
  is locally Lipschitz continuous at every
  $\widebar{\bm{x}} \in \mathrm{int}(\dom{\cost})$.
\end{assumption}

\begin{remark}[On Assumptions~\ref{assumption:regular} and~\ref{assumption:smooth}]
  Assumption~\ref{assumption:smooth} has recently been imposed in~\cite{Kanzow-Mehlitz22,DeMarchi23,Jia-Kanwoz-Mehlitz23,Kanzow-Lehmann25,Yagishita-Ito25A} for convergence analysis of proximal gradient methods as a weaker condition than the standard one, namely the global Lipschitz continuity of
  $\nabla \cost$~\cite{Attouch-Bolte-Svaiter13,Li-Lin15,Beck17,Themelis-Stella-Patrinos18,Yagishita-Ito25B}.
  In contrast, although Assumption~\ref{assumption:regular} seems to have been less emphasized in the literature, it covers many important examples in signal processing and machine learning applications (see Examples~\ref{example:cost} and~\ref{example:phi}; see also our recent conference paper~\cite{Kume-Yamada25C} for robust low-rank matrix recovery).
\end{remark}

The main analytical contributions of this paper are characterizations of stationarity conditions of
$\cost+\phi$
in~\eqref{eq:problem}
(see Definition~\ref{definition:optimality})
via
$\mathcal{M}_{\gamma}^{\cost,\phi}$
in Section~\ref{sec:stationarity} as follows:
\begin{enumerate}[label=\arabic*.]
  \item
        For a sufficiently small
        $\gamma_{\bm{x}^{\star}} \in (0,\gamma_{\phi})$,
        a {\em Fr\'echet stationary point} (resp. {\em proximal stationary point})
        $\bm{x}^{\star} \in \dom{\phi}$
        of
        $\cost +\phi$
        can be characterized in Proposition~\ref{proposition:criticality_stationarity} by the condition
        $\mathcal{M}_{\gamma_{\bm{x}^{\star}}}^{\cost,\phi}(\bm{x}^{\star}) = 0$
        under Assumption~\ref{assumption:regular} (resp.
        Assumption~\ref{assumption:smooth}).
  \item
        Such stationary points
        $\bm{x}^{\star}\in \dom{\phi}$
        of
        $\cost+\phi$
        can also be characterized through the asymptotic condition
        $\liminf_{n\to\infty}\mathcal{M}_{\gamma_{n}}^{\cost,\phi}(\bm{x}_{n})=0$
        for some
        $(\gamma_{n})_{n=1}^{\infty}\subset(0,\gamma_{\phi})$
        and
        $(\bm{x}_{n})_{n=1}^{\infty}\subset\mathcal{X}$
        converging to
        $\bm{x}^{\star}$
        (see Theorem~\ref{theorem:P_stationarity} for the proximal stationarity under Assumption~\ref{assumption:smooth}; Theorem~\ref{theorem:F_stationarity} for the Fr\'echet stationarity under Assumption~\ref{assumption:regular} together with a technical Assumption~\ref{assumption:phi} on
        $\phi$).
        This result is motivated by the fact that iterative algorithms typically produce a sequence
        $(\bm{x}_{n})_{n=1}^{\infty}$
        with stepsizes
        $(\gamma_{n})_{n=1}^{\infty}$
        along which a stationarity measure is expected to vanish asymptotically.
  \item
        Under Assumptions~\ref{assumption:regular} and~\ref{assumption:phi},
        a smoothing-based analogue of the asymptotic characterization
        in Theorem~\ref{theorem:F_stationarity}
        is also presented in Theorem~\ref{theorem:characterization_stationarity_asymptotic}.
        More precisely, every Fr\'echet stationary point
        $\bm{x}^{\star} \in \dom{\phi}$
        of
        $\cost + \phi$
        can be characterized
        through the asymptotic condition
        $\liminf_{n\to\infty}\mathcal{M}_{\gamma_{n}}^{\cost^{\langle \mu_{n}\rangle},\phi}(\bm{x}_{n})=0$
        with some
        $(\gamma_{n})_{n=1}^{\infty}\subset(0,\gamma_{\phi})$,
        $(\mu_{n})_{n=1}^{\infty} \searrow 0$,
        and
        $(\bm{x}_{n})_{n=1}^{\infty}\subset\mathcal{X}$
        converging to
        $\bm{x}^{\star}$,
        where
        $\{\cost^{\langle\mu\rangle}\}_{\mu\in(0,\widetilde{\mu})}\ (\widetilde{\mu} \in \exRp)$
        is a {\em smoothing function}
        of
        $\cost$
        (see Definition~\ref{definition:smoothing}).
        This result enables us to combine the proposed analysis with smoothing techniques~\cite{Zhang-Chen09,Chen12,Burke-Hoheisel-Kanzow13,Bot-Hendrich15,Burke-Hoheisel17,Bohm-Wright21,Liu-Xia24,Bian-Chen20,Yu-Zhang22,Nishioka-Kanno24,Kume-Yamada24A,Kume-Yamada25B}
        developed for nonsmooth optimization.
\end{enumerate}

In Sections~\ref{sec:application} and~\ref{sec:algorithm}, we present two applications of the above asymptotic properties of
$\mathcal{M}_{\gamma}^{\cost,\phi}$
to proximal gradient-type methods.

In Section~\ref{sec:application:existing:PGM},
the first application with Theorem~\ref{theorem:P_stationarity} revisits convergence analyses of proximal gradient methods (PGM) under Assumption~\ref{assumption:smooth}.
Conventional convergence analyses of PGM guarantee, at best, that every cluster point of a generated sequence is a limiting stationary point~\cite{Kanzow-Mehlitz22,DeMarchi23,Jia-Kanwoz-Mehlitz23,Kanzow-Lehmann25}, or Fr\'echet stationary point~\cite{Yagishita-Ito25A}, of
$\cost+\phi$
in~\eqref{eq:problem},
where
the limiting and Fr\'echet stationarities are weaker than the proximal stationarity (see~\eqref{eq:hierarchical_optimality}).
Recently,~\cite{Olikier-Waldspurger25} establishes a subsequence convergence guarantee to a proximal stationary point of a projected gradient descent method, corresponding to the case
$\phi=\iota_{C}$
with a closed set
$C \subset \mathcal{X}$
in the problem~\eqref{eq:problem},
where the indicator function
$\iota_{C}$
is defined by
\begin{equation}
  \iota_{C}:\mathcal{X}\to \exR:\bm{x}\mapsto
  \begin{cases}
    0,       & \mathrm{if}\ \bm{x}\in C; \\
    +\infty, & \mathrm{if}\ \bm{x}\notin C.
  \end{cases}
  \label{eq:indicator}
\end{equation}
However,
as raised in~\cite[Sect. 9]{Olikier-Waldspurger25},
it remains open whether the result in~\cite{Olikier-Waldspurger25} can be extended from
$\iota_{C}$
to a general proper lower semicontinuous and prox-bounded function
$\phi$
or not.
By using our asymptotic analysis of
$\mathcal{M}_{\gamma}^{\cost,\phi}$
in Theorem~\ref{theorem:P_stationarity},
Corollary~\ref{corollary:refine_PGM} shows that every cluster point of a generated sequence of PGM
is actually a proximal stationary point of
$\cost + \phi$
under Assumption~\ref{assumption:smooth}
with a standard setting of PGM where three conditions in~\ref{enum:PGM_step2} in Section~\ref{sec:application} are met.

In Section~\ref{sec:algorithm}, the second application with Theorem~\ref{theorem:characterization_stationarity_asymptotic} presents a proximal variable smoothing algorithm for minimizing
$\cost+\phi$
in~\eqref{eq:problem} under Assumption~\ref{assumption:regular} and the availability of a {\em smoothing function} of
$\cost$
together with a technical assumption on
$\phi$
(Assumption~\ref{assumption:phi}).
Proximal variable smoothing algorithms have been proposed in~\cite{Liu-Xia24,Lopes-Peres-Bilches25,Kume-Yamada25B} originally for the problem~\eqref{eq:problem} with a structured composite nonsmooth function
$\cost=h+g\circ\mathfrak{S}$
(see Example~\ref{example:cost} for the details)
and a proper lower semicontinuous convex function
$\phi$.
In~\cite{Liu-Xia24,Lopes-Peres-Bilches25,Kume-Yamada25B}, a smoothing function of
$\cost$
is given by replacing a nonsmooth weakly convex function
$g$
with its {\em Moreau envelope}
(see also Example~\ref{example:cost}).
The proposed algorithm differs from~\cite{Liu-Xia24,Lopes-Peres-Bilches25,Kume-Yamada25B} in three aspects:
the proposed algorithm admits
(i) more general
$\cost$
satisfying Assumption~\ref{assumption:regular}~\ref{enum:assumption:regular:cost}
and its smoothing function satisfying Definition~\ref{definition:smoothing} than those considered in~\cite{Liu-Xia24,Lopes-Peres-Bilches25,Kume-Yamada25B};
(ii) a prox-regular function
$\phi$
beyond the convexity;
(iii) a {\em nonmonotone linesearch} for stepsize selection, whereas the existing proximal variable smoothing algorithms employ a diminishing stepsize rule~\cite{Liu-Xia24,Lopes-Peres-Bilches25} or a monotone linesearch~\cite{Kume-Yamada25B}.
We note that nonmonotone linesearch algorithms~\cite{Zhang-Hager04,Ahookhosh-Ghaderi17,Themelis-Stella-Patrinos18,Kanzow-Mehlitz22,DeMarchi23,Jia-Kanwoz-Mehlitz23,Kanzow-Lehmann25,Yagishita-Ito25A}
are often employed as numerically efficient choices of stepsizes in the literature (see, e.g.,~\cite{Zhang-Hager04,Ahookhosh-Ghaderi17}).
In Theorem~\ref{theorem:convergence_extension}, for
sequences
$(\bm{x}_{n})_{n=1}^{\infty} \subset \mathcal{X}$
and stepsizes
$(\gamma_{n})_{n=1}^{\infty} \subset (0,\gamma_{\phi})$
generated by the proposed algorithm,
we show
$\liminf_{n\to\infty}\mathcal{M}_{\gamma_{n}}^{\cost^{\langle \mu_{n} \rangle},\phi}(\bm{x}_{n}) = 0$
and that every cluster point of a subsequence
$(\bm{x}_{m(l)})_{l=1}^{\infty}$
such that
$\lim_{l\to\infty} \mathcal{M}_{\gamma_{m(l)}}^{\cost^{\langle \mu_{m(l)} \rangle},\phi}(\bm{x}_{m(l)}) = 0$
is a Fr\'echet stationary point of
$\cost + \phi$.

The remainder of this paper is organized as follows.
The rest of Section~\ref{sec:introduction} is devoted to notation.
Section~\ref{sec:preliminary} reviews basic tools in nonsmooth analysis,
while~\ref{sec:appendix:set_valued} summarizes facts on the set-valued analysis used in the paper.
Section~\ref{sec:stationarity} establishes basic and asymptotic properties of the proposed stationarity measure.
Section~\ref{sec:application:existing} revisits convergence analyses of proximal gradient methods through the asymptotic behavior of the proposed measure.
Section~\ref{sec:algorithm} presents a proximal variable smoothing algorithm with nonmonotone linesearch and its convergence analysis.
Section~\ref{sec:conclusion} concludes the paper.

  {\bf Notation.}
$\mathbb{N}$,
$\mathbb{R}$,
$\mathbb{R}_{+}$,
and
$\mathbb{R}_{++}$
denote respectively the sets of all positive integers, all real numbers, all nonnegative real numbers, and all positive real numbers.
$\lceil \cdot \rceil$
stands for the ceiling function.
$\Id$
stands for the identity operator.
$\|\cdot\|$
and
$\inprod{\cdot}{\cdot}$
are respectively the Euclidean norm and the standard inner product.
For given real sequence
$(a_{n})_{n=1}^{\infty} \subset \mathbb{R}$
and
$\widebar{a} \in \mathbb{R}$,
$a_{n} \searrow \widebar{a}$
means in this paper that
$(a_{n})_{n=1}^{\infty}$
is monotonically nonincreasing and
$\lim_{n\to\infty} a_{n} = \widebar{a}$.
For a given point
$\widebar{\bm{x}} \in \mathcal{X}$
and a set
$\mathcal{E} \subset \mathcal{X}$,
$B(\widebar{\bm{x}},\delta) \coloneqq \{\bm{x}\in \mathcal{X} \mid \norm{\bm{x}-\widebar{\bm{x}}} < \delta \}$
stands for the open ball centered at
$\widebar{\bm{x}}$
with radius
$\delta > 0$,
and
$\dist:\mathcal{X}\times 2^{\mathcal{X}} \to [0,+\infty]$
stands for the distance function
$\dist(\widebar{\bm{x}},\mathcal{E})\coloneqq \inf\{\norm{\bm{v}-\widebar{\bm{x}}}\mid \bm{v} \in \mathcal{E}\}$,
where
$\dist(\widebar{\bm{x}}, \mathcal{E})$
is defined by
$+\infty$
for
$\mathcal{E} = \emptyset$.

A mapping
$\mathcal{F}:\mathcal{X}\to\mathcal{Y}$
to a Euclidean space
$\mathcal{Y}$
is said to be $L_{\mathcal{F}}$-{\em Lipschitz continuous} over
$\mathcal{E} \subset \mathcal{X}$
(with a Lipschitz constant
$L_{\mathcal{F}} > 0$)
if
$\norm{\mathcal{F}(\bm{x}_{1}) - \mathcal{F}(\bm{x}_{2})} \leq L_{\mathcal{F}}\norm{\bm{x}_{1}-\bm{x}_{2}}\ (\forall \bm{x}_{1},\bm{x}_{2} \in \mathcal{E})$.
For a function
$J:\mathcal{X} \to \exR$,
$J$
is called (a) proper if
$\dom{J}\coloneqq \{\bm{x} \in \mathcal{X} \mid J(\bm{x}) < +\infty\} \neq \emptyset$;
(b)~lower semicontinuous if
$\{\bm{x} \in \mathcal{X} \mid J(\bm{x})\leq a\} \subset \mathcal{X}$
is closed for every
$a \in \mathbb{R}$,
or equivalently,
$\liminf_{n\to\infty}J(\bm{x}_{n}) \geq J(\widebar{\bm{x}})$
holds for every
$\widebar{\bm{x}}\in \mathcal{X}$
and for all
$(\bm{x}_{n})_{n=1}^{\infty} \subset \mathcal{X}$
converging to
$\widebar{\bm{x}}$;
(c) convex if
$J(t\bm{x}_{1}+(1-t)\bm{x}_{2}) \leq t J(\bm{x}_{1}) + (1-t)J(\bm{x}_{2})\ (\forall \bm{x}_{1},\bm{x}_{2} \in \dom{J}, \forall t \in [0,1])$;
(d)
locally Lipschitz continuous at
$\widebar{\bm{x}}\in \dom{J}$
if
there exists an open neighborhood
$\mathcal{N}_{\widebar{\bm{x}}} \subset \mathcal{X}$
of
$\widebar{\bm{x}}$
such that
$J$
is Lipschitz continuous over
$\mathcal{N}_{\widebar{\bm{x}}}$.

For convenience, let
$\Ninfty\coloneqq \{\mathcal{N} \subset \mathbb{N} \mid \mathcal{N}\ \mathrm{infinite}\}$.
For a set sequence
$(\mathcal{E}_{n})_{n=1}^{\infty}$
with subsets
$\mathcal{E}_{n} \subset \mathcal{X}$,
{\em the outer limit}~\cite[Def. 4.1]{Rockafellar-Wets98}
of
$(\mathcal{E}_{n})_{n=1}^{\infty}$
is defined by
\begin{equation*}
  \Limsup_{n\to\infty} \mathcal{E}_{n}
  \coloneqq  \left\{\bm{v} \in \mathcal{X} \mid
  \exists\mathcal{N}\in \Ninfty,\  \exists \bm{v}_{n}\in \mathcal{E}_{n}\ (n\in \mathcal{N})\
  {\rm s.t.}\ \bm{v}= \lim_{\mathcal{N}\ni n\to\infty} \bm{v}_{n}
  \right\}.
\end{equation*}
We simply write
$\Limsup_{n\to\infty}\bm{v}_{n} \coloneqq \Limsup_{n\to\infty}\{\bm{v}_{n}\}$
even for the outer limit of a vector sequence
$(\bm{v}_{n})_{n=1}^{\infty} \subset \mathcal{X}$,
i.e.,
the set of all cluster points of
$(\bm{v}_{n})_{n=1}^{\infty}$.
For a set-valued mapping
$\mathcal{S}:\mathcal{X}\rightrightarrows\mathcal{X}$,
we write
$\mathcal{S}(\mathcal{E})\coloneqq \bigcup_{\bm{x} \in \mathcal{E}} \mathcal{S}(\bm{x})$
for
$\mathcal{E} \subset \mathcal{X}$,
and
$\mathcal{S}$
is said to be {\em outer semicontinuous} at
$\widebar{\bm{x}} \in \mathcal{X}$~\cite[Def. 5.4]{Rockafellar-Wets98}
if
$\Limsup_{n\to\infty} \mathcal{S}(\bm{x}_{n}) \subset \mathcal{S}(\widebar{\bm{x}})$
holds
for every sequence
$(\bm{x}_{n})_{n=1}^{\infty} \subset \mathcal{X}$
converging to
$\widebar{\bm{x}}$
(or equivalently, if
$\bigcup_{\mathcal{X}\ni \bm{x}_{n} \to\widebar{\bm{x}}}\Limsup_{n\to\infty}\mathcal{S}(\bm{x}_{n}) \subset \mathcal{S}(\widebar{\bm{x}})$
holds).
For
$t_{1},t_{2} \in \mathbb{R}$,
$\widebar{\bm{x}} \in \mathcal{X}$,
and
$\mathcal{E}_{1}, \mathcal{E}_{2} \subset \mathcal{X}$,
we use Minkowski sums of sets (see~\cite[p.25]{Rockafellar-Wets98}):
$\widebar{\bm{x}} + t_{1} \mathcal{E}_{1} \coloneqq \{\widebar{\bm{x}} + t_{1}\bm{v}_{1} \mid \bm{v}_{1}\in \mathcal{E}_{1}\} \subset \mathcal{X}$;
$t_{1}\mathcal{E}_{1}+t_{2}\mathcal{E}_{2} \coloneqq \{t_{1}\bm{v}_{1}+t_{2}\bm{v}_{2} \mid \bm{v}_{1}\in \mathcal{E}_{1}, \bm{v}_{2} \in \mathcal{E}_{2}\} \subset \mathcal{X}$,
where
these are understood as
$\emptyset$
if
$\mathcal{E}_{1}=\emptyset$
(or
$\mathcal{E}_{2}=\emptyset$
for the latter).

\section{Preliminaries on nonsmooth analysis} \label{sec:preliminary}
\label{sec:preliminary:subdifferential}
The standard subdifferential for convex functions has been extended to those for nonconvex functions as follows.

\begin{definition}[Subdifferentials~{\cite[Def. 8.3]{Rockafellar-Wets98}}]
  \label{definition:subdifferential}
  Let
  $\psi:\mathcal{X} \to \exR$
  be a proper lower semicontinuous function.
  For
  $\widebar{\bm{x}} \in \dom{\psi}$,
  a vector
  $\bm{v}  \in \mathcal{X}$
  is said to be
  \begin{enumerate}[label=(\alph*)]
    \item
          a {\em proximal subgradient} of
          $\psi$
          at
          $\widebar{\bm{x}}$,
          denoted by
          $\bm{v} \in \Psubdiff \psi(\widebar{\bm{x}})$~\cite[Def. 8.45]{Rockafellar-Wets98},
          if there exist
          $\gamma \in \mathbb{R}_{++}$
          and
          $\delta \in \mathbb{R}_{++}$
          such that
          $\psi(\bm{x}) \geq \psi(\widebar{\bm{x}}) + \inprod{\bm{v}}{\bm{x}-\widebar{\bm{x}}} - \frac{1}{2\gamma}\norm{\bm{x}-\widebar{\bm{x}}}^{2}\ (\forall \bm{x} \in B(\widebar{\bm{x}}, \delta))$;
    \item
          \label{enum:definition:subdifferential:Frechet}
          a {\em Fr\'echet (or regular) subgradient} of
          $\psi$
          at
          $\widebar{\bm{x}}$,
          denoted by
          $\bm{v} \in \Fsubdiff \psi(\widebar{\bm{x}})\subset \mathcal{X}$,
          if
            {\thickmuskip=2mu plus 1mu minus 1mu
              \medmuskip=1mu plus 1mu minus 1mu
              $\liminf\limits_{\mathcal{X}\setminus\{\widebar{\bm{x}}\}\ni \bm{x}\to \widebar{\bm{x}}} \frac{\psi(\bm{x})-\psi(\widebar{\bm{x}}) - \inprod{\bm{v}}{\bm{x}-\widebar{\bm{x}}}}{\norm{\bm{x}-\widebar{\bm{x}}}} \coloneqq \sup\limits_{\epsilon>0} \left(\inf\limits_{0<\norm{\bm{x}-\widebar{\bm{x}}}< \epsilon}\frac{\psi(\bm{x})-\psi(\widebar{\bm{x}}) - \inprod{\bm{v}}{\bm{x}-\widebar{\bm{x}}}}{\norm{\bm{x}-\widebar{\bm{x}}}}\right) \geq 0$}
          holds;
    \item
          \label{enum:definition:subdifferential:limiting}
          a {\em limiting (or general) subgradient} of
          $\psi$
          at
          $\widebar{\bm{x}}$,
          denoted by
          $\bm{v} \in \Lsubdiff \psi(\widebar{\bm{x}}) \subset \mathcal{X}$,
          if there exist
          (i)
          $(\bm{x}_{n})_{n=1}^{\infty} \subset \mathcal{X}$
          with
          $\lim_{n\to\infty} \bm{x}_{n} = \widebar{\bm{x}}$
          and
          $\lim_{n\to\infty} \psi(\bm{x}_{n}) = \psi(\widebar{\bm{x}})$,
          and
          (ii)
          $(\bm{v}_{n})_{n=1}^{\infty} \subset \mathcal{X}$
          with
          $\bm{v}_{n} \in \Fsubdiff \psi(\bm{x}_{n})\ (n\in\mathbb{N})$
          such that
          $\bm{v} = \lim_{n\to\infty} \bm{v}_{n}$,
  \end{enumerate}
  where
  $\Psubdiff \psi(\widebar{\bm{x}})$,
  $\Fsubdiff \psi(\widebar{\bm{x}})$
  and
  $\Lsubdiff \psi(\widebar{\bm{x}})$
  are called respectively the
    {\em proximal},
    {\em Fr\'echet} and {\em limiting subdifferentials} of
  $\psi$
  at
  $\widebar{\bm{x}}\in \dom{\psi}$
  [for
    $\widebar{\bm{x}} \notin \dom{\psi}$,
    $\Psubdiff \psi(\widebar{\bm{x}})$,
    $\Fsubdiff \psi(\widebar{\bm{x}})$
    and
    $\Lsubdiff \psi(\widebar{\bm{x}})$
    are defined by
    $\emptyset$].
  By, e.g.,~\cite[Prop. 3.6]{Wang-Wang24},
  we have
  \begin{equation}
    (\bm{x}\in\mathcal{X}) \quad
    \Psubdiff \psi(\bm{x}) \subset \Fsubdiff \psi(\bm{x}) \subset \Lsubdiff \psi(\bm{x}). \label{eq:subdifferential_relation}
  \end{equation}
\end{definition}

The notion of regularities below is a key ingredient for deriving sum rules of subdifferentials (see~\eqref{eq:sum_subdifferential}).

\begin{definition}[Regularity]
  \label{definition:regular}
  Let
  $\psi:\mathcal{X}\to\exR$
  be a proper lower semicontinuous function.
  \begin{enumerate}[label=(\alph*)]
    \item
          \label{enum:definition:regular:lower_regular}
          $\psi$
          is said to be {\em lower regular} at
          $\widebar{\bm{x}} \in \dom{\psi}$~\cite[Eq. (1.49) in p.31]{Mordukhovich24}
          if
          $\Fsubdiff \psi(\widebar{\bm{x}}) = \Lsubdiff \psi(\widebar{\bm{x}})$
          holds,
          where the inclusion
          $\subset$
          always holds by~\eqref{eq:subdifferential_relation}.
    \item
          \label{enum:definition:regular:prox-regular}
          $\psi$
          is said to be {\em prox-regular} at
          $\widebar{\bm{x}}\in \dom{\psi}$
          for
          $\widebar{\bm{v}} \in \Lsubdiff \psi(\widebar{\bm{x}})$~\cite[Def. 1.1]{Poliquin-Rockafellar96}~\cite[Def. 13.27]{Rockafellar-Wets98}
          if
          there exist
          $\delta > 0$
          and
          $\eta \geq 0$
          such that
          \begin{equation}
            (\forall \bm{x}'\in B(\widebar{\bm{x}},\delta)) \quad
            \psi(\bm{x}')
            \geq \psi(\bm{x}) + \inprod{\bm{v}}{\bm{x}' - \bm{x}} - \frac{\eta}{2}\norm{\bm{x}'-\bm{x}}^{2} \label{eq:prox_regular}
          \end{equation}
          whenever
          $\bm{x} \in B(\widebar{\bm{x}}, \delta)$,
          $\bm{v} \in \Lsubdiff \psi(\bm{x})\cap B(\widebar{\bm{v}}, \delta)$,
          and
          $\psi(\bm{x}) < \psi(\widebar{\bm{x}}) + \delta$.
          $\psi$
          is said to be prox-regular at
          $\widebar{\bm{x}} \in \dom{\psi}$
          if
          $\psi$
          is prox-regular at
          $\widebar{\bm{x}}$
          for all
          $\widebar{\bm{v}} \in \Lsubdiff \psi(\widebar{\bm{x}})$.
  \end{enumerate}
\end{definition}
We refer to Examples~\ref{example:cost} and~\ref{example:phi} for examples of regular functions of our interest.
In contrast, the function
$\mathbb{R}\to\mathbb{R}:x\mapsto -\abs{x}$
is not prox-regular\footnote{
  With
  $\psi\coloneqq -\abs{\cdot}$,
  we have
  $\Lsubdiff \psi(0) = \{-1,1\}$.
  Assume contrarily that
  $\psi$
  is prox-regular at
  $0$
  for
  $1$,
  i.e., there exist
  $\delta > 0$
  and
  $\eta \geq 0$
  such that
  $\psi(x') \geq \psi(x)+v(x'-x)-\frac{\eta}{2}(x'-x)^{2}$
  for all
  $x, x' \in B(0,\delta)$,
  $v \in \Lsubdiff\psi(x)$
  with
  $v \in B(1,\delta)$,
  and
  $\psi(x) < \delta$.
  By using convention
  $0^{-1} = +\infty$
  in this footnote,
  for
  $t \in (0,\min\{\delta,\eta^{-1}\})$,
  let
  $x = -t$,
  $x' = t$
  and
  $v = 1 \in \Lsubdiff \psi(-t)$.
  Then, the inequality
  $\psi(x') \geq \psi(x)+v(x'-x)-\frac{\eta}{2}(x'-x)^{2}$
  implies
  $-t \geq -t + 2t - 2\eta t^{2}$
  $(\Leftrightarrow 0 \geq t(1-\eta t))$,
  which contradicts
  $0 < t(1-\eta t)\ (\forall t\in (0,\eta^{-1}))$.
}at
$0$.

Under Assumption~\ref{assumption:regular}~\ref{enum:assumption:regular:phi}, i.e., the prox-regularity of
$\phi$
at every
$\widebar{\bm{x}}\in\dom{\phi}$,
the following equality holds (see, e.g.,~\cite[Paragraph just after Def.~1.1]{Poliquin-Rockafellar96}):
\begin{equation}
  (\bm{x} \in \mathcal{X}) \quad
  \Psubdiff \phi(\bm{x}) = \Fsubdiff \phi(\bm{x}) = \Lsubdiff \phi(\bm{x}), \label{eq:prox_regular_lower_regular}
\end{equation}
which implies the lower regularity of
$\phi$.
By further assuming Assumption~\ref{assumption:regular}~\ref{enum:assumption:regular:cost}, i.e., the lower regularity of
$\cost$
at every
$\widebar{\bm{x}} \in \mathrm{int}(\dom{\cost})$,
we have
\begin{align}
  (\bm{x} \in \mathrm{int}(\dom{\cost}))\quad
  \Fsubdiff \cost(\bm{x}) + \Fsubdiff \phi(\bm{x})
   & = \Fsubdiff(\cost+\phi)(\bm{x})
  = \Lsubdiff(\cost+\phi)(\bm{x}) \\
   & = \Lsubdiff \cost(\bm{x}) + \Lsubdiff \phi (\bm{x})
  = \Lsubdiff \cost(\bm{x}) + \Psubdiff \phi (\bm{x})  \label{eq:sum_subdifferential}
\end{align}
by
$\Fsubdiff \cost(\bm{x}) = \Lsubdiff \cost(\bm{x})$,
$\Psubdiff \phi(\bm{x}) = \Fsubdiff \phi(\bm{x}) = \Lsubdiff \phi(\bm{x})$
in~\eqref{eq:prox_regular_lower_regular}, and
$\Fsubdiff \cost(\bm{x}) + \Fsubdiff \phi(\bm{x}) \subset \Fsubdiff (\cost+\phi)(\bm{x}) \overset{\eqref{eq:subdifferential_relation}}{\subset} \Lsubdiff (\cost+\phi)(\bm{x}) \subset \Lsubdiff \cost(\bm{x}) + \Lsubdiff \phi(\bm{x})$,
where the first and third inclusion hold from well-known sum rules~\cite[Cor. 10.9]{Rockafellar-Wets98} and~\cite[Exe. 10.10]{Rockafellar-Wets98} with the local Lipschitz continuity (or equivalently, strict continuity) of
$\cost$
at
$\bm{x}$.

Under Assumption~\ref{assumption:smooth} alone, i.e., the local Lipschitz smoothness of
$\cost$
at every
$\widebar{\bm{x}} \in \mathrm{int}(\dom{\cost})$,
we have
(see~\cite[Prop. 2]{Wang-Ye-Yuan-Zeng-Zhang22} and~\cite[Exe.
  8.8 (b)]{Rockafellar-Wets98} respectively):
\begin{equation}
  (\bm{x} \in \mathrm{int}(\dom{\cost})) \quad
  \Psubdiff(\cost+\phi)(\bm{x}) = \nabla\cost(\bm{x}) + \Psubdiff\phi(\bm{x}) = \Lsubdiff\cost(\bm{x}) + \Psubdiff\phi(\bm{x}). \label{eq:sum_subdifferential_smooth}
\end{equation}

We define the notion of optimality and stationarity by using the subdifferential.

\begin{definition}[Optimality]
  \label{definition:optimality}
  For
  $\cost+\phi$
  in Assumption~\ref{assumption:basic},
  $\bm{x}^{\star} \in  \dom{\phi}$
  is said to be
  \begin{enumerate}[label=(\alph*)]
    \item
          a {\em local minimizer} of
          $\cost+\phi$
          if there exists an open neighborhood
          $\mathcal{N}_{\bm{x}^{\star}}\subset \mathcal{X}$
          of
          $\bm{x}^{\star}$
          such that
          $(\cost + \phi)(\bm{x}^{\star}) \leq (\cost + \phi)(\bm{x})\ (\bm{x} \in \mathcal{N}_{\bm{x}^{\star}})$;
    \item
          \label{enum:definition:optimality:stationary}
          a {\em proximal stationary point } ({\rm P}-stationary point) of
          $\cost + \phi$
          if
          $\Psubdiff (\cost + \phi)(\bm{x}^{\star}) \ni \bm{0}$;
    \item
          a {\em Fr\'echet  stationary point} ({\rm F}-stationary point) of
          $\cost+\phi$
          if
          $\Fsubdiff (\cost + \phi)(\bm{x}^{\star}) \ni \bm{0}$;
    \item
          \label{enum:optimality:L_stationary}
          a {\em limiting stationary point} ({\rm L}-stationary point) of
          $\cost+\phi$
          if
          $\Lsubdiff (\cost + \phi)(\bm{x}^{\star}) \ni \bm{0}$.
  \end{enumerate}
\end{definition}

For nonconvex optimization, since finding a local minimizer is NP-hard in general (see, e.g.,~\cite{Jain-Kar17}), finding a point satisfying a certain necessary condition for the local optimality has been addressed as typical goals in nonconvex optimization literature, e.g.,~\cite{Bohm-Wright21,Liu-Xia24,Lopes-Peres-Bilches25,Kume-Yamada24A,Kume-Yamada25B,Attouch-Bolte-Svaiter13,Li-Lin15,Beck17,Themelis-Stella-Patrinos18,Kanzow-Mehlitz22,DeMarchi23,Jia-Kanwoz-Mehlitz23,Kanzow-Lehmann25,Yagishita-Ito25A,Yagishita-Ito25B,Davis-Drusvyatskiy19,Chen-Garcia-Shahrampour22,Pougkakiotis-Kalogerias23,Davis-Drusvyatskiy-Shi25}.
Fermat's rule (see, e.g.,~\cite[Cor. 2.7 (a) in Ch.
  2]{Clarke-Ledyaev-Stern-Wolenski98},~\cite[Thm. 10.1]{Rockafellar-Wets98}) with~\eqref{eq:subdifferential_relation}
implies that these stationarities are the most versatile necessary conditions for the local optimality as:
\begin{equation}
  \mathrm{local\ optimality} \Rightarrow {\rm P}\mathchar`-\mathrm{stationarity} \Rightarrow {\rm F}\mathchar`-\mathrm{stationarity}\Rightarrow {\rm L}\mathchar`-\mathrm{stationarity}. \label{eq:hierarchical_optimality}
\end{equation}

\section{Gradient mapping-type stationarity measure} \label{sec:stationarity}
We study relations between stationarities of
$\cost + \phi$
and the stationarity measure in
\eqref{eq:measure}:
\begin{equation}
  (\gamma \in \mathbb{R}_{++}) \quad
  \mathcal{M}_{\gamma}^{\cost,\phi}
  : \mathcal{X} \to [0,+\infty]
  :\widebar{\bm{x}}
  \mapsto
  \dist\left(\bm{0},\frac{\widebar{\bm{x}} - \prox{\gamma\phi}\left(\widebar{\bm{x}}-\gamma \Lsubdiff \cost\left(\widebar{\bm{x}}\right)\right)}{\gamma}\right). \label{eq:measure_limiting}
\end{equation}
To keep generality in Section~\ref{sec:stationarity}, we do not assume
Assumptions~\ref{assumption:regular} or~\ref{assumption:smooth} unless stated otherwise.
Since
$\phi$
is not assumed to be convex here (unlike~\cite{Kume-Yamada25B}), we cannot rely on standard facts in convex analysis, e.g.,
the single-valuedness and continuity of
$\prox{\gamma \phi}$
(see also Remark~\ref{remark:obstacle} for some obstacles in the absence of convexity of $\phi$).
Instead, in Section~\ref{sec:stationarity},
we leverage tools from nonsmooth and set-valued analyses;
\ref{sec:appendix:set_valued} summarizes basic facts on set-valued analysis that we will use.
\subsection{Gradient mapping-type stationarity measure, criticality and stationarity} \label{sec:measure}
We introduce {\em critical point} of
$\cost+\phi$,
which has previously been discussed in the case where
$\nabla \cost$
is globally Lipschitz continuous
(see, e.g.,~\cite{Beck-Hallak16},~\cite[Def. 3.1 (ii)]{Themelis-Stella-Patrinos18}).
A critical point was originally referred to as ``{\rm L}-stationary point''
in, e.g.,~\cite{Beck-Hallak16}, while this paper adopts ``critical point'' in order to avoid confusion with ``limiting stationary point''
in Definition~\ref{definition:optimality}~\ref{enum:optimality:L_stationary}.

\begin{definition}[Critical point]
  \label{definition:critical}
  For
  $\cost+\phi$
  satisfying Assumption~\ref{assumption:basic},
  $\widebar{\bm{x}} \in \dom{\phi}$
  is said to be a {\em critical point}
  of
  $\cost+\phi$
  if
  $\mathcal{M}_{\gamma_{\widebar{\bm{x}}}}^{\cost,\phi}(\widebar{\bm{x}}) = 0$
  holds for some
  $\gamma_{\widebar{\bm{x}}} \in (0,\gamma_{\phi})$.
\end{definition}

We collect basic properties of
$\mathcal{M}_{\gamma}^{\cost,\phi}$
to derive relationship with stationarities.

\begin{lemma}[Basic properties of $\mathcal{M}_{\gamma}^{\cost,\phi}$]
  \label{lemma:measure_optimality}
  Let
  $\cost$
  and
  $\phi$
  satisfy Assumption~\ref{assumption:basic}.
  \begin{enumerate}[label=(\alph*)]
    \item
          \label{enum:lemma:measure_optimality:existence}
          For
          $\gamma \in (0,\gamma_{\phi})$
          and
          $\widebar{\bm{x}} \in \mathrm{int}(\dom{\cost})$,
          there exist
          $\widebar{\bm{v}} \in \Lsubdiff \cost(\widebar{\bm{x}})$
          and
          $\widebar{\bm{p}} \in \prox{\gamma\phi}(\widebar{\bm{x}}-\gamma\widebar{\bm{v}})$
          such that
          $\mathcal{M}_{\gamma}^{\cost,\phi}(\widebar{\bm{x}}) = \norm{(\widebar{\bm{x}}-\widebar{\bm{p}})/\gamma}<+\infty$.
    \item
          \label{enum:lemma:measure_optimality:any_gamma}
          If
          $\widebar{\bm{x}} \in \dom{\phi}$
          is a critical point of
          $\cost+\phi$,
          i.e.,
          $\mathcal{M}_{\gamma_{\widebar{\bm{x}}}}^{\cost,\phi}(\widebar{\bm{x}}) = 0$
          holds for some
          $\gamma_{\widebar{\bm{x}}} \in (0,\gamma_{\phi})$,
          then
          $\mathcal{M}_{\gamma}^{\cost,\phi}(\widebar{\bm{x}}) = 0$
          holds
          for all
          $\gamma \in (0,\gamma_{\widebar{\bm{x}}}]$.
    \item
          \label{enum:lemma:measure_optimality:stationarity_p}
          If
          $\widebar{\bm{x}} \in \dom{\phi}$
          is a critical point of
          $\cost + \phi$,
          then
          $\Lsubdiff \cost(\widebar{\bm{x}}) + \Psubdiff \phi(\widebar{\bm{x}}) \ni \bm{0}$.
  \end{enumerate}
  Moreover, by further assuming Assumption~\ref{assumption:smooth} on
  $\cost$,
  the following hold:
  \begin{enumerate}[label=(\alph*')]
    \item
          \label{enum:lemma:measure_optimality:existence_smooth}
          If
          $\mathcal{M}_{\widehat{\gamma}}^{\cost,\phi}(\widebar{\bm{x}}) <+\infty$
          holds for some
          $\widebar{\bm{x}} \in \mathrm{int}(\dom{\cost})$
          and
          $\widehat{\gamma} \in \mathbb{R}_{++}$,
          then
          there exists
          $\widebar{\bm{p}} \in \prox{\widehat{\gamma}\phi}(\widebar{\bm{x}}-\widehat{\gamma}\nabla\cost(\widebar{\bm{x}}))$
          such that
          $\mathcal{M}_{\widehat{\gamma}}^{\cost,\phi}(\widebar{\bm{x}}) = \norm{(\widebar{\bm{x}}-\widebar{\bm{p}})/\widehat{\gamma}}$.
    \item
          \label{enum:lemma:measure_optimality:any_gamma_smooth}
          If
          $\mathcal{M}_{\widehat{\gamma}}^{\cost,\phi}(\widebar{\bm{x}}) = 0$
          holds for some
          $\widebar{\bm{x}} \in \dom{\phi}$
          and
          $\widehat{\gamma} \in \mathbb{R}_{++}$,
          then
          $\mathcal{M}_{\gamma}^{\cost,\phi}(\widebar{\bm{x}}) = 0$
          holds for all
          $\gamma \in (0,\widehat{\gamma}]$.
    \item
          \label{enum:lemma:measure_optimality:stationarity_p_smooth}
          For
          $\widebar{\bm{x}} \in \dom{\phi}$,
          $\mathcal{M}_{\widehat{\gamma}}^{\cost,\phi}(\widebar{\bm{x}}) = 0$
          holds
          for some
          $\widehat{\gamma} \in \mathbb{R}_{++}$
          if and only if
          $\nabla \cost(\widebar{\bm{x}}) + \Psubdiff \phi(\widebar{\bm{x}}) \ni \bm{0}$.
  \end{enumerate}
\end{lemma}
\begin{proof}
  For~\ref{enum:lemma:measure_optimality:existence} and~\ref{enum:lemma:measure_optimality:existence_smooth}, it is enough to show the nonemptiness and closedness of
  $\frac{\widebar{\bm{x}}-\prox{\gamma\phi}(\widebar{\bm{x}}-\gamma\Lsubdiff \cost(\widebar{\bm{x}}))}{\gamma}\subset \mathcal{X}$,
  in particular of
  $\prox{\gamma\phi}(\widebar{\bm{x}}-\gamma\Lsubdiff \cost(\widebar{\bm{x}}))$,
  because, for a nonempty closed set
  $\mathcal{E}\subset \mathcal{X}$,
  $\dist(\bm{0}, \mathcal{E}) = \norm{\bm{w}}$
  holds for some
  $\bm{w} \in \mathcal{E}$~\cite[Exm. 1.20]{Rockafellar-Wets98}.

  \ref{enum:lemma:measure_optimality:existence}
  By Fact~\ref{fact:subdifferential_nonempty} in~\ref{sec:appendix:set_valued},
  $\Lsubdiff \cost(\widebar{\bm{x}})\neq \emptyset$
  is compact,
  and thus
  $E\coloneqq \widebar{\bm{x}}-\gamma\Lsubdiff \cost(\widebar{\bm{x}}) \subset \mathcal{X}$
  is nonempty and compact.
  Then, the {\em local boundedness} of
  $\prox{\gamma \phi}$
  by Fact~\ref{fact:prox_set_mapping}~\ref{enum:fact:prox_set_mapping:locally_bounded} ensures the boundedness of
  $\prox{\gamma\phi}(E)$
  (see Definition~\ref{definition:bounded}~\ref{enum:definition:bounded:locally_bounded} for local boundedness).
  Since
  $E$
  is compact and
  $\prox{\gamma\phi}$
  is outer semicontinuous by Fact~\ref{fact:prox_set_mapping}~\ref{enum:fact:prox_set_mapping:locally_bounded},
  $\prox{\gamma\phi}(E)$
  is a closed subset of
  $\mathcal{X}$
  by the standard fact~\cite[Thm. 5.25 (a)]{Rockafellar-Wets98} that
  an outer semicontinuous mapping maps a compact set to a closed set.

  \ref{enum:lemma:measure_optimality:existence_smooth}
  Suppose that
  $\mathcal{M}_{\widehat{\gamma}}^{\cost,\phi}(\widebar{\bm{x}}) < +\infty$
  holds for some
  $\widebar{\bm{x}} \in \mathrm{int}(\dom{\cost})$
  and
  $\widehat{\gamma} \in \mathbb{R}_{++}$.
  Then,
  $\prox{\widehat{\gamma}\phi}(\widebar{\bm{x}}-\widehat{\gamma} \Lsubdiff \cost(\widebar{\bm{x}}))=\prox{\widehat{\gamma} \phi}(\widebar{\bm{x}} - \widehat{\gamma} \nabla \cost(\widebar{\bm{x}}))$
  is nonempty by~\eqref{eq:measure_limiting} with the relation for
  $\mathcal{E}\subset \mathcal{X}$:
  $\dist(\bm{0},\mathcal{E})<+\infty\Leftrightarrow \mathcal{E}\neq \emptyset$.
  Since
  $\phi$
  is lower semicontinuous and
  $\frac{1}{2\widehat{\gamma}}\norm{\cdot - (\widebar{\bm{x}}-\widehat{\gamma}\nabla\cost(\widebar{\bm{x}}))}^{2}$
  is continuous,
  $\phi + \frac{1}{2\widehat{\gamma}}\norm{\cdot - (\widebar{\bm{x}}-\widehat{\gamma}\nabla\cost(\widebar{\bm{x}}))}^{2}$
  is lower semicontinuous~\cite[Lemma 1.27]{Bauschke-Combettes17}.
  Hence,
  $\prox{\widehat{\gamma} \phi}(\widebar{\bm{x}} - \widehat{\gamma} \nabla \cost(\widebar{\bm{x}})) =
    \argmin_{\bm{x} \in \mathcal{X}} (\phi(\bm{x}) + \frac{1}{2\widehat{\gamma}}\|\bm{x}-(\widebar{\bm{x}}-\widehat{\gamma}\nabla \cost(\widebar{\bm{x}}))\|^{2})$
  is closed in
  $\mathcal{X}$.

  \ref{enum:lemma:measure_optimality:any_gamma}
  We have
  $\prox{\gamma\phi}(\widebar{\bm{x}}-\gamma\bm{v}) = \argmin_{\bm{x}\in \mathcal{X}} \left(\phi(\bm{x}) + \frac{1}{2\gamma}\norm{\bm{x}-\widebar{\bm{x}}+\gamma\bm{v}}^{2}\right) = \argmin_{\bm{x} \in \mathcal{X}} \left(\cost(\widebar{\bm{x}}) + \inprod{\bm{v}}{\bm{x}-\widebar{\bm{x}}} + \phi(\bm{x}) + \frac{1}{2\gamma}\norm{\bm{x}-\widebar{\bm{x}}}^{2}\right)$
  for
  $\widebar{\bm{x}} \in \dom{\phi}$,
  $\bm{v} \in \Lsubdiff \cost(\widebar{\bm{x}})$
  and
  $\gamma \in \mathbb{R}_{++}$.
  Together with
  Lemma~\ref{lemma:measure_optimality}~\ref{enum:lemma:measure_optimality:existence} and~\eqref{eq:measure_limiting}, we have the relation for each fixed
  $\gamma \in (0,\gamma_{\phi})$:
  \begin{align}
     &
    \mathcal{M}_{\gamma}^{\cost, \phi}(\widebar{\bm{x}}) = 0
    \quad \Leftrightarrow
    (\exists \bm{v}_{\gamma} \in \Lsubdiff \cost (\widebar{\bm{x}})) \quad
    \widebar{\bm{x}} \in \prox{\gamma\phi}(\widebar{\bm{x}}-\gamma\bm{v}_{\gamma})
    \label{eq:relation:stationarity_gamma}
    \\
     & \Leftrightarrow (\exists \bm{v}_{\gamma} \in \Lsubdiff \cost (\widebar{\bm{x}}), \forall \bm{x} \in \mathcal{X})
    \\
     & \quad \quad \quad
    \cost(\widebar{\bm{x}}) + \phi(\widebar{\bm{x}}) \leq \cost(\widebar{\bm{x}}) + \inprod{\bm{v}_{\gamma}}{\bm{x}-\widebar{\bm{x}}} + \phi(\bm{x}) + \frac{1}{2\gamma}\norm{\bm{x}-\widebar{\bm{x}}}^{2}.
    \label{eq:relation:stationarity_gamma_inequality}
  \end{align}

  Suppose
  $\mathcal{M}_{\gamma_{\widebar{\bm{x}}}}^{\cost,\phi}(\widebar{\bm{x}}) = 0$
  holds with some
  $\gamma_{\widebar{\bm{x}}} \in (0,\gamma_{\phi})$.
  By the relation between~\eqref{eq:relation:stationarity_gamma} and~\eqref{eq:relation:stationarity_gamma_inequality}, there exists
  $\bm{v}_{\gamma_{\widebar{\bm{x}}}} \in \Lsubdiff \cost(\widebar{\bm{x}})$
  such that
  $\cost(\widebar{\bm{x}}) + \phi(\widebar{\bm{x}}) \leq \cost(\widebar{\bm{x}}) + \inprod{\bm{v}_{\gamma_{\widebar{\bm{x}}}}}{\bm{x}-\widebar{\bm{x}}} + \phi(\bm{x}) + \frac{1}{2\gamma_{\widebar{\bm{x}}}}\norm{\bm{x}-\widebar{\bm{x}}}^{2}$
  for all
  $\bm{x} \in \mathcal{X}$.
  For every
  $\gamma \in (0,\gamma_{\widebar{\bm{x}}}]$,
  we have
  $\cost(\widebar{\bm{x}}) + \phi(\widebar{\bm{x}}) \leq \cost(\widebar{\bm{x}}) + \inprod{\bm{v}_{\gamma_{\widebar{\bm{x}}}}}{\bm{x}-\widebar{\bm{x}}} + \phi(\bm{x}) + \frac{1}{2\gamma}\norm{\bm{x}-\widebar{\bm{x}}}^{2}\ (\bm{x}\in\mathcal{X})$
  by
  $\frac{1}{2\gamma_{\widebar{\bm{x}}}}\norm{\bm{x}-\widebar{\bm{x}}}^{2} \leq \frac{1}{2\gamma}\norm{\bm{x}-\widebar{\bm{x}}}^{2}$.
  Since
  the condition in~\eqref{eq:relation:stationarity_gamma_inequality} holds
  for every
  $\gamma \in (0,\gamma_{\widebar{\bm{x}}}]$,
  we get
  $\mathcal{M}_{\gamma}^{\cost,\phi}(\widebar{\bm{x}}) = 0$
  by~\eqref{eq:relation:stationarity_gamma}.

  \ref{enum:lemma:measure_optimality:any_gamma_smooth}
  Under Assumption~\ref{assumption:smooth},
  Lemma~\ref{lemma:measure_optimality}~\ref{enum:lemma:measure_optimality:existence_smooth} ensures the relation, analogous to~\eqref{eq:relation:stationarity_gamma} (but beyond
  $\gamma \in (0,\gamma_{\phi})$),
  for each fixed
  $\gamma\in\mathbb{R}_{++}$:
  $\mathcal{M}_{\gamma}^{\cost,\phi}(\widebar{\bm{x}}) = 0$
  $\Leftrightarrow$
  $\widebar{\bm{x}} \in \prox{\gamma\phi}(\widebar{\bm{x}} - \gamma\nabla \cost(\widebar{\bm{x}}))$
  $\Leftrightarrow$
  $(\forall \bm{x} \in \mathcal{X})\ \cost(\widebar{\bm{x}}) + \phi(\widebar{\bm{x}}) \leq \cost(\widebar{\bm{x}}) + \inprod{\nabla \cost(\widebar{\bm{x}})}{\bm{x}-\widebar{\bm{x}}} + \phi(\bm{x}) + \frac{1}{2\gamma}\norm{\bm{x}-\widebar{\bm{x}}}^{2}$.
  Hence, the essentially same discussion after~\eqref{eq:relation:stationarity_gamma_inequality} yields~\ref{enum:lemma:measure_optimality:any_gamma_smooth}.

  \ref{enum:lemma:measure_optimality:stationarity_p}
  In this proof, we use the {\em $\gamma(>0)$-level proximal subdifferential}~\cite{Rockafellar21,Wang-Wang24},
  defined by
  $\Psubdiff^{\gamma}\phi(\widebar{\bm{x}}) \coloneqq \{\bm{v} \in \mathcal{X} \mid \phi(\bm{x}) \geq \phi(\widebar{\bm{x}}) + \inprod{\bm{v}}{\bm{x}-\widebar{\bm{x}}} - \frac{1}{2\gamma}\norm{\bm{x}-\widebar{\bm{x}}}^{2}\ (\forall \bm{x} \in \mathcal{X}) \}$
  for
  $\widebar{\bm{x}} \in \mathcal{X}$.
  For
  $\gamma > 0$
  and for
  $\widebar{\bm{x}}, \widebar{\bm{u}} \in \mathcal{X}$,
  we have the relation~\cite[Prop. 3.4]{Wang-Wang24}:
  $\widebar{\bm{u}} \in \Psubdiff^{\gamma} \phi(\widebar{\bm{x}})$
  if and only if
  $\widebar{\bm{x}} \in \prox{\gamma\phi}(\widebar{\bm{x}} + \gamma\widebar{\bm{u}})$.
  Together with Lemma~\ref{lemma:measure_optimality}~\ref{enum:lemma:measure_optimality:existence}, the proof is completed by
  \begin{align}
     & (\exists \gamma_{\widebar{\bm{x}}}\in (0,\gamma_{\phi})) \ \mathcal{M}_{\gamma_{\widebar{\bm{x}}}}^{\cost,\phi}(\widebar{\bm{x}}) = 0
    \
    \Leftrightarrow
    (\exists \gamma_{\widebar{\bm{x}}}\in (0,\gamma_{\phi}),\exists \bm{v}_{\gamma_{\widebar{\bm{x}}}} \in \Lsubdiff \cost(\widebar{\bm{x}}))  \ \widebar{\bm{x}} \in \prox{\gamma_{\widebar{\bm{x}}}\phi}(\widebar{\bm{x}} - \gamma_{\widebar{\bm{x}}}\bm{v}_{\gamma_{\widebar{\bm{x}}}}) \hspace{-2em} \\
     & \Leftrightarrow
    (\exists \gamma_{\widebar{\bm{x}}}\in (0,\gamma_{\phi}),\exists \bm{v}_{\gamma_{\widebar{\bm{x}}}} \in \Lsubdiff \cost(\widebar{\bm{x}})) \quad
    -\bm{v}_{\gamma_{\widebar{\bm{x}}}} \in \Psubdiff^{\gamma_{\widebar{\bm{x}}}} \phi (\widebar{\bm{x}}) \\
     & \Leftrightarrow (\exists \gamma_{\widebar{\bm{x}}}\in (0,\gamma_{\phi}))\quad \bm{0}\in \Lsubdiff \cost(\widebar{\bm{x}}) + \Psubdiff^{\gamma_{\widebar{\bm{x}}}} \phi(\widebar{\bm{x}})
    \quad \Rightarrow  \bm{0}\in \Lsubdiff \cost(\widebar{\bm{x}}) + \Psubdiff \phi(\widebar{\bm{x}}), \label{eq:relation:criticality_stationarity}
  \end{align}
  where the last implication follows from
  $\Psubdiff \phi(\widebar{\bm{x}})=\bigcup_{\gamma\in\mathbb{R}_{++}}\Psubdiff^{\gamma}\phi(\widebar{\bm{x}})$~\cite[Prop. 3.6 (ii)]{Wang-Wang24}.

  \ref{enum:lemma:measure_optimality:stationarity_p_smooth}
  Under Assumption~\ref{assumption:smooth},
  a similar discussion in~\eqref{eq:relation:criticality_stationarity} together with Lemma~\ref{lemma:measure_optimality}~\ref{enum:lemma:measure_optimality:existence_smooth} yields
  the relation:
  $ (\exists \widehat{\gamma}\in \mathbb{R}_{++}) \ \mathcal{M}_{\widehat{\gamma}}^{\cost,\phi}(\widebar{\bm{x}}) = 0 \Leftrightarrow
    (\exists \widehat{\gamma}\in \mathbb{R}_{++}) \ \widebar{\bm{x}} \in \prox{\widehat{\gamma}\phi}(\widebar{\bm{x}} - \widehat{\gamma}\nabla \cost(\widebar{\bm{x}}))
    \Leftrightarrow (\exists \widehat{\gamma}\in \mathbb{R_{++}})\ -\nabla \cost(\widebar{\bm{x}}) \in \Psubdiff^{\widehat{\gamma}} \phi(\widebar{\bm{x}})$.
  The last condition is equivalent to
  $- \nabla \cost(\widebar{\bm{x}}) \in \Psubdiff \phi(\widebar{\bm{x}})$
  by
  $\Psubdiff \phi(\widebar{\bm{x}})=\bigcup_{\gamma\in\mathbb{R}_{++}}\Psubdiff^{\gamma}\phi(\widebar{\bm{x}})$~\cite[Prop. 3.6 (ii)]{Wang-Wang24}.
\end{proof}

Lemma~\ref{lemma:measure_optimality}~\ref{enum:lemma:measure_optimality:stationarity_p} indicates that the criticality of
$\cost + \phi$
may depend on the decomposition of the cost function
$\cost + \phi$
into
$\cost$
and
$\phi$
under Assumption~\ref{assumption:basic} alone
(see also Remark~\ref{remark:measure}).
Under Assumption~\ref{assumption:regular} (resp.~\ref{assumption:smooth}),
Proposition~\ref{proposition:criticality_stationarity} shows that the criticality of
$\cost + \phi$
is equivalent to the
  {\rm F}-stationarity (resp. {\rm P}-stationarity) of
$\cost + \phi$.
Since stationarities of
$\cost + \phi$
depend on the cost function
$\cost + \phi$
itself,
the criticality of
$\cost + \phi$
is independent of the chosen pair
$(\cost,\phi)$
within the class of decompositions satisfying Assumption~\ref{assumption:regular} (resp.~\ref{assumption:smooth}).
Proposition~\ref{proposition:criticality_stationarity} together with~\eqref{eq:hierarchical_optimality} implies that
the criticality of
$\cost + \phi$
is a necessary condition for the local optimality of
$\cost+\phi$
under Assumptions~\ref{assumption:regular} or~\ref{assumption:smooth}.

\begin{proposition}[Relation between criticality and stationarities]
  \label{proposition:criticality_stationarity}
  Let
  $\cost$
  and
  $\phi$
  satisfy Assumption~\ref{assumption:basic}.
  For
  $\bm{x}^{\star} \in \dom{\phi}$,
  consider the following conditions:
  \begin{enumerate}[label=(\roman*)]
    \item
          \label{enum:proposition:criticality_stationarity:criticality}
          $\bm{x}^{\star}$
          is a critical point of
          $\cost+\phi$,
          i.e.,
          $\mathcal{M}_{\gamma_{\bm{x}^{\star}}}^{\cost,\phi}(\bm{x}^{\star})=0$
          holds with some
          $\gamma_{\bm{x}^{\star}} \in (0,\gamma_{\phi})$.
    \item
          \label{enum:proposition:criticality_stationarity:criticality_bound}
          $\mathcal{M}_{\gamma}^{\cost,\phi}(\bm{x}^{\star}) = 0$
          holds for all
          $\gamma \in (0,\gamma_{\bm{x}^{\star}}]$
          with some
          $\gamma_{\bm{x}^{\star}} \in (0,\gamma_{\phi})$.
  \end{enumerate}
  Then, for
  $\bm{x}^{\star} \in \dom{\phi}$,
  the following relations hold:
  \begin{enumerate}[label=(\alph*)]
    \item
          \label{enum:proposition:criticality_stationarity:regular}
          Under Assumption~\ref{assumption:regular},
          $\bm{x}^{\star}$
          is an {\rm F}-stationary point of
          $\cost+\phi$
          $\Leftrightarrow$
          \ref{enum:proposition:criticality_stationarity:criticality}
          $\Leftrightarrow$ \ref{enum:proposition:criticality_stationarity:criticality_bound}.
    \item
          \label{enum:proposition:criticality_stationarity:smooth}
          Under Assumption~\ref{assumption:smooth},
          $\bm{x}^{\star}$
          is a {\rm P}-stationary point of
          $\cost+\phi$
          $\Leftrightarrow$
          \ref{enum:proposition:criticality_stationarity:criticality}
          $\Leftrightarrow$ \ref{enum:proposition:criticality_stationarity:criticality_bound}.
  \end{enumerate}
\end{proposition}
\begin{proof}
  The relation \ref{enum:proposition:criticality_stationarity:criticality}
  $\Leftrightarrow$ \ref{enum:proposition:criticality_stationarity:criticality_bound} can be verified by Definition~\ref{definition:critical} and Lemma~\ref{lemma:measure_optimality}~\ref{enum:lemma:measure_optimality:any_gamma}.

  \ref{enum:proposition:criticality_stationarity:regular}
  We show the relation under Assumption~\ref{assumption:regular}:
  ``$\bm{x}^{\star}$
  is an {\rm F}-stationary point of
  $\cost + \phi$
  $\Leftrightarrow$
  $\ref{enum:proposition:criticality_stationarity:criticality}$,''
  where the implication
  $(\Leftarrow)$
  is verified by Lemma~\ref{lemma:measure_optimality}~\ref{enum:lemma:measure_optimality:stationarity_p} with~\eqref{eq:sum_subdifferential}.
  For
  $(\Rightarrow)$,
  assume that
  $\bm{x}^{\star} \in \dom{\phi}$
  is an F-stationary point of
  $\cost+\phi$,
  i.e.,
  $\bm{0}\in \Lsubdiff \cost(\bm{x}^{\star}) + \Lsubdiff \phi(\bm{x}^{\star})$
  by~\eqref{eq:sum_subdifferential}.
  There exists
  $\bm{v} \in \Lsubdiff \phi(\bm{x}^{\star})$
  such that
  $- \bm{v} \in \Lsubdiff \cost(\bm{x}^{\star})$.
  The prox-regularity of
  $\phi$
  at
  $\bm{x}^{\star}$
  for
  $\bm{v}$
  yields
  $\prox{\gamma_{\bm{x}^{\star}}\phi}(\bm{x}^{\star} + \gamma_{\bm{x}^{\star}}\bm{v}) = \{\bm{x}^{\star}\}$
  with some
  $\gamma_{\bm{x}^{\star}} \in (0,\gamma_{\phi})$
  by Fact~\ref{fact:prox_set_mapping}~\ref{enum:fact:prox_set_mapping:prox_single}.
  Hence,
  by~\eqref{eq:measure_limiting},
  $0 \leq \mathcal{M}_{\gamma_{\bm{x}^{\star}}}^{\cost,\phi}(\bm{x}^{\star}) \leq \norm{(\bm{x}^{\star}-\prox{\gamma_{\bm{x}^{\star}}\phi}(\bm{x}^{\star} + \gamma_{\bm{x}^{\star}}\bm{v}))/\gamma_{\bm{x}^{\star}}} = 0$.

  \ref{enum:proposition:criticality_stationarity:smooth}
  The relation
  ``$\bm{x}^{\star}$
  is a {\rm P}-stationary point of
  $\cost + \phi$
  $\Leftrightarrow$
  \ref{enum:proposition:criticality_stationarity:criticality_bound}''
  can be verified under Assumption~\ref{assumption:smooth} by:
  $\bm{x}^{\star}$
  is a {\rm P}-stationary point of
  $\cost + \phi$
  $\overset{\eqref{eq:sum_subdifferential_smooth}}{\Leftrightarrow}$
  $\bm{0} \in \nabla \cost(\bm{x}^{\star}) + \Psubdiff \phi(\bm{x}^{\star})$
  $\overset{\rm Lemma~\ref{lemma:measure_optimality}~\ref{enum:lemma:measure_optimality:stationarity_p_smooth}}{\Leftrightarrow}$
  $(\exists \widehat{\gamma} \in \mathbb{R}_{++})\ \mathcal{M}_{\widehat{\gamma}}^{\cost,\phi}(\bm{x}^{\star})=0$
  $\overset{\rm Lemma~\ref{lemma:measure_optimality}~\ref{enum:lemma:measure_optimality:any_gamma_smooth}}{\Leftrightarrow}$
  $(\exists \widehat{\gamma} \in \mathbb{R}_{++},\ \forall \gamma \in (0, \widehat{\gamma}])\ \mathcal{M}_{\gamma}^{\cost,\phi}(\bm{x}^{\star})=0$
  $\Leftrightarrow$
  \ref{enum:proposition:criticality_stationarity:criticality_bound}
  $(\exists \gamma_{\bm{x}^{\star}} \in (0,\gamma_{\phi}),\ \forall \gamma \in (0,\gamma_{\bm{x}^{\star}}])\ \mathcal{M}_{\gamma}^{\cost,\phi}(\bm{x}^{\star})=0$.

\end{proof}

\begin{remark}[``(Local) optimality $\Rightarrow$ criticality'' may not hold without regularities]
  \label{remark:measure}
  Consider the cost function
  $J(x)\coloneqq \abs{x}$
  with the decomposition of
  $J=\cost + \phi$
  into
  $\cost(x)\coloneqq 2\abs{x}$
  and
  $\phi(x)\coloneqq -\abs{x}$,
  where
  $\phi$
  is not prox-regular at
  $0$
  (see the paragraph just after Definition~\ref{definition:regular});
  hence Assumption~\ref{assumption:regular} is not satisfied.
  Clearly,
  $x^{\star}=0$
  is the unique (local) minimizer of
  $J=\cost+\phi$.
  However,
  since
  $\Lsubdiff \cost(0) + \Psubdiff \phi(0) = \emptyset$
  holds due to
  $\Psubdiff \phi(0) = \emptyset$,
  the contraposition of Lemma~\ref{lemma:measure_optimality}~\ref{enum:lemma:measure_optimality:stationarity_p} implies that
  $x^{\star}$
  is not a critical point of
  $\cost + \phi$.
  In contrast, consider a trivial decomposition of
  $J=\widetilde{\cost}+\widetilde{\phi}$
  into
  $\widetilde{\cost}(x) \coloneqq \abs{x}$
  and
  $\widetilde{\phi} \equiv 0$.
  Then,
  $\widetilde{\cost}$
  and
  $\widetilde{\phi}$
  satisfy Assumption~\ref{assumption:regular},
  and
  $0$
  is a critical point of
  $\widetilde{\cost}+\widetilde{\phi}$
  because
  $\mathcal{M}_{\gamma}^{\widetilde{\cost},\widetilde{\phi}}(0) = \dist(0, \Lsubdiff \widetilde{\cost}(0)) = \dist(0, [-1,1]) = 0\ (\forall \gamma \in \mathbb{R}_{++})$.
  This example indicates that, under Assumption~\ref{assumption:basic} only,
  (i)
  the criticality of
  $\cost+\phi$
  at
  $\bm{x}^{\star} \in \dom{\phi}$
  may not be a necessary condition for the (local) optimality of
  $\cost+\phi$
  at
  $\bm{x}^{\star}$;
  and
  (ii)
  the criticality may depend on the decomposition of
  $J=\cost+\phi$
  into
  $\cost$
  and
  $\phi$.
\end{remark}

\subsection{Asymptotic analysis of $\mathcal{M}_{\gamma}^{\cost,\phi}$} \label{sec:stationarity:asymptotic}
By Proposition~\ref{proposition:criticality_stationarity}
under Assumption~\ref{assumption:regular} or~\ref{assumption:smooth},
a stationary point of
$\cost+\phi$
can be characterized as its critical point
$\widebar{\bm{x}}\in\dom{\phi}$
where
$\mathcal{M}_{\gamma}^{\cost,\phi}(\widebar{\bm{x}})=0$
holds with a sufficiently small
$\widebar{\gamma} \in (0,\gamma_{\widebar{\bm{x}}}]$
and with some
$\gamma_{\widebar{\bm{x}}} \in (0,\gamma_{\phi})$.
However, it is not trivial to choose
$\widebar{\gamma}$
properly because its upper bound
$\gamma_{\widebar{\bm{x}}}$
depends on each critical point
$\widebar{\bm{x}}$
of
$\cost + \phi$.
Moreover, from a viewpoint of designing and analyzing iterative algorithms that produce a sequence
$(\bm{x}_{n})_{n=1}^{\infty} \subset \mathcal{X}$,
a characterization of stationarity at a fixed
$\widebar{\bm{x}}$
through the condition
$\mathcal{M}_{\widebar{\gamma}}^{\cost,\phi}(\widebar{\bm{x}})=0$
is not enough.
These situations motivate us to investigate relations between stationary points and asymptotic behaviors of
$\mathcal{M}_{\gamma_{n}}^{\cost,\phi}(\bm{x}_{n})$
with sequences
$(\gamma_{n})_{n=1}^{\infty} \subset (0,\gamma_{\phi})$
and
$(\bm{x}_{n})_{n=1}^{\infty} \subset \mathcal{X}$
converging to
$\widebar{\bm{x}}$.

The goal of this subsection is to show a key result (Theorem~\ref{theorem:asymptotic_approximation})
for characterizing stationary points of
$\cost+\phi$
via asymptotic behaviors of
$\mathcal{M}_{\gamma_{n}}^{\cost,\phi}$
(see Section~\ref{sec:measure:stationarity}).
Moreover, we also show asymptotic properties of a variant
$\mathcal{M}_{\gamma_{n}}^{\cost^{\langle \mu_{n} \rangle},\phi}$
of
$\mathcal{M}_{\gamma_{n}}^{\cost,\phi}$
with
a {\em smoothing function}
$\cost^{\langle\mu\rangle}\ (\mu > 0)$
of a nonsmooth function
$\cost$,
where
$\mu_{n} \searrow 0$.
A smoothing function, e.g.,~\cite{Chen12}, has been utilized for minimization of a nonsmooth function~\cite{Zhang-Chen09,Chen12,Burke-Hoheisel-Kanzow13,Bot-Hendrich15,Burke-Hoheisel17,Bohm-Wright21,Liu-Xia24,Bian-Chen20,Yu-Zhang22,Nishioka-Kanno24,Kume-Yamada24A,Kume-Yamada25B,Long-Zeng-Li-Dao-Peng26}.
In what follows, Section~\ref{sec:stationarity:asymptotic:smoothing} introduces a smoothing function and its useful property for our analysis, and then Section~\ref{sec:stationarity:asymptotic:theorem} shows asymptotic properties of the gradient mapping-type stationarity measure.

\subsubsection{Smoothing function and boundedness of its gradient sequence} \label{sec:stationarity:asymptotic:smoothing}
We introduce a smoothing function of
$\cost$,
and present its boundedness property of a gradient sequence that we will use in the proof of Theorem~\ref{theorem:asymptotic_approximation} in Section~\ref{sec:stationarity:asymptotic:theorem}.

\begin{definition}[Smoothing function]
  \label{definition:smoothing}
  Let
  $\cost:\mathcal{X}\to\exR$
  satisfy Assumption~\ref{assumption:basic}~\ref{enum:problem:origin:cost}.
  Then, a parameterized function
  $\{\cost^{\langle \mu \rangle}:\mathcal{X}\to \exR \mid \mu \in (0,\widetilde{\mu})\}$
  with some
  $\widetilde{\mu} \in \exRp$
  is said to be a {\em smoothing function} of
  $\cost$
  if
  $\mathrm{int}(\dom{\cost^{\langle \mu \rangle}}) = \mathrm{int}(\dom{\cost})\ (\mu\in (0,\widetilde{\mu}))$
  and
  the following five conditions hold
  (Note: we also use a simpler notation
  $\{\cost^{\langle \mu \rangle}\}_{\mu \in (0,\widetilde{\mu})}$
  with
  $\widetilde{\mu} \in \exRp$
  for a smoothing function of
  $\cost$):
  \begin{enumerate}[label=(\alph*)]
    \item
          (Pointwise convergence) \label{enum:consistency:pointwise}
          $\lim\limits_{(0,\widetilde{\mu})\ni \mu\to 0}\cost^{\langle \mu \rangle}(\widebar{\bm{x}}) = \cost(\widebar{\bm{{x}}})$
          holds for all
          $\widebar{\bm{x}} \in \mathrm{int}(\dom{\cost})$.
    \item
          (Uniform bound) \label{enum:consistency:uniform}
          There exists
          $\kappa \in \mathbb{R}_{++}$
          such that\footnote{
            Instead of the condition~\ref{enum:consistency:uniform} in Definition~\ref{definition:smoothing},
            in~\cite{Bian-Chen20,Yu-Zhang22},
            its weaker condition
            $(\exists \kappa \in \mathbb{R}_{++},\ \forall \mu,\mu' \in (0,\widetilde{\mu})\ \mathrm{s.t.
              }\ \mu' \leq \mu,\ \forall \widebar{\bm{x}} \in \mathrm{int}(\dom{\cost}))\  \abs{\cost^{\langle \mu' \rangle }(\widebar{\bm{x}}) - \cost^{\langle \mu \rangle}(\widebar{\bm{x}})} \leq \kappa (\mu-\mu')$
            is imposed on
            $\{\cost^{\langle \mu \rangle}\}_{\mu \in (0,\widetilde{\mu})}$
            together with
            the conditions~\ref{enum:consistency:pointwise} and~\ref{enum:consistency:gradient_Lipschitz}-\ref{enum:consistency:consistency} in Definition~\ref{definition:smoothing}.
            Fortunately, even for
            $\{\cost^{\langle \mu \rangle}\}_{\mu \in (0,\widetilde{\mu})}$
            satisfying this weaker condition,
            its modification
            $\widehat{\cost}^{\langle \mu \rangle} \coloneqq \cost^{\langle \mu \rangle} - \kappa \mu\ (\mu \in (0,\widetilde{\mu}))$
            achieves
            $0 \leq \widehat{\cost}^{\langle \mu' \rangle}(\bm{x}) - \widehat{\cost}^{\langle \mu \rangle}(\bm{x}) \leq 2\kappa(\mu-\mu')\ (\forall \mu,\mu'\in (0,\widetilde{\mu})\ \mathrm{s.t.
              }\ \mu'\leq \mu, \forall \bm{x} \in \mathrm{int}(\dom{\cost}))$,
            i.e.,
            the condition~\ref{enum:consistency:uniform} in Definition~\ref{definition:smoothing}, while keeping the other conditions.
            Hence,
            $\{\widehat{\cost}^{\langle \mu \rangle}\}_{\mu \in (0,\widetilde{\mu})}$
            is a smoothing function in
            Definition~\ref{definition:smoothing}.
          }
          $0 \leq \cost^{\langle \mu' \rangle}(\bm{x}) - \cost^{\langle \mu \rangle}(\bm{x}) \leq \kappa(\mu-\mu')\ (\forall \mu,\mu'\in (0,\widetilde{\mu})\ \mathrm{s.t.}\ \mu'\leq \mu, \forall \bm{x} \in \mathrm{int}(\dom{\cost}))$.
    \item
          (Smoothness) \label{enum:consistency:gradient_Lipschitz}
          For every
          $\mu \in (0,\widetilde{\mu})$,
          $\cost^{\langle \mu \rangle}$
          is continuously differentiable over
          $\mathrm{int}(\dom{\cost})$,
          and its gradient
          $\nabla \cost^{\langle \mu\rangle}: \mathrm{int}(\dom{\cost}) \to \mathcal{X}$
          is $L_{\nabla \cost^{\langle \mu \rangle}}$-Lipschitz continuous
          over the convex set
          $\mathrm{int}(\dom{\cost})$
          with a Lipschitz constant
          $ L_{\nabla \cost^{\langle \mu \rangle}} > 0$.
    \item
          (Lipschitz constant $L_{\nabla \cost^{\langle \mu \rangle}}$) \label{enum:consistency:gradient_Lipschitz_constant}
          There exist
          $\varpi_{1}\in \mathbb{R}_{+}$
          and
          $\varpi_{2} \in \mathbb{R}_{++}$
          such that
          $L_{\nabla \cost^{\langle \mu \rangle}}$
          in~\ref{enum:consistency:gradient_Lipschitz}
          is given by
          $L_{\nabla \cost^{\langle \mu \rangle}} \coloneqq \varpi_{1}+\varpi_{2}\mu^{-1}\ (\forall\mu \in (0,\widetilde{\mu}))$.
    \item
          (Gradient consistency for $\cost$) \label{enum:consistency:consistency}
          For every
          $\widebar{\bm{x}} \in \mathrm{int}(\dom{\cost})$,
          we have
          \begin{equation}
            \thickmuskip=0.3\thickmuskip
            \medmuskip=0.3\medmuskip
            \thinmuskip=0.3\thinmuskip
            \arraycolsep=0.3\arraycolsep
            \label{eq:consistency}
            \Lsubdiff \cost(\widebar{\bm{x}})
            = \left\{\bm{v} \in \mathcal{X} \middle|
            \begin{array}{l}
              \displaystyle
              \exists (\bm{x}_{n})_{n=1}^{\infty}\subset \mathrm{int}(\dom{\cost}),\
              \exists (\mu_{n})_{n=1}^{\infty}\subset (0,\widetilde{\mu}) \\
              \displaystyle
              {\rm s.t.}\ \widebar{\bm{x}}=\lim_{n\to\infty}\bm{x}_{n}, \
              \mu_{n}\searrow 0,\ {\rm and}\ \bm{v}=\lim_{n\to\infty}\nabla \cost^{\langle \mu_{n} \rangle} (\bm{x}_{n})
            \end{array}
            \right\},
          \end{equation}
          i.e.,
          $\Lsubdiff \cost(\widebar{\bm{x}}) = \bigcup_{\mathrm{int}(\dom{\cost}) \ni \bm{x}_{n}\to \widebar{\bm{x}},\ (0,\widetilde{\mu}) \ni \mu_{n} \searrow 0} \Limsup_{n\to\infty} \nabla \cost^{\langle \mu_{n} \rangle} (\bm{x}_{n})$.
          The condition~\eqref{eq:consistency} is known as a {\em gradient consistency condition}~\cite{Chen12}.
  \end{enumerate}
\end{definition}
Constructions of such smoothing functions are found in, e.g.,~\cite{Chen12,Burke-Hoheisel-Kanzow13,Burke-Hoheisel17,Kume-Yamada24A}.
Example~\ref{example:cost} below illustrates an example of a smoothing function designed with the Moreau envelope.
By combining the conditions~\ref{enum:consistency:pointwise} and~\ref{enum:consistency:uniform} in Definition~\ref{definition:smoothing}, we have the uniform convergence property (see, e.g.,~\cite{Bian-Chen20,Yu-Zhang22}):
\begin{equation}
  (\forall \mu \in (0,\widetilde{\mu}),\ \forall \widebar{\bm{x}} \in \mathrm{int}(\dom{\cost})) \quad
  0\leq \cost(\widebar{\bm{x}}) - \cost^{\langle \mu \rangle}(\widebar{\bm{x}}) \leq \kappa \mu.\label{eq:uniform_bounded_smoothing_original}
\end{equation}

\begin{example}[On $\cost$ achieving Assumption~\ref{assumption:regular}~\ref{enum:assumption:regular:cost} and smoothing function]
  \label{example:cost}
  The function
  $\cost \coloneqq h+ g\circ \mathfrak{S} + \iota_{\smoothingDomain}:\mathcal{X}\to\exR$
  satisfies Assumption~\ref{assumption:basic}~\ref{enum:problem:origin:cost} and Assumption~\ref{assumption:regular}~\ref{enum:assumption:regular:cost}
  by~\cite[Lemma 2.4 and its proof in Appendix A]{Kume-Yamada24A}, if
  (i)
  $h:\mathcal{X} \to \mathbb{R}$
  is differentiable such that
  $\nabla h:\mathcal{X}\to\mathcal{X}$
  is Lipschitz continuous;
  (ii)
  $g:\mathcal{Z} \to \mathbb{R}$
  is Lipschitz continuous and $\eta(>0)$-weakly convex, i.e.,
  $g+\frac{\eta}{2}\norm{\cdot}^{2}$
  is convex over the Euclidean space
  $\mathcal{Z}$;
  (iii)
  $\mathfrak{S}:\mathcal{X} \to \mathcal{Z}$
  is continuously differentiable;
  and
  (iv)
  $\smoothingDomain\subset \mathcal{X}$
  is a certain closed convex set such that
  $\dom{\phi}\subset \mathrm{int}(\smoothingDomain) (= \mathrm{int}(\dom{\cost}))$
  (see~\eqref{eq:indicator}
  for the indicator function
  $\iota_{\smoothingDomain}$).
  For this
  $\cost=h+g\circ\mathfrak{S} + \iota_{\smoothingDomain}$,
  we can construct a smoothing function
  $\{\cost^{\langle \mu \rangle}\}_{\mu \in (0,\widetilde{\mu})}$
  in Definition~\ref{definition:smoothing} as follows.
  Consider
  $\cost^{\langle \mu \rangle}(\widebar{\bm{x}}) \coloneqq (h+\moreau{g}{\mu}\circ\mathfrak{S} + \iota_{\smoothingDomain})(\widebar{\bm{x}})$
  for
  $\mu \in (0,\widetilde{\mu})$
  with
  $\widetilde{\mu} \coloneqq (2\eta)^{-1}$
  and
  $\widebar{\bm{x}} \in \mathcal{X}$,
  where
  $\moreau{g}{\mu}:\mathcal{Z}\to\mathbb{R}:\widebar{\bm{z}} \mapsto g(\prox{\mu g}(\widebar{\bm{z}})) + \frac{1}{2\mu}\norm{\prox{\mu g}(\widebar{\bm{z}}) - \widebar{\bm{z}}}^{2}$
  is
  the {\em Moreau envelope} of
  $g$,
  and
  $\prox{\mu g}$
  with
  $\mu \in (0,\eta^{-1})$
  is single-valued.
  Then,
  $\{\cost^{\langle \mu \rangle}\}_{\mu \in (0,\widetilde{\mu})}$
  is a smoothing function of
  $\cost$
  if
  $\mathfrak{S}$
  and its Fr\'echet derivative
  are Lipschitz continuous over
  $\mathrm{int}(\smoothingDomain)$~\cite[Lemma 4.1, Prop. 4.2, and Thm. 4.4]{Kume-Yamada24A},
  where
  $\kappa$
  in Definition~\ref{definition:smoothing}~\ref{enum:consistency:uniform} is given by
  $L_{g}^{2}$
  with a Lipschitz constant
  $L_{g}$
  of
  $g$~\cite[Lemma 4.1]{Kume-Yamada24A},
  $\varpi_{1}$
  and
  $\varpi_{2}$
  in Definition~\ref{definition:smoothing}~\ref{enum:consistency:gradient_Lipschitz_constant} are given in~\cite[Prop. 4.2]{Kume-Yamada24A}.
  This type of smoothing function has been used in variable smoothing-type algorithms, e.g.,~\cite{Bot-Hendrich15,Bohm-Wright21,Liu-Xia24,Kume-Yamada24A,Lopes-Peres-Bilches25,Kume-Yamada25B,Yazawa-Kume-Yamada26,Long-Zeng-Li-Dao-Peng26}.
  This function
  $\cost \coloneqq h+g\circ \mathfrak{S}+\iota_{D}$
  appears, e.g., in robust signal processing and machine learning applications~\cite{Bian-Chen20,Yu-Zhang22,Duchi-Ruan18,Zhen-Ma-Xue24,Yazawa-Kume-Yamada26,Charisopoulos-Chen-Davis-Diaz-Ding-Drusvyatskiy21,Suzuki-Yukawa21,Yang-Chen-Ma-Chen-Gu-So19,Kume-Yamada25C,Carrillo-Ramirez-Arce-Barner-Sadler16} for estimating
  $\bm{x}^{\diamond} \in \mathcal{X}$
  from an observation
  $\bm{y} \coloneqq \mathcal{A}(\bm{x}^{\diamond}) + \bm{\epsilon}\in\mathbb{R}^{m}$
  with noise
  $\bm{\epsilon} \in \mathbb{R}^{m}$
  and known measurement mapping
  $\mathcal{A}:\mathcal{X} \to \mathbb{R}^{m}$
  such that
  $\mathcal{A}$
  is at least continuously differentiable.
  Indeed, by letting
  $h\equiv 0$,
  $\mathfrak{S}:\mathcal{X}\to\mathbb{R}^{m}:\widebar{\bm{x}}\mapsto \bm{y} - \mathcal{A}(\widebar{\bm{x}})$,
  $g:\mathbb{R}^{m} \to \mathbb{R}:\widebar{\bm{z}} \mapsto \sum_{i=1}^{m} \ell([\widebar{\bm{z}}]_{i})$
  with a Lipschitz continuous and weakly convex function
  $\ell:\mathbb{R}\to\mathbb{R}$,
  and
  $\smoothingDomain\coloneqq \mathcal{X}$,
  the function
  $h+g\circ \mathfrak{S}+\iota_{\smoothingDomain}=g\circ\mathfrak{S}$
  serves as a data fidelity function for robust signal estimation, where
  the loss function
  $\ell$
  is chosen from, e.g.,
  the $\ell_{1}$-norm
  $\norm{\cdot}_{1}$,
  the smoothly clipped absolute deviation (SCAD)~\cite{Fan-Li01} and the Minimax Concave Penalty (MCP)~\cite{Zhang10}.
  For other examples of this type of
  $\cost$,
  see, e.g.,~\cite{Bohm-Wright21,Liu-Xia24,Kume-Yamada24A,Kume-Yamada25B,Kume-Yamada25C}.
\end{example}

Lemma~\ref{lemma:bounded_gradient} presents a key property of a smoothing function
for deriving our asymptotic analysis of
$\mathcal{M}_{\gamma}^{\cost,\phi}$
in Theorem~\ref{theorem:asymptotic_approximation}.
We note that this boundedness property has been shown for specific smoothing functions, e.g.,~\cite{Kume-Yamada25B} for a smoothing function in Example~\ref{example:cost}, and~\cite{Borges-Sagastizbal-Solodov21} for a smoothing function of an optimal value function.
Moreover, this boundedness property is not necessarily ensured by the gradient consistency property of a smoothing function alone, as remarked in~\cite[p.138]{Borges-Sagastizbal-Solodov21}.
Lemma~\ref{lemma:bounded_gradient} reveals that this boundedness property holds for every smoothing function in Definition~\ref{definition:smoothing}.

\begin{lemma}[Boundedness of gradient sequences of smoothing function]
  \label{lemma:bounded_gradient}
  Let
  $\cost$
  satisfy Assumption~\ref{assumption:basic}~\ref{enum:problem:origin:cost} and let
  $\{\cost^{\langle \mu \rangle}\}_{\mu \in (0,\widetilde{\mu})}$
  with some
  $\widetilde{\mu} \in \exRp$
  be a smoothing function of
  $\cost$.
  Let
  $\widebar{\bm{x}} \in \mathrm{int}(\dom{\cost})$
  and
  $(\bm{x}_{n})_{n=1}^{\infty} \subset \mathrm{int}(\dom{\cost})$
  satisfy
  $\lim_{n\to\infty}\bm{x}_{n} = \widebar{\bm{x}}$.
  Then, for
  $\mu_{n} (\in (0,\widetilde{\mu})) \searrow 0\ (n\in\mathbb{N})$,
  $(\nabla \cost^{\langle \mu_{n} \rangle}(\bm{x}_{n}))_{n=1}^{\infty}$ is bounded.
\end{lemma}

\begin{proof}
  Suppose, to the contrary, that
  $(\nabla \cost^{\langle \mu_{n} \rangle}(\bm{x}_{n}))_{n=1}^{\infty}$
  is unbounded.
  We can assume
  $\lim_{n\to\infty} \norm{\nabla \cost^{\langle \mu_{n} \rangle}(\bm{x}_{n})} = +\infty$
  and
  $\nabla \cost^{\langle \mu_{n} \rangle}(\bm{x}_{n}) \neq \bm{0}\ (n\in\mathbb{N})$
  (by passing through subsequence if necessary).
  Let
  $\bm{g}_{n} \coloneqq \nabla \cost^{\langle \mu_{n} \rangle}(\bm{x}_{n})/\norm{\nabla \cost^{\langle \mu_{n} \rangle}(\bm{x}_{n})}\in\mathcal{X}$.
  In the remainder of this paragraph, we show that there exists
  $(\alpha_{n})_{n=1}^{\infty} \subset \mathbb{R}_{++}$
  satisfying
  \begin{align}
     & (n\in\mathbb{N})\quad
    \bm{y}_{n} \coloneqq \bm{x}_{n} - \alpha_{n} \mu_{n}\bm{g}_{n} \in \mathrm{int}(\dom{\cost}),\label{eq:y_int_dom} \\
     & \alpha_{\rm low}\coloneqq \inf_{n\in\mathbb{N}} \alpha_{n} > 0,
    \ \mathrm{and}\
    \alpha_{\rm up}\coloneqq \sup_{n\in\mathbb{N}} \alpha_{n} < +\infty. \label{eq:alpha_bound}
  \end{align}
  From
  $\widebar{\bm{x}} \in \mathrm{int}(\dom{\cost})$,
  there exists
  $r \in \mathbb{R}_{++}$
  such that
  $B(\widebar{\bm{x}},r)\subset \mathrm{int}(\dom{\cost})$.
  From
  $\lim_{n\to\infty}\bm{x}_{n} = \widebar{\bm{x}}$,
  there exists
  $n_{0} \in \mathbb{N}$
  such that
  $\norm{\bm{x}_{n}-\widebar{\bm{x}}} < r/2\ (\forall n\geq n_{0})$.
  For
  $n < n_{0}$,
  we can choose
  a sufficiently small
  $\alpha_{n} \in \mathbb{R}_{++}$
  satisfying
  $\bm{y}_{n} = \bm{x}_{n} - \alpha_{n} \mu_{n}\bm{g}_{n} \in \mathrm{int}(\dom{\cost})$
  due to
  $\bm{x}_{n} \in \mathrm{int}(\dom{\cost})$.
  For
  $n \geq n_{0}$,
  by letting
  $\alpha_{n} \coloneqq \frac{r}{2\mu_{1}}$,
  we get
  $\bm{y}_{n} \in B(\widebar{\bm{x}}, r) \subset \mathrm{int}(\dom{\cost})$
  from
  $\norm{\bm{y}_{n}-\widebar{\bm{x}}} \leq \norm{\bm{x}_{n}-\widebar{\bm{x}}} + \alpha_{n}\mu_{n}\norm{\bm{g}_{n}} < \frac{r}{2} + \frac{r\mu_{n}}{2\mu_{1}} \leq \frac{r}{2} + \frac{r}{2} = r$
  due to
  $\bm{x}_{n} \in B(\widebar{\bm{x}}, r/2)$,
  $\norm{\bm{g}_{n}} = 1$
  and
  $\mu_{n} \leq \mu_{1}$.
  Consequently, we get
  $\bm{y}_{n} \in \mathrm{int}(\dom{\cost})\ (\forall n\in\mathbb{N})$,
  i.e.,~\eqref{eq:y_int_dom},
  and
  $\inf_{n\in\mathbb{N}}\alpha_{n} = \min\{\alpha_{1},\alpha_{2},\alpha_{3},\ldots,\alpha_{n_{0}-1}, \frac{r}{2\mu_{1}} \} > 0$
  and
  $\sup_{n\in\mathbb{N}}\alpha_{n} = \max\{\alpha_{1},\alpha_{2},\alpha_{3},\ldots,\alpha_{n_{0}-1}, \frac{r}{2\mu_{1}} \} < +\infty$,
  i.e.,~\eqref{eq:alpha_bound}.

  Since
  $\nabla \cost^{\langle \mu_{n} \rangle}$
  is $L_{\nabla \cost^{\langle \mu_{n} \rangle}}$-Lipschitz continuous
  over the convex set
  $\mathrm{int}(\dom{\cost})$
  (see Definition~\ref{definition:smoothing}~\ref{enum:consistency:gradient_Lipschitz} and Assumption~\ref{assumption:basic}~\ref{enum:problem:origin:cost}),
  the descent lemma~\cite[Lemma 5.7]{Beck17} yields
  \begin{equation}
    \thickmuskip=0.2\thickmuskip
    \medmuskip=0.2\medmuskip
    \thinmuskip=0.2\thinmuskip
    \arraycolsep=0.2\arraycolsep
    (n\in\mathbb{N}) \quad
    \cost^{\langle \mu_{n} \rangle}(\bm{y}_{n}) \leq \cost^{\langle \mu_{n} \rangle}(\bm{x}_{n}) + \inprod{\nabla \cost^{\langle \mu_{n} \rangle}(\bm{x}_{n})}{\bm{y}_{n}-\bm{x}_{n}} + \frac{L_{\nabla \cost^{\langle \mu_{n} \rangle}}}{2}\norm{\bm{x}_{n}-\bm{y}_{n}}^{2}. \label{eq:descent_inequality}
  \end{equation}
  By
  $\bm{y}_{n}-\bm{x}_{n} = -\alpha_{n}\mu_{n}\bm{g}_{n}$,
  $\bm{g}_{n} = \nabla \cost^{\langle \mu_{n} \rangle}(\bm{x}_{n})/\norm{\nabla \cost^{\langle \mu_{n} \rangle}(\bm{x}_{n})}$
  and
  $0<\alpha_{\rm low}\leq \alpha_{n} \leq \alpha_{\rm up}$
  in~\eqref{eq:alpha_bound},
  we have
  $\inprod{\nabla \cost^{\langle \mu_{n} \rangle}(\bm{x}_{n})}{\bm{y}_{n}-\bm{x}_{n}} = -\alpha_{n}\mu_{n}\norm{\nabla \cost^{\langle \mu_{n} \rangle}(\bm{x}_{n})} \leq -\alpha_{\rm low}\mu_{n}\norm{\nabla \cost^{\langle \mu_{n} \rangle}(\bm{x}_{n})}$
  and
  $\frac{L_{\nabla \cost^{\langle \mu_{n} \rangle}}}{2}\norm{\bm{x}_{n}-\bm{y}_{n}}^{2} \leq L_{\nabla \cost^{\langle \mu_{n} \rangle}}\norm{\bm{x}_{n}-\bm{y}_{n}}^{2}=
    L_{\nabla \cost^{\langle \mu_{n} \rangle}}\alpha_{n}^{2}\mu_{n}^{2} = \alpha_{n}^{2}\mu_{n}(\varpi_{1}\mu_{n} + \varpi_{2}) \leq \alpha_{\rm up}^{2}\mu_{n}(\varpi_{1}\mu_{1} + \varpi_{2})$,
  where we used
  $\mu_{n}L_{\nabla \cost^{\langle \mu_{n} \rangle}} = \varpi_{1}\mu_{n} + \varpi_{2} \leq \varpi_{1}\mu_{1} + \varpi_{2}$
  (see Definition~\ref{definition:smoothing}~\ref{enum:consistency:gradient_Lipschitz_constant}).
  Together with~\eqref{eq:descent_inequality},
  we get
  $\cost^{\langle \mu_{n} \rangle}(\bm{y}_{n}) \leq \cost^{\langle \mu_{n} \rangle}(\bm{x}_{n}) -\alpha_{\rm low}\mu_{n}\norm{\nabla \cost^{\langle \mu_{n} \rangle}(\bm{x}_{n})} + \alpha_{\rm up}^{2}\mu_{n}(\varpi_{1}\mu_{1} + \varpi_{2})$.
  By
  $\alpha_{\rm low} \mu_{n} > 0$,
  \begin{equation}
    (n\in\mathbb{N})\quad
    \norm{\nabla \cost^{\langle \mu_{n} \rangle}(\bm{x}_{n})}
    \leq \frac{\cost^{\langle \mu_{n} \rangle}(\bm{x}_{n})-\cost^{\langle \mu_{n} \rangle}(\bm{y}_{n})}{\alpha_{\rm low}\mu_{n}} + \frac{\alpha_{\rm up}^{2}(\varpi_{1}\mu_{1} + \varpi_{2})}{\alpha_{\rm low}} . \label{eq:smoothing_gradient_upper}
  \end{equation}

  In the following, we derive a finite upper bound of
  $\sup_{n\in\mathbb{N}}\norm{\nabla \cost^{\langle \mu_{n} \rangle}(\bm{x}_{n})}$,
  which contradicts
  $\lim_{n\to\infty} \norm{\nabla \cost^{\langle \mu_{n} \rangle}(\bm{x}_{n})} = +\infty$.
  To this end, it is enough to derive
  an upper bound, independent of
  $n$,
  of the first term in the RHS of~\eqref{eq:smoothing_gradient_upper}.

  We show that
  $\cost$
  is Lipschitz continuous over some compact convex set containing
  $\widebar{\bm{x}}$,
  $\bm{x}_{n}$
  and
  $\bm{y}_{n}\ (n\in\mathbb{N})$.
  By
  $\alpha_{n}\mu_{n}\bm{g}_{n} \to \bm{0}$
  and
  $\bm{x}_{n} \to \widebar{\bm{x}}$
  $(n\to\infty)$,
  we have
  $\lim_{n\to\infty} \bm{y}_{n} = \widebar{\bm{x}}$.
  Together with
  $\widebar{\bm{x}},\bm{x}_{n},\bm{y}_{n} \in \mathrm{int}(\dom{\cost})$
  (see~\eqref{eq:y_int_dom}),
  $\{\widebar{\bm{x}}\}\cup \bigcup_{n=1}^{\infty}\{\bm{x}_{n},\bm{y}_{n}\}\subset \mathrm{int}(\dom{\cost})$
  is a compact set in the Euclidean space
  $\mathcal{X}$.
  Let
  $K \coloneqq \Conv(\{\widebar{\bm{x}}\}\cup \bigcup_{n=1}^{\infty}\{\bm{x}_{n},\bm{y}_{n}\})$
  be the convex hull, i.e.,
  the smallest convex subset of
  $\mathcal{X}$
  containing
  $\{\widebar{\bm{x}}\}\cup \bigcup_{n=1}^{\infty}\{\bm{x}_{n},\bm{y}_{n}\}$.
  Then,
  a corollary of Carath\'eodory's theorem (see, e.g.,~\cite[Thm. 17.2]{Rockafellar70}) yields the compactness of
  $K$.
  Moreover,
  we have
  $K \subset \mathrm{int}(\dom{\cost})$,
  because
  $\mathrm{int}(\dom{\cost})$
  is convex and
  $K$
  is the convex hull of the subset
  $\{\widebar{\bm{x}}\}\cup \bigcup_{n=1}^{\infty}\{\bm{x}_{n},\bm{y}_{n}\}$
  of
  $\mathrm{int}(\dom{\cost})$.
  Since
  $\cost$
  is locally Lipschitz continuous at every
  $\bm{x} \in \mathrm{int}(\dom{\cost})$
  by Assumption~\ref{assumption:basic}~\ref{enum:problem:origin:cost},
  $\cost$
  is $L_{\cost}$-Lipschitz continuous with some
  $L_{\cost} > 0$
  over the compact convex set
  $K (\subset\mathrm{int}(\dom{\cost}))$
  by~\cite[Thm. 9.2]{Rockafellar-Wets98}.

  For
  $n\in\mathbb{N}$,
  $\cost^{\langle \mu_{n} \rangle}(\bm{x}_{n}) \leq \cost(\bm{x}_{n})$
  and
  $\cost(\bm{y}_{n}) \leq \cost^{\langle \mu_{n} \rangle}(\bm{y}_{n}) + \kappa \mu_{n}$
  follow from~\eqref{eq:uniform_bounded_smoothing_original}.
  By
  $\bm{x}_{n},\bm{y}_{n} \in K$,
  we obtain
  $\cost^{\langle \mu_{n} \rangle}(\bm{x}_{n}) - \cost^{\langle \mu_{n} \rangle}(\bm{y}_{n}) \leq \cost(\bm{x}_{n}) - \cost(\bm{y}_{n}) + \kappa \mu_{n} \leq L_{\cost}\norm{\bm{x}_{n}-\bm{y}_{n}} + \kappa \mu_{n} \overset{\eqref{eq:y_int_dom}}{=} L_{\cost}\alpha_{n}\mu_{n} + \kappa \mu_{n} \overset{\eqref{eq:alpha_bound}}{\leq} \mu_{n}(L_{\cost}\alpha_{\rm up}+\kappa)$.
  By combining~\eqref{eq:smoothing_gradient_upper},
  we get
  \begin{equation}
    (n\in\mathbb{N})\quad
    \norm{\nabla \cost^{\langle \mu_{n} \rangle}(\bm{x}_{n})}
    \leq \frac{L_{\cost}\alpha_{\rm up}+\kappa + \alpha_{\rm up}^{2} (\varpi_{1}\mu_{1} + \varpi_{2})}{\alpha_{\rm low}}
    < +\infty.
  \end{equation}
  This inequality contradicts
  $\lim_{n\to\infty} \norm{\nabla \cost^{\langle \mu_{n} \rangle}(\bm{x}_{n})} = +\infty$.
\end{proof}

\subsubsection{Asymptotic properties of $\mathcal{M}_{\gamma}^{\cost,\phi}$} \label{sec:stationarity:asymptotic:theorem}
We derive asymptotic properties of
$\mathcal{M}_{\gamma}^{\cost,\phi}$
in Theorem~\ref{theorem:asymptotic_approximation}, which is the goal of Section~\ref{sec:stationarity:asymptotic}.
A part of Theorem~\ref{theorem:asymptotic_approximation} requires Assumption~\ref{assumption:phi} on
$\phi$
below.
Its sufficient condition is the continuity of
$\phi$
over
$\dom{\phi}$~\cite[Prop. 8.7]{Rockafellar-Wets98}, while some discontinuous functions, e.g., $\ell_{0}$-pseudonorm, satisfy Assumption~\ref{assumption:phi} (see Example~\ref{example:phi}).

\begin{assumption}[Outer semicontinuity of $\Lsubdiff \phi$ over $\dom{\phi}$]
  \label{assumption:phi}
  For
  $\phi$
  in Assumption~\ref{assumption:basic},
  $\Lsubdiff \phi$
  is outer semicontinuous at every
  $\widebar{\bm{x}} \in \dom{\phi}$,
  i.e.,
    {\thickmuskip=2mu plus 1mu minus 1mu
      \medmuskip=1mu plus 1mu minus 1mu
      $\Limsup_{n\to\infty} \Lsubdiff \phi(\bm{x}_{n}) \subset \Lsubdiff \phi(\widebar{\bm{x}})$}
  holds for every convergent sequence
  $(\bm{x}_{n})_{n=1}^{\infty} \subset \mathcal{X}$
  to
  $\widebar{\bm{x}}$.
\end{assumption}

\begin{example}[On $\phi$ achieving Assumption~\ref{assumption:regular}~\ref{enum:assumption:regular:phi} and Assumption~\ref{assumption:phi}]
  \label{example:phi}
  Typical examples of
  $\phi$
  achieving
  Assumptions~\ref{assumption:basic}~\ref{enum:problem:origin:phi},~\ref{assumption:regular}~\ref{enum:assumption:regular:phi} and~\ref{assumption:phi}
  are (i) proper lower semicontinuous convex functions (see~\cite[Exm. 13.30]{Rockafellar-Wets98} for Assumption~\ref{assumption:regular}~\ref{enum:assumption:regular:phi}, and~\cite[Thm. 5.7 (a)]{Rockafellar-Wets98} with~\cite[Prop. 16.36]{Bauschke-Combettes17} for Assumption~\ref{assumption:phi}), (ii) prox-bounded differentiable functions with locally Lipschitz continuous gradient~\cite[Prop. 13.34]{Rockafellar-Wets98}, and
  (iii)
  the indicator function
  $\iota_{\constraint}$
  with a {\em prox-regular set}
  $\constraint\subset \mathcal{X}$~\cite[Thm. 1.3]{Poliquin-Rockafellar-Thibault00},
  where Assumption~\ref{assumption:phi} always holds under the continuity of
  $\phi$
  over
  $\dom{\phi}$
  by~\cite[Prop. 8.7]{Rockafellar-Wets98}.
  Here,
  $C$
  is said to be prox-regular if the {\em metric projection}
  $P_{\constraint}:\mathcal{X} \rightrightarrows\constraint: \widebar{\bm{x}}\mapsto \mathop{\mathrm{argmin}}_{\bm{x}\in\constraint}\norm{\widebar{\bm{x}}-\bm{x}}$
  onto
  $\constraint$
  is single-valued on some open superset of
  $\constraint$.
  Such prox-regular sets include, e.g., closed convex sets,
  $C^{2}$ embedded submanifolds of
  $\mathcal{X}$~\cite[Lemma 2.1]{Lewis-Malick08}, and {\em proximally smooth sets}\footnote{\label{foot:proximally_smooth}
    Proximally smooth sets include, e.g.,
    $\mathrm{Ob}^{+}(n,p) \coloneqq \mathrm{Ob}(n,p)\cap \mathbb{R}_{+}^{n\times p}$~\cite[Lemma 3 (ii)]{Hu-Deng-Wu-Li24}
    with the oblique manifold
    $\mathrm{Ob}(n,p) \coloneqq \{\bm{X} \in \mathbb{R}^{n\times p} \mid \mathrm{diag}(\bm{X}^{\TT}\bm{X}) = \bm{I}_{p}\}$,
    and a closed subset
    $\mathfrak{L}_{r,\sigma}\coloneqq \left\{\bm{X} \in \mathbb{R}^{n_{1}\times n_{2}} \mid 0 < \mathrm{Rank}(\bm{X}) \leq r,\ \sigma \leq \sigma_{j}(\bm{X})\ \mathrm{or}\ \sigma_{j}(\bm{X})=0 \ (j=1,2,\ldots, r)\right\}$
    of
    low-rank matrices
    $\{\bm{X}\in \mathbb{R}^{n_{1}\times n_{2}} \mid \mathrm{Rank}(\bm{X}) \leq r\}$
    with
    $r \leq \min\{n_{1},n_{2}\}$
    and
    $\sigma \in \mathbb{R}_{++}$~\cite[Thm. 5]{Balashov-Kamalov21}.
  }~\cite{Federer59,Clarke-Stern-Wolenski95,Balashov-Kamalov21,Hu-Deng-Wu-Li24,Davis-Drusvyatskiy-Shi25}.
  Even if
  $\phi$
  is discontinuous over
  $\dom{\phi}$,
  we have examples of
  $\phi$
  satisfying all these assumptions,
  e.g.,
  (iv) $\ell_{0}$-pseudonorm\footnote{
    The prox-regularity of $\ell_{0}$-pseudonorm is shown in~\cite[Exm. 4.1]{Liang-Fadili-Peyre16}.
    We can verify that $\ell_{0}$-pseudonorm
    $\norm{\widebar{\bm{x}}}_{0}\coloneqq \sum_{i=1}^{n}\xi([\widebar{\bm{x}}]_{i})$
    satisfies Assumption~\ref{assumption:phi} as follows, where
    $\xi(t) \coloneqq 1$
    if
    $t \in \mathbb{R}\setminus \{0\}$;
    $\xi(t)\coloneqq 0$
    if
    $t=0$.
    By~\cite[Prop. 10.5]{Rockafellar-Wets98}, we have
    $\Lsubdiff \norm{\widebar{\bm{x}}}_{0} = \Lsubdiff \xi([\widebar{\bm{x}}]_{1}) \times \Lsubdiff \xi([\widebar{\bm{x}}]_{2}) \times \cdots \times \Lsubdiff \xi([\widebar{\bm{x}}]_{n})\ (\widebar{\bm{x}} \in \mathbb{R}^{n})$.
    Since
    $\Lsubdiff \xi(t)= \{0\}$
    holds if
    $t \in \mathbb{R}\setminus\{0\}$;
    $\Lsubdiff \xi(t) = \mathbb{R}$
    holds if
    $t=0$,
    $\Lsubdiff \norm{\cdot}_{0}$
    is outer semicontinuous over
    $\mathbb{R}^{n}$.
    We can check the outer semicontinuity of
    $\Lsubdiff \mathrm{Rank}$
    in a similar way with~\cite[Thm. 7.1]{Lewis-Sendov05}.
  }~\cite[Exm. 4.1]{Liang-Fadili-Peyre16}
  and (v) the rank function~\cite[Exm. 4.2]{Liang-Fadili-Peyre16}.
\end{example}

Theorem~\ref{theorem:asymptotic_approximation} below is a key ingredient for our characterizations of stationarity conditions of
$\cost+\phi$
via asymptotic behaviors of
$\mathcal{M}_{\gamma}^{\cost,\phi}$
(see Section~\ref{sec:measure:stationarity}).

\begin{theorem}[Asymptotic properties of $\mathcal{M}_{\gamma_{n}}^{\cost,\phi}$ and $\mathcal{M}_{\gamma_{n}}^{\cost^{\langle \mu_{n} \rangle},\phi}$]
  \label{theorem:asymptotic_approximation}
  Let
  $\cost$
  and
  $\phi$
  satisfy Assumption~\ref{assumption:basic}.
  Let
  $(\bm{x}_{n})_{n=1}^{\infty} \subset \mathrm{int}(\dom{\cost})$
  converge to some
  $\widebar{\bm{x}} \in \mathrm{int}(\dom{\cost})$,
  and
  $(\gamma_{n})_{n=1}^{\infty} \subset (0,\gamma_{\phi})$
  converge to some
  $\widebar{\gamma} \in [0,\gamma_{\phi})$.
  Then, the following hold,
  where the second inequalities in~\eqref{eq:asymptotic_measure_positive} and~\eqref{eq:asymptotic_measure_zero} hold
  if
  $\{\cost^{\langle \mu \rangle}\}_{\mu\in (0,\widetilde{\mu})}$
  is a smoothing function of
  $\cost$
  and
  $\mu_{n}(\in (0,\widetilde{\mu})) \searrow 0\ (n\in\mathbb{N})$:
  \begin{enumerate}[label=(\alph*)]
    \item
          \label{enum:theorem:asymptotic_approximation:positive_gamma}
          Consider the case
          $\widebar{\gamma}\coloneqq \lim_{n\to\infty}\gamma_{n} \in (0,\gamma_{\phi})$.
          Then, we have
          \begin{equation}
            \left(\widebar{\bm{x}} \in \mathrm{int}(\dom{\cost})\right) \quad
            \mathcal{M}_{\widebar{\gamma}}^{\cost,\phi}(\widebar{\bm{x}}) \leq
            \begin{cases}
              \liminf\limits_{n\to\infty} \mathcal{M}_{\gamma_{n}}^{\cost,\phi}(\bm{x}_{n}); \\
              \liminf\limits_{n\to\infty} \mathcal{M}_{\gamma_{n}}^{\cost^{\langle \mu_{n} \rangle},\phi}(\bm{x}_{n}).
            \end{cases}
            \label{eq:asymptotic_measure_positive}
          \end{equation}
    \item
          \label{enum:theorem:asymptotic_approximation:zero_gamma}
          Consider the case
          $\widebar{\gamma} \coloneqq \lim_{n\to\infty}\gamma_{n}=0$.
          Under Assumption~\ref{assumption:phi} on
          $\phi$,
          we have
          \begin{equation}
            (\widebar{\bm{x}} \in \dom{\phi})\quad
            \dist(\bm{0},  \Lsubdiff \cost(\widebar{\bm{x}}) + \Lsubdiff\phi(\widebar{\bm{x}})) \leq
            \begin{cases}
              \liminf\limits_{n\to\infty} \mathcal{M}_{\gamma_{n}}^{\cost,\phi}(\bm{x}_{n}); \\
              \liminf\limits_{n\to\infty} \mathcal{M}_{\gamma_{n}}^{\cost^{\langle \mu_{n} \rangle},\phi}(\bm{x}_{n}).
            \end{cases}
            \label{eq:asymptotic_measure_zero}
          \end{equation}
  \end{enumerate}
\end{theorem}

In order to show Theorem~\ref{theorem:asymptotic_approximation} in a unified way, we consider the following two cases of a sequence
$(S_{n})_{n=1}^{\infty}$
of set-valued mappings
$S_{n}:\mathcal{X} \rightrightarrows \mathcal{X}$
as:
\begin{enumerate}[label=Case (\Roman*),leftmargin=*,align=left]
  \item
        \label{enum:claim:set_convergence:subdifferential}
        Let
        $S_{n}(\bm{x}) \coloneqq \Lsubdiff \cost(\bm{x})\ (n\in\mathbb{N}, \bm{x} \in \mathcal{X})$;
  \item
        \label{enum:claim:set_convergence:gradient}
        Let
        $\{\cost^{\langle \mu \rangle}\}_{\mu\in (0,\widetilde{\mu})}$
        with
        $\widetilde{\mu} \in \exRp$
        be a smoothing function of
        $\cost$,
        and
        $\mu_{n} ( \in (0,\widetilde{\mu})) \searrow 0$.
        For
        $n\in\mathbb{N}$,
        let
        $S_{n}(\bm{x})\coloneqq\{\nabla \cost^{\langle \mu_{n} \rangle}(\bm{x})\}\ (\bm{x}\in \mathrm{int}(\dom{\cost}))$
        and
        $S_{n}(\bm{x})\coloneqq \emptyset\ (\bm{x}\notin \mathrm{int}(\dom{\cost}))$.
\end{enumerate}
More precisely, we can express
$\mathcal{M}_{\gamma_{n}}^{\cost,\phi}(\bm{x}_{n})$
and
$\mathcal{M}_{\gamma_{n}}^{\cost^{\langle \mu_{n} \rangle},\phi}(\bm{x}_{n})$
in Theorem~\ref{theorem:asymptotic_approximation}
by
$\dist\left(\bm{0},\frac{\bm{x}_{n}- \prox{\gamma_{n}\phi}(\bm{x}_{n} - \gamma_{n} S_{n}(\bm{x}_{n}))}{\gamma_{n}}\right)$
with
$(S_{n})_{n=1}^{\infty}$
in~\ref{enum:claim:set_convergence:subdifferential} and~\ref{enum:claim:set_convergence:gradient}
in a unified way.
We first establish the following auxiliary claim, which shows that the two choices of
$(S_{n})_{n=1}^{\infty}$
above share the two properties required in the proof of Theorem~\ref{theorem:asymptotic_approximation}.
\begin{claim}
  \label{claim:S_n}
  Let
  $\cost$
  satisfy Assumption~\ref{assumption:basic}~\ref{enum:problem:origin:cost}.
  Let
  $(\bm{x}_{n})_{n=1}^{\infty} \subset \mathrm{int}(\dom{\cost})$
  converge to some
  $\widebar{\bm{x}} \in \mathrm{int}(\dom{\cost})$.
  Then,
  we have
  (i)
  $\Limsup_{n\to\infty} S_{n}(\bm{x}_{n}) \subset \Lsubdiff \cost(\widebar{\bm{x}})$;
  and
  (ii)
  $(S_{n}(\bm{x}_{n}))_{n=1}^{\infty}$
  is {\em eventually bounded}
  (see Definition~\ref{definition:bounded}~\ref{enum:definition:bounded:eventually_bounded}),
  if
  $S_{n}:\mathcal{X}\rightrightarrows\mathcal{X}\ (n\in\mathbb{N})$
  is given in~\ref{enum:claim:set_convergence:subdifferential} or~\ref{enum:claim:set_convergence:gradient}
  above.
\end{claim}
\begin{proof}[Proof of Claim~\ref{claim:S_n}]
  Consider~\ref{enum:claim:set_convergence:subdifferential}.
  By Fact~\ref{fact:subdifferential_nonempty},
  $\Lsubdiff\cost=S_{n} \ (n\in \mathbb{N})$
  is outer semicontinuous and locally bounded at
  $\widebar{\bm{x}} \in \mathrm{int}(\dom{\cost})$.
  Hence,
  by
  $\lim_{n\to\infty}\bm{x}_{n} = \widebar{\bm{x}}$,
  the outer semicontinuity of
  $\Lsubdiff \cost$
  yields
  $\Limsup_{n\to\infty} S_{n}(\bm{x}_{n}) = \Limsup_{n\to\infty}\Lsubdiff\cost(\bm{x}_{n}) \subset \Lsubdiff \cost(\widebar{\bm{x}})$,
  and the local boundedness of
  $\Lsubdiff \cost$
  implies that
  $(\Lsubdiff \cost(\bm{x}_{n}))_{n=1}^{\infty}=(S_{n}(\bm{x}_{n}))_{n=1}^{\infty}$
  is eventually bounded.
  For~\ref{enum:claim:set_convergence:gradient}, the gradient consistency of
  $\{\cost^{\langle \mu \rangle}\}_{\mu \in (0,\widetilde{\mu})}$
  in~\eqref{eq:consistency}
  gives
  $\Limsup_{n\to\infty} S_{n}(\bm{x}_{n}) = \Limsup_{n\to\infty}\nabla \cost^{\langle \mu_{n} \rangle}(\bm{x}_{n}) \subset \Lsubdiff \cost(\widebar{\bm{x}})$
  and Lemma~\ref{lemma:bounded_gradient} ensures the boundedness of
  $(\nabla\cost^{\langle \mu_{n}\rangle}(\bm{x}_{n}))_{n=1}^{\infty}$,
  i.e.,
  the eventual boundedness of
  $(S_{n}(\bm{x}_{n}))_{n=1}^{\infty}=(\{\nabla\cost^{\langle \mu_{n}\rangle}(\bm{x}_{n})\})_{n=1}^{\infty}$.
\end{proof}

Claim~\ref{claim:set_valued_elementary} below is the main technical step in the proof of Theorem~\ref{theorem:asymptotic_approximation}.
Since our proof of Claim~\ref{claim:set_valued_elementary} is somewhat lengthy, we first state this claim and use it to prove Theorem~\ref{theorem:asymptotic_approximation}.
The proof of Claim~\ref{claim:set_valued_elementary} is then given immediately afterward.

\begin{claim}
  \label{claim:set_valued_elementary}
  Let
  $\cost$
  and
  $\phi$
  satisfy Assumption~\ref{assumption:basic}.
  Suppose that
  $S_{n}:\mathcal{X} \rightrightarrows \mathcal{X}\ (n\in\mathbb{N})$
  satisfies the conditions~(i)
  $\Limsup_{n\to\infty} S_{n}(\bm{x}_{n}) \subset \Lsubdiff \cost(\widebar{\bm{x}})$
  and~(ii)
  the eventual boundedness of
  $(S_{n}(\bm{x}_{n}))_{n=1}^{\infty}$
  (see Claim~\ref{claim:S_n} with~\ref{enum:claim:set_convergence:subdifferential} and~\ref{enum:claim:set_convergence:gradient}).
  Let
  $(\bm{x}_{n})_{n=1}^{\infty} \subset \mathrm{int}(\dom{\cost})$
  converge to some
  $\widebar{\bm{x}} \in \mathrm{int}(\dom{\cost})$,
  and
  $(\gamma_{n})_{n=1}^{\infty} \subset (0,\gamma_{\phi})$
  converge to some
  $\widebar{\gamma} \in [0,\gamma_{\phi})$.
  Then, for
  $E_{n} \coloneqq \bm{x}_{n} - \gamma_{n} S_{n}(\bm{x}_{n}) \subset \mathcal{X} \ (n\in\mathbb{N})$,
  the following hold:
  \begin{enumerate}[label=(\alph*)]
    \item
          \label{enum:claim:set_convergence:gamma_positive}
          If
          $\widebar{\gamma} \in (0,\gamma_{\phi})$,
          then
          $\liminf_{n\to\infty} \dist\left(\bm{0},\frac{\bm{x}_{n}- \prox{\gamma_{n}\phi}(E_{n})}{\gamma_{n}}\right) \geq \mathcal{M}_{\widebar{\gamma}}^{\cost,\phi}(\widebar{\bm{x}})$
          holds.
    \item
          \label{enum:claim:set_convergence:gamma_zero}
          If
          $\widebar{\gamma} = 0$,
          $\widebar{\bm{x}} \in \dom{\phi} (\subset \mathrm{int}(\dom{\cost}))$,
          and
          $\phi$
          satisfies Assumption~\ref{assumption:phi},
          then
          $\liminf_{n\to\infty} \dist\left(\bm{0},\frac{\bm{x}_{n}- \prox{\gamma_{n}\phi}(E_{n})}{\gamma_{n}}\right) \geq \dist(\bm{0},\Lsubdiff \cost(\widebar{\bm{x}}) + \Lsubdiff \phi(\widebar{\bm{x}}))$
          holds.
  \end{enumerate}
\end{claim}

\begin{proof}[Proof of Theorem~\ref{theorem:asymptotic_approximation}]
  Recall that
  $(S_{n})_{n=1}^{\infty}$
  in~\ref{enum:claim:set_convergence:subdifferential} or~\ref{enum:claim:set_convergence:gradient}
  satisfies the conditions (i) and (ii) in Claim~\ref{claim:set_valued_elementary} by Claim~\ref{claim:S_n}.

  \ref{enum:theorem:asymptotic_approximation:positive_gamma}
  The first inequality in~\eqref{eq:asymptotic_measure_positive} is verified by Claim~\ref{claim:set_valued_elementary}~\ref{enum:claim:set_convergence:gamma_positive} under~\ref{enum:claim:set_convergence:subdifferential}.
  The second one is also verified by Claim~\ref{claim:set_valued_elementary}~\ref{enum:claim:set_convergence:gamma_positive} under~\ref{enum:claim:set_convergence:gradient}.

  \ref{enum:theorem:asymptotic_approximation:zero_gamma}
  In a similar way to~\ref{enum:theorem:asymptotic_approximation:positive_gamma}, we get the two inequalities in~\eqref{eq:asymptotic_measure_zero}
  by applying Claim~\ref{claim:set_valued_elementary}~\ref{enum:claim:set_convergence:gamma_zero} under~\ref{enum:claim:set_convergence:subdifferential}
  and~\ref{enum:claim:set_convergence:gradient} respectively.
\end{proof}

Here, we show Claim~\ref{claim:set_valued_elementary} by using facts on set-valued analysis in~\ref{sec:appendix:set_valued}.
\begin{proof}[Proof of Claim~\ref{claim:set_valued_elementary}]
  Since we will use the following claim~\eqref{eq:property_E} repeatedly in the proofs of Claim~\ref{claim:set_valued_elementary}~\ref{enum:claim:set_convergence:gamma_positive} and~\ref{enum:claim:set_convergence:gamma_zero},
  we begin with showing
  \begin{equation}
    \Limsup_{n\to\infty}
    E_{n} \subset E \coloneqq \widebar{\bm{x}} - \widebar{\gamma}\Lsubdiff \cost(\widebar{\bm{x}}) \subset \mathcal{X},
    \ \mathrm{and}\
    (E_{n})_{n=1}^{\infty}
    \mathrm{\ is\ eventually\ bounded}.
    \label{eq:property_E}
  \end{equation}

  Recall that
  $r_{\bm{x}}\coloneqq \sup_{n\in\mathbb{N}}\norm{\bm{x}_{n}} < +\infty$
  by
  $\lim_{n\to\infty}\bm{x}_{n} = \widebar{\bm{x}}$,
  $\widetilde{\gamma}\coloneqq \sup_{n\in\mathbb{N}}\gamma_{n} < +\infty$
  by
  $\lim_{n\to\infty} \gamma_{n} = \widebar{\gamma}$,
  and
  $(S_{n}(\bm{x}_{n}))_{n=1}^{\infty}$
  is eventually bounded by the condition~(ii), i.e., there exist
  $n_{0}\in \mathbb{N}$
  and
  $r_{S} \in \mathbb{R}_{++}$
  such that
  $\bigcup_{n\geq n_{0}} S_{n}(\bm{x}_{n}) \subset B(\bm{0}, r_{S})$.
  For every
  $n\geq n_{0}$
  and
  $\bm{v}_{n} \in E_{n}=\bm{x}_{n}-\gamma_{n}S_{n}(\bm{x}_{n})$,
  since there exists
  $\bm{u}_{n} \in S_{n}(\bm{x}_{n})$
  such that
  $\bm{v}_{n} = \bm{x}_{n} - \gamma_{n}\bm{u}_{n}$,
  we get
  $\norm{\bm{v}_{n}} \leq \norm{\bm{x}_{n}} + \gamma_{n}\norm{\bm{u}_{n}} < r_{\bm{x}} + \widetilde{\gamma}r_{S}$.
  Hence,
  $\bigcup_{n\geq n_{0}}E_{n} \subset B(\bm{0}, r_{\bm{x}} + \widetilde{\gamma}r_{S})$
  holds, i.e.,
  $(E_{n})_{n=1}^{\infty}$
  is eventually bounded.

  To establish
  $\Limsup_{n\to\infty} E_{n} \subset E$,
  let
  $\bm{v} \in \Limsup_{n\to\infty}E_{n}$.
  Then, there exist
  $\mathcal{N} \in \Ninfty$
  and
  $\bm{u}_{n}\in S_{n}(\bm{x}_{n})\ (n\in\mathcal{N})$
  such that
  $\bm{v} = \lim_{\mathcal{N}\ni n\to\infty}(\bm{x}_{n} - \gamma_{n}\bm{u}_{n})$.
  Since
  $(S_{n}(\bm{x}_{n}))_{n=1}^{\infty}$
  is eventually bounded,
  $(\bm{u}_{n})_{n\in\mathcal{N}}$
  is bounded.
  We can assume that
  $(\bm{u}_{n})_{n\in\mathcal{N}}$
  converges to some
  $\widebar{\bm{u}} \in \Limsup_{n\to\infty} S_{n}(\bm{x}_{n})$
  (by passing through further subsequence of
  $(\bm{u}_{n})_{n\in\mathcal{N}}$
    if necessary).
    By
  $\lim_{n\to\infty} \bm{x}_{n} = \widebar{\bm{x}}$
    and
  $\lim_{n\to\infty} \gamma_{n} = \widebar{\gamma}$,
    we get
  $\bm{v}=\lim_{\mathcal{N}\ni n\to\infty} (\bm{x}_{n} - \gamma_{n}\bm{u}_{n}) = \widebar{\bm{x}} - \widebar{\gamma} \widebar{\bm{u}} \in \widebar{\bm{x}} - \widebar{\gamma}\Limsup_{n\to\infty} S_{n}(\bm{x}_{n})$.
    Hence,
    the condition~(i)
  $\Limsup_{n\to\infty} S_{n}(\bm{x}_{n}) \subset \Lsubdiff \cost(\widebar{\bm{x}})$
    gives
  $\bm{v} \in \widebar{\bm{x}} - \widebar{\gamma}\Limsup_{n\to\infty} S_{n}(\bm{x}_{n}) \subset \widebar{\bm{x}} - \widebar{\gamma}\Lsubdiff \cost(\widebar{\bm{x}}) = E$,
    implying thus
  $\Limsup_{n\to\infty} E_{n} \subset E$.

    \ref{enum:claim:set_convergence:gamma_positive}
    Suppose
  $\widebar{\gamma} \in (0,\gamma_{\phi})$.
    We get
  $\Limsup_{n\to \infty} \prox{\gamma_{n}\phi}(E_{n}) \subset \prox{\widebar{\gamma}\phi}(E)$
    from
    Lemma~\ref{lemma:convergence_mapping_set}
    by letting
  $(\mathcal{S}, (\mathcal{S}_{n})_{n=1}^{\infty}, \mathcal{E},(\mathcal{E}_{n})_{n=1}^{\infty}) \coloneqq (\prox{\widebar{\gamma}\phi}, (\prox{\gamma_{n}\phi})_{n=1}^{\infty}, E, (E_{n})_{n=1}^{\infty})$
    with~\eqref{eq:property_E} and~\eqref{eq:gLimsup_prox_positive_gamma} in Fact~\ref{fact:prox_set_mapping}~\ref{enum:fact:prox_set_mapping:eventually_locall_bounded}.
    Then, we get the inclusion
  $\Limsup_{n\to\infty}\left(\bm{x}_{n}- \prox{\gamma_{n}\phi}(E_{n})\right) = \widebar{\bm{x}}- \Limsup_{n\to\infty} \prox{\gamma_{n}\phi}(E_{n}) \subset \widebar{\bm{x}}- \prox{\widebar{\gamma}\phi}(E)$,
    where the equality follows from
  $\lim_{n\to\infty}\bm{x}_{n} = \widebar{\bm{x}}$.
    By this inclusion, we get
    \begin{align}
       & \liminf_{n\to\infty} \gamma_{n}\dist\left(\bm{0},\frac{\bm{x}_{n}-\prox{\gamma_{n}\phi}(E_{n})}{\gamma_{n}}\right)
      =
      \liminf_{n\to\infty}\dist(\bm{0}, \bm{x}_{n}-\prox{\gamma_{n}\phi}(E_{n})) \\
       & \overset{\rm Fact~\ref{fact:outer}}{=}\dist\left(\bm{0},\Limsup_{n\to\infty}\left(\bm{x}_{n}-  \prox{\gamma_{n}\phi}(E_{n})\right)\right)
      \geq \dist\left(\bm{0},\widebar{\bm{x}}- \prox{\widebar{\gamma}\phi}(E)\right)                                                                                      \nonumber \\
       & = \widebar{\gamma} \dist\left(\bm{0},\frac{\widebar{\bm{x}}-\prox{\widebar{\gamma}\phi}(E)}{\widebar{\gamma}}\right)
      = \widebar{\gamma}\mathcal{M}_{\widebar{\gamma}}^{\cost,\phi}(\widebar{\bm{x}}).  \label{eq:asymptotic_measure_positive_gamma}
    \end{align}
    Recall
  $(\gamma_{n})_{n=1}^{\infty} \subset \mathbb{R}_{++}$,
  $\lim_{n\to\infty}\gamma_{n}=\widebar{\gamma} > 0$
    and
  $\dist\left(\bm{0},\frac{\bm{x}_{n}-\prox{\gamma_{n}\phi}(E_{n})}{\gamma_{n}}\right) \geq 0\ (n\in\mathbb{N})$.
    Then,
    we get\footnote{
  $\widebar{\gamma}\liminf_{n\to\infty}a_{n} = \liminf_{n\to\infty} (\gamma_{n}a_{n})$
    holds
    with
  $a_{n} \coloneqq \dist\left(\bm{0},\frac{\bm{x}_{n}-\prox{\gamma_{n}\phi}(E_{n})}{\gamma_{n}}\right) \geq 0$
    as follows.
    For any
  $\epsilon \in(0,\widebar{\gamma})$,
    there exists
  $n_{\epsilon} \in \mathbb{N}$
    such that
  $0 < \widebar{\gamma}-\epsilon < \gamma_{n} < \widebar{\gamma}+\epsilon\ (\forall n\geq n_{\epsilon})$.
    Then, we have
  $0 \leq (\widebar{\gamma}-\epsilon)a_{n} \leq \gamma_{n}a_{n} \leq (\widebar{\gamma}+\epsilon)a_{n}\ (\forall n\geq n_{\epsilon})$.
    By taking the limit infimum, we get
  $(\widebar{\gamma}-\epsilon)\liminf_{n\to\infty}a_{n} \leq \liminf_{n\to\infty}(\gamma_{n}a_{n}) \leq (\widebar{\gamma}+\epsilon) \liminf_{n\to\infty}a_{n}$.
    Since
  $\epsilon \in (0,\widebar{\gamma})$
    is chosen arbitrarily, we have
  $\widebar{\gamma}\liminf_{n\to\infty}a_{n} = \liminf_{n\to\infty} (\gamma_{n}a_{n})$.
    }
  $\liminf\limits_{n\to\infty} \gamma_{n}\dist\left(\bm{0},\frac{\bm{x}_{n}-\prox{\gamma_{n}\phi}(E_{n})}{\gamma_{n}}\right) = \widebar{\gamma} \liminf\limits_{n\to\infty}\dist\left(\bm{0},\frac{\bm{x}_{n}-\prox{\gamma_{n}\phi}(E_{n})}{\gamma_{n}}\right)$.
    Hence, we deduce
  $\widebar{\gamma} \liminf_{n\to\infty}\dist\left(\bm{0},\frac{\bm{x}_{n}-\prox{\gamma_{n}\phi}(E_{n})}{\gamma_{n}}\right) \geq \widebar{\gamma}\mathcal{M}_{\widebar{\gamma}}^{\cost,\phi}(\widebar{\bm{x}})$
    by~\eqref{eq:asymptotic_measure_positive_gamma}.
    By dividing both sides by
  $\widebar{\gamma} > 0$,
    we get the desired inequality.

    \ref{enum:claim:set_convergence:gamma_zero}
    Suppose
  $\widebar{\gamma} = 0$,
  $\widebar{\bm{x}} \in \dom{\phi}$,
    and Assumption~\ref{assumption:phi}.
    It suffices to show
    \begin{equation}
      \Limsup_{n\to\infty} \left(\frac{\bm{x}_{n}- \prox{\gamma_{n}\phi}(E_{n})}{\gamma_{n}}\right) \subset \Lsubdiff \cost(\widebar{\bm{x}}) + \Lsubdiff \phi(\widebar{\bm{x}}) \label{eq:outer_limit_GM}
    \end{equation}
    with
  $E_{n} =\bm{x}_{n} - \gamma_{n}
  S_{n}(\bm{x}_{n})$
    because we have
  $\liminf\limits_{n\to\infty} \dist\left(\bm{0},\frac{\bm{x}_{n}- \prox{\gamma_{n}\phi}(E_{n})}{\gamma_{n}}\right) \overset{\rm Fact~\ref{fact:outer}}{=}$
  $\dist\!
  \left(\bm{0}, \Limsup\limits_{n\to\infty}\left(\frac{\bm{x}_{n}- \prox{\gamma_{n}\phi}(E_{n})}{\gamma_{n}}\right)\right) \overset{\eqref{eq:outer_limit_GM}}{\geq} \dist(\bm{0}, \Lsubdiff \cost(\widebar{\bm{x}}) + \Lsubdiff\phi(\widebar{\bm{x}}))$.

    In what follows, we show~\eqref{eq:outer_limit_GM}.
    For every
  $n\in\mathbb{N}$,
  $\bm{v}_{n} \in S_{n}(\bm{x}_{n})$,
    and
  $\bm{p}_{n}\in \prox{\gamma_{n}\phi}(\bm{x}_{n}-\gamma_{n}\bm{v}_{n})$,
    we have
  $\bm{w}_{n}\coloneqq\frac{\bm{x}_{n} - \gamma_{n}\bm{v}_{n} - \bm{p}_{n}}{\gamma_{n}} \in \Lsubdiff \phi(\bm{p}_{n})$
    from Fact~\ref{fact:prox_set_mapping}~\ref{enum:fact:prox_set_mapping:nonempty_resolvent}.
    We get
    \begin{align}
       & \frac{\bm{x}_{n}- \prox{\gamma_{n}\phi}(E_{n})}{\gamma_{n}}
      = \left\{\frac{\bm{x}_{n}-\bm{p}_{n}}{\gamma_{n}}\mid \bm{v}_{n} \in S_{n}(\bm{x}_{n}),\ \bm{p}_{n}\in \prox{\gamma_{n}\phi}(\bm{x}_{n}-\gamma_{n}\bm{v}_{n}) \right\} \\
       & = \left\{\bm{v}_{n} + \frac{\bm{x}_{n} - \gamma_{n}\bm{v}_{n}-\bm{p}_{n}}{\gamma_{n}}\mid \bm{v}_{n} \in S_{n}(\bm{x}_{n}),\ \bm{p}_{n}\in \prox{\gamma_{n}\phi}(\bm{x}_{n}-\gamma_{n}\bm{v}_{n}) \right\} \nonumber \\
       & \subset \left\{\bm{v}_{n} + \bm{w}_{n} \mid \bm{v}_{n} \in S_{n}(\bm{x}_{n}),\ \bm{p}_{n}\in \prox{\gamma_{n}\phi}(\bm{x}_{n}-\gamma_{n}\bm{v}_{n}),\ \bm{w}_{n} \in \Lsubdiff \phi(\bm{p}_{n})\right\}                              \nonumber \\
       & = \left\{\bm{v}_{n} + \bm{w}_{n}\mid \bm{v}_{n} \in S_{n}(\bm{x}_{n}),\ \bm{w}_{n} \in \Lsubdiff \phi(\prox{\gamma_{n}\phi}(\bm{x}_{n}-\gamma_{n}\bm{v}_{n}))\right\} \nonumber \\
       & \subset S_{n}(\bm{x}_{n}) + \Lsubdiff \phi
      \left(\prox{\gamma_{n}\phi}\left(\bm{x}_{n} - \gamma_{n}
      S_{n}(\bm{x}_{n})\right)\right)
      = S_{n}(\bm{x}_{n}) + \Lsubdiff \phi
      \left(\prox{\gamma_{n}\phi}(E_{n})\right).
      \label{eq:measure_inequality}
    \end{align}
    By taking the outer limit of~\eqref{eq:measure_inequality}, we have
    \begin{align}
      \Limsup_{n\to\infty} \left(\frac{\bm{x}_{n}- \prox{\gamma_{n}\phi}(E_{n})}{\gamma_{n}}\right)
       & \subset \Limsup_{n\to\infty}\left(S_{n}(\bm{x}_{n}) + \Lsubdiff \phi\left( \prox{\gamma_{n}\phi}(E_{n})\right)\right) \\
       & \subset \Limsup_{n\to\infty}
      S_{n}(\bm{x}_{n}) + \Limsup_{n\to\infty}\Lsubdiff \phi \left(\prox{\gamma_{n}\phi}(E_{n})\right) \nonumber \\
       & \subset \Lsubdiff \cost(\widebar{\bm{x}}) + \Limsup_{n\to\infty}\Lsubdiff \phi\left( \prox{\gamma_{n}\phi}(E_{n})\right),
      \label{eq:sum_set_valued_outer_limit}
    \end{align}
    where the second inclusion holds by Lemma~\ref{lemma:sum_rule_outer_limit} with
    the eventual boundedness of
  $(S_{n}(\bm{x}_{n}))_{n=1}^{\infty}$
    (see the condition~(ii)),
    and the last one holds by the condition (i):
  $\Limsup_{n\to\infty}S_{n}(\bm{x}_{n}) \subset \Lsubdiff \cost(\widebar{\bm{x}})$.

    In view of~\eqref{eq:sum_set_valued_outer_limit},
    it remains to show
  $\Limsup_{n\to\infty} \Lsubdiff \phi(\prox{\gamma_{n}\phi}(E_{n})) \subset \Lsubdiff \phi(\widebar{\bm{x}})$
    in order to establish the inclusion~\eqref{eq:outer_limit_GM}.
    To this end, the next paragraph shows
    \begin{equation}
      \Limsup_{n\to\infty}\prox{\gamma_{n}\phi}(E_{n}) \subset  \{\widebar{\bm{x}}\},
      \ \mathrm{and}\
      (\prox{\gamma_{n}\phi}(E_{n}))_{n=1}^{\infty}
      \mathrm{\ is\ eventually\ bounded}.
      \label{eq:property_prox_E}
    \end{equation}
    Indeed, the inclusion
  $\Limsup_{n\to\infty} \Lsubdiff \phi(\prox{\gamma_{n}\phi}(E_{n})) \subset \Lsubdiff \phi(\widebar{\bm{x}})$
    can be verified by applying Lemma~\ref{lemma:convergence_mapping_set} with
  $(\mathcal{S}, (\mathcal{S}_{n})_{n=1}^{\infty}, \mathcal{E},(\mathcal{E}_{n})_{n=1}^{\infty}) \coloneqq (\Lsubdiff \phi, (\Lsubdiff \phi)_{n=1}^{\infty}, \{\widebar{\bm{x}}\},\allowbreak (\prox{\gamma_{n}\phi}(E_{n}))_{n=1}^{\infty})$
    because the assumptions in Lemma~\ref{lemma:convergence_mapping_set} are satisfied by~\eqref{eq:property_prox_E}
    and the outer semicontinuity of
  $\Lsubdiff \phi$
    at
  $\widebar{\bm{x}} \in \dom{\phi}$
    (see Assumption~\ref{assumption:phi}).

    The inclusion
  $\Limsup_{n\to\infty}\prox{\gamma_{n}\phi}(E_{n}) \subset \{\widebar{\bm{x}}\}(= \widebar{\bm{x}} - \widebar{\gamma}\Lsubdiff \cost(\widebar{\bm{x}})=E)$
    in~\eqref{eq:property_prox_E}
    follows from
    Lemma~\ref{lemma:convergence_mapping_set}
    by letting
  $(\mathcal{S}, (\mathcal{S}_{n})_{n=1}^{\infty}, \mathcal{E},(\mathcal{E}_{n})_{n=1}^{\infty}) \coloneqq (\Id, (\prox{\gamma_{n}\phi})_{n=1}^{\infty}, \{\widebar{\bm{x}}\}, (E_{n})_{n=1}^{\infty})$
    with~\eqref{eq:property_E} and \eqref{eq:gLimsup_prox_zero_gamma} in Fact~\ref{fact:prox_set_mapping}~\ref{enum:fact:prox_set_mapping:eventually_locall_bounded}.
    Suppose, to the contrary, that
  $(\prox{\gamma_{n}\phi}(E_{n}))_{n=1}^{\infty}$
    is not eventually bounded, i.e.,
    there exist
  $\mathcal{N} \in \Ninfty$
    and
  $(\bm{p}_{n})_{n=1}^{\infty}\subset \mathcal{X}$
    satisfying
  $\bm{p}_{n} \in \prox{\gamma_{n}\phi}(\bm{u}_{n}) \subset \prox{\gamma_{n}\phi}(E_{n})\ (n\in\mathcal{N})$
    with some
  $\bm{u}_{n} \in E_{n}$
    and
  $\lim_{\mathcal{N} \ni n\to\infty} \norm{\bm{p}_{n}} = +\infty$.
    Since
  $(E_{n})_{n=1}^{\infty}$
    is eventually bounded by
  $\eqref{eq:property_E}$,
  $(\bm{u}_{n})_{n \in \mathcal{N}}$
    is a bounded sequence.
    From
  $\Limsup_{n\to\infty}E_{n} \subset E = \{\widebar{\bm{x}}\}$
    (see~\eqref{eq:property_E}),
  $\widebar{\bm{x}}$
    is the unique cluster point of a bounded sequence
  $(\bm{u}_{n})_{n\in\mathcal{N}}$,
    from which
  $\lim_{\mathcal{N}\ni n\to\infty} \bm{u}_{n} = \widebar{\bm{x}}$
    is deduced (see, e.g.,~\cite[Lemma 1.14]{Bauschke-Combettes17}).
    Together with
  $\bm{p}_{n} \in \prox{\gamma_{n}\phi}(\bm{u}_{n})\ (n\in\mathcal{N})$
    and
  $\lim_{n\to\infty} \gamma_{n} = \widebar{\gamma} = 0$,
    Fact~\ref{fact:prox_set_mapping}~\ref{enum:fact:prox_set_mapping:eventually_locall_bounded}
    ensures
  $\lim_{\mathcal{N}\ni n\to\infty}\bm{p}_{n}= \lim_{\mathcal{N}\ni n\to\infty}\bm{u}_{n}=\widebar{\bm{x}}$,
    which contradicts
  $\lim_{\mathcal{N} \ni n\to\infty} \norm{\bm{p}_{n}} = +\infty$.
\end{proof}

\begin{remark}[Theoretical obstacle in the absence of the convexity of $\phi$]
  \label{remark:obstacle}
  Under the convexity of
  $\phi$
  and a smoothing function
  $\{\cost^{\langle \mu \rangle}\}_{\mu \in (0,\widetilde{\mu})}$
  in Example~\ref{example:cost},
  our recent paper~\cite{Kume-Yamada25B} establishes the following inequality with an arbitrarily chosen
  $\widetilde{\gamma} \in \mathbb{R}_{++}$:
  \begin{equation}
    \left(\widebar{\bm{x}} \in \mathcal{X} \right) \quad
    \mathcal{M}_{\widetilde{\gamma}}^{\cost,\phi}(\widebar{\bm{x}}) \leq
    \liminf\limits_{n\to\infty} \mathcal{M}_{\widetilde{\gamma}}^{\cost^{\langle \mu_{n} \rangle},\phi}(\bm{x}_{n})
    \label{eq:asymptotic_measure_positive_2}
  \end{equation}
  for
  $\bm{x}_{n} \to \widebar{\bm{x}} \in \mathcal{X}$
  and
  $\mu_{n} \searrow 0$
  as
  $n\to\infty$.
  By combining~\eqref{eq:asymptotic_measure_positive_2} with a monotonicity property (see, e.g.,~\cite[Thm. 10.9]{Beck17} or~\cite[Lemma 6]{Liu-Xia24}) under the convexity of
  $\phi$:
  \begin{equation}
    (0 < \gamma \leq \widetilde{\gamma},\ \mu \in (0,\widetilde{\mu}),\ \bm{x}\in \mathcal{X}) \quad
    \mathcal{M}_{\widetilde{\gamma}}^{\cost^{\langle \mu \rangle},\phi}(\bm{x})
    \leq
    \mathcal{M}_{\gamma}^{\cost^{\langle \mu \rangle},\phi}(\bm{x}), \label{eq:stepsize_monotonicity}
  \end{equation}
  we get an inequality with variable stepsizes analogous to Theorem~\ref{theorem:asymptotic_approximation}~\ref{enum:theorem:asymptotic_approximation:positive_gamma} as
  $\mathcal{M}_{\widetilde{\gamma}}^{\cost,\phi}(\widebar{\bm{x}}) \leq \liminf_{n\to\infty} \mathcal{M}_{\gamma_{n}}^{\cost^{\langle \mu_{n} \rangle},\phi}(\bm{x}_{n})\
    (\widebar{\bm{x}} \in \mathcal{X} ) $
  for
  $\gamma_{n} \in (0,\widetilde{\gamma}]\ (n\in\mathbb{N})$,
  and
  $\bm{x}_{n} \to \widebar{\bm{x}} \in \mathcal{X}$,
  and
  $\mu_{n} \searrow 0$
  as
  $n\to\infty$.
  However,
  $\phi$
  is not assumed to be convex in this paper.
  Hence,
  the above monotonicity property in~\eqref{eq:stepsize_monotonicity} is not available in general.
  Instead,
  our proof of Theorem~\ref{theorem:asymptotic_approximation} directly shows the inequalities in~\eqref{eq:asymptotic_measure_positive} and~\eqref{eq:asymptotic_measure_zero} with variable stepsizes by using tools from set-valued analysis.
\end{remark}

\subsection{Stationarity\,characterization\,via\,asymptotic\,behavior\,of\,$\mathcal{M}_{\gamma}^{\cost,\phi}$}\,\label{sec:measure:stationarity}
By exploiting asymptotic properties of the measure
$\mathcal{M}_{\gamma}^{\cost,\phi}$
in Theorem~\ref{theorem:asymptotic_approximation}, we characterize stationary points of
$\cost+\phi$
via asymptotic behaviors of the measure under Assumptions~\ref{assumption:regular} and~\ref{assumption:smooth} respectively.
We begin with a simple case in Assumption~\ref{assumption:smooth}.

\begin{theorem}[Characterization of {\rm P}-stationarity via asymptotic behaviors $\mathcal{M}_{\gamma_{n}}^{\cost,\phi}$]
  \label{theorem:P_stationarity}
  Let
  $\cost$
  and
  $\phi$
  satisfy Assumptions~\ref{assumption:basic} and~\ref{assumption:smooth}.
  Choose
  $\widetilde{\gamma} \in (0,\gamma_{\phi})$
  arbitrarily.
  For
  $\bm{x}^{\star} \in \dom{\phi}$,
  the following conditions are equivalent:
  \begin{enumerate}[label=(\roman*)]
    \item
          \label{enum:theorem:P_stationarity:stationarity}
          $\bm{x}^{\star}$
          is a {\rm P}-stationary point of
          $\cost+\phi$.
    \item
          \label{enum:theorem:P_stationarity:positive}
          There exist
          $(\bm{x}_{n})_{n=1}^{\infty} \subset \mathrm{int}(\dom{\cost})$
          with
          $\lim_{n\to\infty} \bm{x}_{n} = \bm{x}^{\star}$
          and
          $(\gamma_{n})_{n=1}^{\infty} \subset [\gamma_{\rm low},\widetilde{\gamma}]$
          with some
          $\gamma_{\rm low} \in (0,\widetilde{\gamma}]$
          such that
          $\liminf_{n\to\infty}\mathcal{M}_{\gamma_{n}}^{\cost,\phi}(\bm{x}_{n}) = 0$.
  \end{enumerate}
\end{theorem}
\begin{proof}
  \ref{enum:theorem:P_stationarity:stationarity} $\Rightarrow$ \ref{enum:theorem:P_stationarity:positive}:
  Since
  $\bm{x}^{\star}$
  is a {\rm P}-stationary point of
  $\cost+\phi$,
  there exists
  $\gamma_{\bm{x}^{\star}} \in (0,\gamma_{\phi})$
  with
  $\mathcal{M}_{\gamma}^{\cost,\phi}(\bm{x}^{\star}) = 0\ (\forall \gamma  \in (0,\gamma_{\bm{x}^{\star}}])$
  by Proposition~\ref{proposition:criticality_stationarity}~\ref{enum:proposition:criticality_stationarity:smooth}.
  For
  $\bm{x}_{n}\coloneqq \bm{x}^{\star}\in \dom{\phi}\subset \mathrm{int}(\dom{\cost})$
  and
  $\gamma_{n}\coloneqq \min\{\gamma_{\bm{x}^{\star}}, \widetilde{\gamma}\}(>0)\ (n\in \mathbb{N})$,
  we get
  $\liminf_{n\to\infty}\mathcal{M}_{\gamma_{n}}^{\cost,\phi}(\bm{x}_{n}) = 0$.

  \ref{enum:theorem:P_stationarity:positive} $\Rightarrow$ \ref{enum:theorem:P_stationarity:stationarity}:
  From
  $\liminf_{n\to\infty}\mathcal{M}_{\gamma_{n}}^{\cost,\phi}(\bm{x}_{n})=0$
  and
  the boundedness of
  $(\gamma_{n})_{n=1}^{\infty} (\subset [\gamma_{\rm low},\widetilde{\gamma}])$,
  there exists an index set
  $\mathcal{N} \in \Ninfty$
  satisfying
  $\lim_{\mathcal{N}\ni n\to\infty} \mathcal{M}_{\gamma_{n}}^{\cost,\phi}(\bm{x}_{n}) = 0$
  and
  $\lim_{\mathcal{N}\ni n\to\infty} \gamma_{n} = \widebar{\gamma}$
  with some cluster point
  $\widebar{\gamma}\in[\gamma_{\rm low},\widetilde{\gamma}]$
  of
  $(\gamma_{n})_{n=1}^{\infty}$.
  Then, the first inequality in~\eqref{eq:asymptotic_measure_positive} in Theorem~\ref{theorem:asymptotic_approximation}~\ref{enum:theorem:asymptotic_approximation:positive_gamma} yields
  $ \mathcal{M}_{\widebar{\gamma}}^{\cost,\phi}(\bm{x}^{\star})= 0$.
  By Proposition~\ref{proposition:criticality_stationarity}~\ref{enum:proposition:criticality_stationarity:smooth},
  $\bm{x}^{\star}$
  is a P-stationary point of
  $\cost+\phi$.
\end{proof}

We present a counterpart of Theorem~\ref{theorem:P_stationarity} under Assumptions~\ref{assumption:regular} and~\ref{assumption:phi}.

\begin{theorem}[Characterization of {\rm F}-stationarity via asymptotic behaviors $\mathcal{M}_{\gamma_{n}}^{\cost,\phi}$]
  \label{theorem:F_stationarity}
  Let
  $\cost$
  and
  $\phi$
  satisfy Assumptions~\ref{assumption:basic},~\ref{assumption:regular}, and~\ref{assumption:phi}.
  Choose
  $\widetilde{\gamma} \in (0,\gamma_{\phi})$
  arbitrarily.
  For
  $\bm{x}^{\star} \in \dom{\phi}$,
  the following conditions are equivalent:
  \begin{enumerate}[label=(\roman*)]
    \item
          \label{enum:theorem:F_stationarity:stationarity}
          $\bm{x}^{\star}$
          is an {\rm F}-stationary point of
          $\cost+\phi$.
    \item
          \label{enum:theorem:F_stationarity:positive}
          There exist
          $(\bm{x}_{n})_{n=1}^{\infty} \subset \mathrm{int}(\dom{\cost})$
          with
          $\lim_{n\to\infty} \bm{x}_{n} = \bm{x}^{\star}$
          and
          $(\gamma_{n})_{n=1}^{\infty} \subset (0,\widetilde{\gamma}]$
          such that
          $\liminf_{n\to\infty}\mathcal{M}_{\gamma_{n}}^{\cost,\phi}(\bm{x}_{n}) = 0$.
  \end{enumerate}
\end{theorem}
\begin{proof}
  \ref{enum:theorem:F_stationarity:stationarity} $\Rightarrow$ \ref{enum:theorem:F_stationarity:positive}:
  We can verify it verbatim as in the proof of \ref{enum:theorem:P_stationarity:stationarity} $\Rightarrow$ \ref{enum:theorem:P_stationarity:positive} in Theorem~\ref{theorem:P_stationarity} by using Proposition~\ref{proposition:criticality_stationarity}~\ref{enum:proposition:criticality_stationarity:regular} instead of Proposition~\ref{proposition:criticality_stationarity}~\ref{enum:proposition:criticality_stationarity:smooth}.

  \ref{enum:theorem:F_stationarity:positive}
  $\Rightarrow$
  \ref{enum:theorem:F_stationarity:stationarity}:
  From
  $\liminf_{n\to\infty}\mathcal{M}_{\gamma_{n}}^{\cost,\phi}(\bm{x}_{n})=0$
  and
  the boundedness of
  $(\gamma_{n})_{n=1}^{\infty} (\subset (0,\widetilde{\gamma}])$,
  there exists an index set
  $\mathcal{N} \in \Ninfty$
  such that
  $\lim_{\mathcal{N}\ni n\to\infty}\mathcal{M}_{\gamma_{n}}^{\cost,\phi}(\bm{x}_{n}) = 0$
  and
  $\lim_{\mathcal{N}\ni n\to\infty} \gamma_{n} = \widebar{\gamma}$
  with some cluster point
  $\widebar{\gamma} \in [0,\widetilde{\gamma}]$
  of
  $(\gamma_{n})_{n=1}^{\infty}$.
  We consider two cases
  $\widebar{\gamma} \in (0,\widetilde{\gamma}]$
  and
  $\widebar{\gamma} = 0$.
  For the case
  $\widebar{\gamma} \in (0,\widetilde{\gamma}]$,
  the first inequality in~\eqref{eq:asymptotic_measure_positive} in Theorem~\ref{theorem:asymptotic_approximation}~\ref{enum:theorem:asymptotic_approximation:positive_gamma} yields
  $\mathcal{M}_{\widebar{\gamma}}^{\cost,\phi}(\bm{x}^{\star}) = 0$,
  from which
  $\bm{x}^{\star}$
  is an {\rm F}-stationary point of
  $\cost+\phi$
  by Proposition~\ref{proposition:criticality_stationarity}~\ref{enum:proposition:criticality_stationarity:regular}.
  For the other case
  $\widebar{\gamma} = 0$,
  the first inequality in~\eqref{eq:asymptotic_measure_zero} in Theorem~\ref{theorem:asymptotic_approximation}~\ref{enum:theorem:asymptotic_approximation:zero_gamma} yields
  $\dist(\bm{0}, \Lsubdiff \cost(\bm{x}^{\star}) + \Lsubdiff\phi(\bm{x}^{\star}))\overset{\eqref{eq:sum_subdifferential}}{=}\dist(\bm{0}, \Fsubdiff(\cost+\phi)(\bm{x}^{\star}))=0$.
  Since
  $\Fsubdiff (\cost+\phi)(\bm{x}^{\star})$
  is closed by~\cite[Thm. 8.6]{Rockafellar-Wets98},
  $\bm{x}^{\star}$
  is an {\rm F}-stationary point of
  $\cost+\phi$.
\end{proof}

Here, we also present a smoothing-based variant of Theorem~\ref{theorem:F_stationarity}.

\begin{theorem}[Characterization of {\rm F}-stationarity via asymptotic behaviors $\mathcal{M}_{\gamma_{n}}^{\cost^{\langle \mu_{n} \rangle},\phi}$ with smoothing function]
  \label{theorem:characterization_stationarity_asymptotic}
  Let
  $\cost$
  and
  $\phi$
  satisfy Assumptions~\ref{assumption:basic},~\ref{assumption:regular}, and~\ref{assumption:phi}.
  Let
  $\{\cost^{\langle \mu \rangle}\}_{\mu \in (0,\widetilde{\mu})}$
  with some
  $\widetilde{\mu} \in \exRp$
  be a smoothing function of
  $\cost$.
  Choose
  $\widetilde{\gamma} \in (0,\gamma_{\phi})$
  arbitrarily.
  For
  $\bm{x}^{\star} \in \dom{\phi}$,
  the following conditions are equivalent:
  \begin{enumerate}[label=(\roman*)]
    \item
          \label{enum:corollary:characterization_stationarity_asymptotic:stationary}
          $\bm{x}^{\star}$
          is an {\rm F}-stationary point of
          $\cost+\phi$.
    \item
          \label{enum:corollary:characterization_stationarity_asymptotic:positive}
          There exist
          $(\bm{x}_{n})_{n=1}^{\infty} \subset \mathrm{int}(\dom{\cost})$
          with
          $\lim_{n\to\infty} \bm{x}_{n} = \bm{x}^{\star}$,
          $\mu_{n}(\in (0,\widetilde{\mu})) \searrow 0$
          and
          $(\gamma_{n})_{n=1}^{\infty} \subset (0,\widetilde{\gamma}]$
          such that
          $\liminf_{n\to\infty}\mathcal{M}_{\gamma_{n}}^{\cost^{\langle \mu_{n} \rangle},\phi}(\bm{x}_{n}) = 0$.
  \end{enumerate}
\end{theorem}
\begin{proof}
  \ref{enum:corollary:characterization_stationarity_asymptotic:positive}
  $\Rightarrow$
  \ref{enum:corollary:characterization_stationarity_asymptotic:stationary}:
  This implication can be verified
  as in  the proof of
  \ref{enum:theorem:F_stationarity:positive}
  $\Rightarrow$
  \ref{enum:theorem:F_stationarity:stationarity} in Theorem~\ref{theorem:F_stationarity} by replacing
  $\mathcal{M}_{\gamma_{n}}^{\cost,\phi}(\bm{x}_{n})$
  with
  $\mathcal{M}_{\gamma_{n}}^{\cost^{\langle \mu_{n} \rangle},\phi}(\bm{x}_{n})$
  and
  the first inequalities in~\eqref{eq:asymptotic_measure_positive} and~\eqref{eq:asymptotic_measure_zero} with the corresponding second inequalities respectively.

  \ref{enum:corollary:characterization_stationarity_asymptotic:stationary} $\Rightarrow$
  \ref{enum:corollary:characterization_stationarity_asymptotic:positive}:
  Since
  $\bm{x}^{\star} \in \dom{\phi}$
  is an {\rm F}-stationary point of
  $\cost+\phi$,
  there exists
  $\widebar{\bm{v}} \in \Lsubdiff \phi(\bm{x}^{\star})$
  such that
  $-\widebar{\bm{v}} \in \Lsubdiff \cost(\bm{x}^{\star})$
  by~\eqref{eq:sum_subdifferential}.
  Then,
  Fact~\ref{fact:prox_set_mapping}~\ref{enum:fact:prox_set_mapping:prox_single} in~\ref{sec:appendix:set_valued} with
  the prox-regularity of
  $\phi$
  at
  $\bm{x}^{\star}$
  for
  $\widebar{\bm{v}}$
  ensures the existence of
  $\gamma_{\bm{x}^{\star}} \in (0,\gamma_{\phi})$
  such that
  \begin{align}
    (\forall \gamma \in (0,\gamma_{\bm{x}^{\star}}],\ \exists \tau_{\gamma} \in \mathbb{R}_{++}) & \quad
    \prox{\gamma\phi}
    \ \mathrm{is\ single\mathchar`-valued\ and} \\
                                                                                                 & \hspace{3em} \mathrm{Lipschitz\ continuous\ over\ }
    B(\bm{x}^{\star}+\gamma \widebar{\bm{v}}, \tau_{\gamma}); \label{eq:prox_single_Lipschitz} \\
    (\forall \gamma \in (0,\gamma_{\bm{x}^{\star}}])                                             & \quad \prox{\gamma \phi}(\bm{x}^{\star} + \gamma \widebar{\bm{v}}) = \{\bm{x}^{\star}\}.
    \label{eq:prox_fixed}
  \end{align}

  From~\eqref{eq:consistency} with
  $-\widebar{\bm{v}} \in \Lsubdiff \cost(\bm{x}^{\star})$,
  there exist
  $(\bm{\xi}_{k})_{k=1}^{\infty} \subset \mathrm{int}(\dom{\cost})$
  and
  $(\nu_{k})_{k=1}^{\infty} \subset (0,\widetilde{\mu})$
  with
  $\nu_{k}\searrow 0$
  satisfying
  \begin{equation}
    \lim_{k\to\infty} \bm{\xi}_{k} = \bm{x}^{\star}
    \,\mathrm{and}\,
    \lim_{k\to\infty}\nabla \cost^{\langle\nu_{k}\rangle}(\bm{\xi}_{k}) = -\widebar{\bm{v}}.
    \label{eq:lim_subsequence}
  \end{equation}
  Let
  $\widebar{\gamma} \coloneqq \min\{\gamma_{\bm{x}^{\star}}, \widetilde{\gamma}\}$,
  where
  we will set
  $\gamma_{n}\coloneqq \widebar{\gamma}\ (n\in\mathbb{N})$
  later.
  By~\eqref{eq:lim_subsequence},
  we can define a function
  $m:\mathbb{N}\cup\{0\} \to \mathbb{N}\cup\{0\}$
  with
  $m(0)\coloneqq 0$
  satisfying recursively
  \begin{equation}
    m(n) \coloneqq \min \left\{k \in \mathbb{N} \mid k > m(n-1),\ \norm{\bm{\xi}_{k}-\bm{x}^{\star}} < \frac{\tau_{\widebar{\gamma}}}{2},\ \norm{\nabla \cost^{\langle\nu_{k}\rangle}(\bm{\xi}_{k})+\widebar{\bm{v}}} < \frac{\tau_{\widebar{\gamma}}}{2\widebar{\gamma}}\right\}
  \end{equation}
  for
  $n\in \mathbb{N}$,
  where
  $\tau_{\widebar{\gamma}} \in \mathbb{R}_{++}$
  is well-defined due to
  $\widebar{\gamma} \in (0,\gamma_{\bm{x}^{\star}}]$
  (see~\eqref{eq:prox_single_Lipschitz}).
  By~\eqref{eq:lim_subsequence},
  $m$
  is a monotonically increasing function such that
  $\lim_{n\to\infty}m(n) = +\infty$.
  With
  $m$,
  we consider subsequences
  $(\bm{x}_{n})_{n=1}^{\infty}\coloneqq (\bm{\xi}_{m(n)})_{n=1}^{\infty}$
  of
  $(\bm{\xi}_{k})_{k=1}^{\infty}$,
  and
  $(\mu_{n})_{n=1}^{\infty} \coloneqq (\nu_{m(n)})_{n=1}^{\infty}$
  of
  $(\nu_{k})_{k=1}^{\infty}$,
  and let
  $\bm{g}_{n} \coloneqq \nabla \cost^{\langle\nu_{m(n)}\rangle}(\bm{\xi}_{m(n)})=\nabla \cost^{\langle \mu_{n} \rangle}(\bm{x}_{n})$.

  For
  $n\in\mathbb{N}$,
  since
  $\norm{\bm{x}_{n}-\bm{x}^{\star}} < \frac{\tau_{\widebar{\gamma}}}{2}$
  and
  $\norm{\bm{g}_{n}+\widebar{\bm{v}}} < \frac{\tau_{\widebar{\gamma}}}{2\widebar{\gamma}}$
  hold, we get
  $\bm{x}_{n}-\widebar{\gamma}\bm{g}_{n} \in B(\bm{x}^{\star} + \widebar{\gamma}\widebar{\bm{v}}, \tau_{\widebar{\gamma}})$
  from
  $\norm{(\bm{x}_{n}-\widebar{\gamma}\bm{g}_{n}) - (\bm{x}^{\star} + \widebar{\gamma}\widebar{\bm{v}})} \leq \norm{\bm{x}_{n} - \bm{x}^{\star}} + \widebar{\gamma} \norm{\bm{g}_{n}+\widebar{\bm{v}}} < \frac{\tau_{\widebar{\gamma}}}{2} + \frac{\tau_{\widebar{\gamma}}}{2}=\tau_{\widebar{\gamma}}$.
  By~\eqref{eq:prox_single_Lipschitz}
  with
  $\widebar{\gamma} \in (0,\gamma_{\bm{x}^{\star}}]$,
  $\prox{\widebar{\gamma}\phi}(\bm{x}_{n}-\widebar{\gamma}\bm{g}_{n})$
  is singleton for every
  $n\in\mathbb{N}$,
  and
  $\prox{\widebar{\gamma}\phi}$
  is
  $L_{\prox{\widebar{\gamma}\phi}}$-Lipschitz continuous with some
  $L_{\prox{\widebar{\gamma}\phi}}>0$
  over
  $B(\bm{x}^{\star}+\widebar{\gamma}\widebar{\bm{v}}, \tau_{\widebar{\gamma}})$.
  Hence, we have
    {
      \thickmuskip=0.2\thickmuskip
      \medmuskip=0.2\medmuskip
      \thinmuskip=0.2\thinmuskip
      \arraycolsep=0.2\arraycolsep
      \begin{align}
         & \mathcal{M}_{\widebar{\gamma}}^{\cost^{\langle \mu_{n}\rangle},\phi}(\bm{x}_{n})
        = \norm{\frac{\bm{x}_{n}-\prox{\widebar{\gamma}\phi}(\bm{x}_{n}-\widebar{\gamma}\bm{g}_{n})}{\widebar{\gamma}}}
        \leq
        \norm{\frac{\bm{x}_{n}-\bm{x}^{\star}}{\widebar{\gamma}}} + \norm{\frac{\bm{x}^{\star}-\prox{\widebar{\gamma}\phi}(\bm{x}_{n}-\widebar{\gamma}\bm{g}_{n})}{\widebar{\gamma}}}                                                                    \hspace{-2em} \\
         & =
        \norm{\frac{\bm{x}_{n}-\bm{x}^{\star}}{\widebar{\gamma}}} + \norm{\frac{\prox{\widebar{\gamma}\phi}(\bm{x}^{\star}+\widebar{\gamma}\widebar{\bm{v}})-\prox{\widebar{\gamma}\phi}(\bm{x}_{n}-\widebar{\gamma}\bm{g}_{n})}{\widebar{\gamma}}} \quad (\because \eqref{eq:prox_fixed}) \\
         & \leq
        \norm{\frac{\bm{x}_{n}-\bm{x}^{\star}}{\widebar{\gamma}}} + L_{\prox{\widebar{\gamma}\phi}}\norm{\frac{(\bm{x}^{\star}+\widebar{\gamma}\widebar{\bm{v}}) - (\bm{x}_{n}-\widebar{\gamma}\bm{g}_{n})}{\widebar{\gamma}}}
        \leq
        (1+L_{\prox{\widebar{\gamma}\phi}})\norm{\frac{\bm{x}_{n}-\bm{x}^{\star}}{\widebar{\gamma}}} + L_{\prox{\widebar{\gamma}\phi}}\norm{\bm{g}_{n} + \widebar{\bm{v}}}.
      \end{align}
    }%

  Recall
  $\lim_{n\to\infty}\norm{\bm{x}_{n}-\bm{x}^{\star}} = 0$
  and
  $\lim_{n\to\infty}\norm{\bm{g}_{n}+\widebar{\bm{v}}} = 0$
  by~\eqref{eq:lim_subsequence} with
  $\bm{x}_{n}=\bm{\xi}_{m(n)}$
  and
  $\bm{g}_{n} = \nabla \cost^{\langle \nu_{m(n)} \rangle}(\bm{\xi}_{m(n)})$.
  Hence,
  by letting
  $\gamma_{n} \coloneqq \widebar{\gamma}\ (n\in\mathbb{N})$,
  $\liminf_{n\to\infty}\mathcal{M}_{\gamma_{n}}^{\cost^{\langle \mu_{n} \rangle},\phi}(\bm{x}_{n}) = 0$
  holds.
\end{proof}

\section{Proximal stationarity of cluster points generated by proximal gradient methods under Assumption~\ref{assumption:smooth}}
\label{sec:application}\label{sec:application:existing}
\label{sec:application:existing:PGM}
In this section, we revisit convergence analyses of proximal gradient methods (PGM), e.g.,~\cite{Attouch-Bolte-Svaiter13,Li-Lin15,Beck17,Themelis-Stella-Patrinos18,Kanzow-Mehlitz22,DeMarchi23,Jia-Kanwoz-Mehlitz23,Kanzow-Lehmann25,Yagishita-Ito25A,Yagishita-Ito25B}, for the problem~\eqref{eq:problem} under Assumption~\ref{assumption:smooth}, where~\cite{Attouch-Bolte-Svaiter13,Li-Lin15,Beck17,Themelis-Stella-Patrinos18,Yagishita-Ito25B} assume additionally that
$\nabla \cost$
is globally Lipschitz continuous.
The PGM updates a sequence
$(\bm{x}_{n})_{n=1}^{\infty} \subset \dom{\phi}$
with an appropriate stepsize
$\gamma_{n} > 0$
(see Remark~\ref{remark:stepsize_selection} below)
as
\begin{equation}
  (n\in\mathbb{N}) \quad \bm{x}_{n+1} \in \prox{\gamma_{n}\phi}(\bm{x}_{n}-\gamma_{n}\nabla \cost(\bm{x}_{n})). \label{eq:PGM}
\end{equation}
To the best of the authors' knowledge,
conventional convergence analyses of PGM
under Assumption~\ref{assumption:smooth}
guarantee at most that every cluster point
$\widebar{\bm{x}} \in \dom{\phi}$
of
$(\bm{x}_{n})_{n=1}^{\infty}$
is an {\rm L}-stationary point~\cite{Kanzow-Mehlitz22,DeMarchi23,Jia-Kanwoz-Mehlitz23,Kanzow-Lehmann25}, or an {\rm F}-stationary point~\cite{Yagishita-Ito25A},
of
$\cost + \phi$.

Recently,
\citeauthor{Olikier-Waldspurger25}~\cite[Sect. 9]{Olikier-Waldspurger25} raised the following open question:
\begin{quote}
  If
  $\widebar{\bm{x}} \in \mathcal{X}$
  is a cluster point of
  a sequence
  $(\bm{x}_{n})_{n=1}^{\infty}\subset \dom{\phi}$
  generated by PGM in~\eqref{eq:PGM} under Assumption~\ref{assumption:smooth},
  is
  $\widebar{\bm{x}}$
  a {\rm P}-stationary point of
  $\cost+\phi$ ?
  (Note: the {\rm P}-stationarity is stronger than the {\rm L}-stationarity and {\rm F}-stationarity by~\eqref{eq:hierarchical_optimality}).
\end{quote}
By virtue of results in Section~\ref{sec:stationarity}, we provide an affirmative answer to this open question.

Typical convergence analysis of PGM, e.g.,~\cite{Li-Lin15,Kanzow-Mehlitz22,DeMarchi23,Jia-Kanwoz-Mehlitz23,Kanzow-Lehmann25,Yagishita-Ito25B},
proceeds as follows:
\begin{enumerate}[label=Step (\arabic*),leftmargin=*,align=left]
  \item
        choose a cluster point
        $\widebar{\bm{x}} \in \mathcal{X}$
        of
        $(\bm{x}_{n})_{n=1}^{\infty} \subset \dom{\phi}$
        generated by PGM in~\eqref{eq:PGM}, and choose an index set
        $\mathcal{N} \in \Ninfty$
        such that
        $\lim_{\mathcal{N}\ni n \to \infty}\bm{x}_{n}=\widebar{\bm{x}}$;
  \item
        \label{enum:PGM_step2}
        show that
        $\widebar{\bm{x}} \in \dom{\phi}$,
        $\lim_{\mathcal{N}\ni n\to\infty}\norm{\bm{x}_{n+1}-\bm{x}_{n}} = 0$,
        $(\gamma_{n})_{n\in\mathcal{N}} \subset [\gamma_{\rm low},\gamma_{\rm up}] (\subset (0,\gamma_{\phi}))$
        with some positive constants
        $\gamma_{\rm low}$
        and
        $\gamma_{\rm up}$
        (see, e.g.,~\cite[Lemma 4.3 (i), (iii) and Cor. 4.5]{DeMarchi23}; see also Remark~\ref{remark:stepsize_selection} for stepsizes);
  \item
        show that the limit point
        $\widebar{\bm{x}}$
        of
        $(\bm{x}_{n})_{n\in\mathcal{N}}$
        is an  {\rm L}-stationary point of
        $\cost + \phi$
        by using results in~\ref{enum:PGM_step2} (see, e.g.,~\cite[Thm. 4.6]{DeMarchi23}).
\end{enumerate}

\begin{remark}[Stepsize selections]
  \label{remark:stepsize_selection}
  In order to ensure the condition
  $(\gamma_{n})_{n\in\mathcal{N}} \subset [\gamma_{\rm low},\gamma_{\rm up}] (\subset (0,\gamma_{\phi}))$
  in~\ref{enum:PGM_step2} above,
  we have mainly three stepsize selections:
  (I) the so-called (monotone or nonmonotone) linesearch algorithm is used to obtain
  $\gamma_{n}$
  satisfying certain Armijo-type conditions, e.g.,~\cite{Li-Lin15,Beck17,Kanzow-Mehlitz22,DeMarchi23,Jia-Kanwoz-Mehlitz23,Kanzow-Lehmann25,Yagishita-Ito25A};
  (II)
  $\gamma_{n}$
  is set as the reciprocal of the Lipschitz constant of
  $\nabla \cost$, e.g.,~\cite{Attouch-Bolte-Svaiter13,Li-Lin15,Beck17};
  (III) the linesearch-free algorithm~\cite{Yagishita-Ito25B} is employed.
  We note that (II) and (III) require the global Lipschitz continuity of
  $\nabla \cost$,
  which is a stronger condition than Assumption~\ref{assumption:smooth}.
\end{remark}
By combining Theorem~\ref{theorem:P_stationarity} with the intermediate results in~\ref{enum:PGM_step2} of typical convergence analyses for PGM, we guarantee that
$\widebar{\bm{x}}$
is actually a {\rm P}-stationary point of
$\cost+\phi$
as follows.

\begin{corollary}[PGM accumulates at {\rm P}-stationary points]
  \label{corollary:refine_PGM}
  Let
  $(\bm{x}_{n})_{n=1}^{\infty} \subset \dom{\phi}$
  be generated by PGM in~\eqref{eq:PGM} under
  Assumptions~\ref{assumption:basic} and~\ref{assumption:smooth}.
  For every cluster point
  $\widebar{\bm{x}} \in \mathcal{X}$
  of
  $(\bm{x}_{n})_{n=1}^{\infty} \subset \dom{\phi}$,
  $\widebar{\bm{x}}$
  is a {\rm P}-stationary point of
  $\cost+\phi$
  in~\eqref{eq:problem}
  if all conditions in~\ref{enum:PGM_step2} hold with an index set
  $\mathcal{N}\in \Ninfty$
  satisfying
  $\lim_{\mathcal{N}\ni n\to\infty}\bm{x}_{n}=\widebar{\bm{x}}$,
  i.e.,
  $\widebar{\bm{x}} \in \dom{\phi}$,
  $\lim_{\mathcal{N}\ni n\to\infty}\norm{\bm{x}_{n+1}-\bm{x}_{n}} = 0$
  and
  $(\gamma_{n})_{n\in\mathcal{N}} \subset [\gamma_{\rm low},\gamma_{\rm up}]$
  $(\exists \gamma_{\rm low}, \gamma_{\rm up} \in (0,\gamma_{\phi}))$.
\end{corollary}
\begin{proof}
  We have
  $\norm{\bm{x}_{n+1}-\bm{x}_{n}} = \gamma_{n}\norm{(\bm{x}_{n+1}-\bm{x}_{n})/\gamma_{n}} \geq \gamma_{\rm low}\norm{(\bm{x}_{n+1}-\bm{x}_{n})/\gamma_{n}} \geq \gamma_{\rm low}\mathcal{M}_{\gamma_{n}}^{\cost,\phi}(\bm{x}_{n})\geq 0\ (n\in\mathcal{N})$
  from~\eqref{eq:measure_limiting} and~\eqref{eq:PGM}.
  The condition
  $\lim_{\mathcal{N}\ni n\to\infty}\norm{\bm{x}_{n+1}-\bm{x}_{n}} = 0$
  ensures
  $\lim_{\mathcal{N}\ni n\to\infty}\mathcal{M}_{\gamma_{n}}^{\cost,\phi}(\bm{x}_{n}) = 0$,
  from which
  $\widebar{\bm{x}}$
  is a {\rm P}-stationary point of
  $\cost+\phi$
  by Theorem~\ref{theorem:P_stationarity} together with
  $\lim_{\mathcal{N}\ni n\to\infty} \bm{x}_{n}=\widebar{\bm{x}} \in \dom{\phi}$
  and
  $(\gamma_{n})_{n\in \mathcal{N}} \subset [\gamma_{\rm low},\gamma_{\rm up}] \subset (0,\gamma_{\phi})$.
\end{proof}

Corollary~\ref{corollary:refine_PGM} streamlines a convergence analysis of PGM in~\eqref{eq:PGM}, and leads to
the stronger conclusion that every cluster point of
$(\bm{x}_{n})_{n=1}^{\infty}$
is a
  {\rm P}-stationary point.
We finally note that another open question in~\cite[Sect. 9]{Olikier-Waldspurger25} still remains:
whether
every cluster point of
$(\bm{x}_{n})_{n=1}^{\infty}$
generated by PGM in~\eqref{eq:PGM} is an {\rm F}-stationary point of
$\cost+\phi$
or not
under a weaker condition, namely the continuous differentiability of
$\cost$,
than Assumption~\ref{assumption:smooth}.

\section{Proximal variable smoothing algorithm with nonmonotone linesearch under Assumption~\ref{assumption:regular}} \label{sec:algorithm}
In this section, we propose a proximal variable smoothing algorithm with a nonmonotone linesearch for approximating iteratively an {\rm F}-stationary point of
$\cost+\phi$
under Assumptions~\ref{assumption:basic},~\ref{assumption:regular} and~\ref{assumption:phi}, and the availability of a smoothing function
$\{ \cost^{\langle \mu \rangle} \}_{\mu \in (0,\widetilde{\mu})}$
of
$\cost$.
By Theorem~\ref{theorem:characterization_stationarity_asymptotic},
the goal of finding an {\rm F}-stationary point
$\bm{x}^{\star} \in \mathcal{X}$
of
$\cost + \phi$
is reduced to the following problem:
\begin{equation}
  \mathrm{find\ a\ convergent\ sequence}\
  (\bm{x}_{n})_{n=1}^{\infty}\subset \mathrm{int}(\dom{\cost})\
  \mathrm{such\ that}\
  \liminf\limits_{n\to\infty}\mathcal{M}_{\gamma_{n}}^{\cost_{n},\phi}(\bm{x}_{n}) = 0
  \label{eq:Moreau_optimality}
\end{equation}
with some
$(\gamma_{n})_{n=1}^{\infty} \subset (0,\widetilde{\gamma}]$
and with some
$\widetilde{\gamma} \in (0,\gamma_{\phi})$,
where we use
$\cost_{n}\coloneqq \cost^{\langle \mu_{n} \rangle}$
with
$\mu_{n}\searrow 0$
for a notational simplicity in this section.

The proposed algorithm illustrated in Algorithm~\ref{alg:proposed} is designed to find a sequence
$(\bm{x}_{n})_{n=1}^{\infty}$
satisfying
$\liminf_{n\to\infty}\mathcal{M}_{\gamma_{n}}^{\cost_{n},\phi}(\bm{x}_{n}) = 0$.
More precisely, Algorithm~\ref{alg:proposed} is inspired by proximal gradient updates for time-varying function
$\cost_{n} + \phi=\cost^{\langle \mu_{n} \rangle}+\phi$
(see line~\ref{line:projection} in Algorithm~\ref{alg:proposed}),
where smoothing parameters
$(\mu_{n})_{n=1}^{\infty} \subset (0,\widetilde{\mu})$
are chosen to satisfy the conditions introduced in~\cite[Eq. (26)]{Kume-Yamada24A}:
\begin{equation}
  \textstyle
  \mathrm{(i)}\
  \lim_{n\to\infty} \mu_{n} = 0, \quad
  \mathrm{(ii)}\
  \sum_{n=1}^{\infty} \mu_{n} = +\infty, \quad
  \mathrm{(iii)}\
  (\forall n \in \mathbb{N}) \quad \mu_{n+1}\leq \mu_{n}. \label{eq:smoothing_parameter}
\end{equation}
For example,
$(\mu_{n})_{n=1}^{\infty} \coloneqq (\tau n^{-1/\alpha})$
with
$\tau \in (0,\widetilde{\mu})$
and
$\alpha \geq 1$
satisfies the condition~\eqref{eq:smoothing_parameter} (see, e.g.,~\cite[Exm. 4.6]{Kume-Yamada24A}).

In line~\ref{line:projection} of Algorithm~\ref{alg:proposed}, although we require choosing one element from the nonempty set
$\prox{\rho^{l}\gamma_{\rm init}\phi}(\bm{x}_{n} - \rho^{l}\gamma_{\rm init} \nabla \cost_{n}(\bm{x}_{n}))$,
this step is available for the so-called {\em prox-friendly} functions
$\phi$, e.g.,
$\ell_{1}$-norm, $\ell_{0}$-pseudonorm, rank function, and the indicator function
of
$\mathfrak{L}_{r,\sigma}$
(see footnote\footref{foot:proximally_smooth}) in
Example~\ref{example:phi} (see~\cite{Bauschke-Combettes17,Chierchia-Chouzenoux-Combettes-Pesquet,Beck17,Balashov-Kamalov21}).
Here, prox-friendly in this paper means that
at least one point
$\bm{p} \in \prox{\gamma \phi}(\widebar{\bm{x}})$
of the output of the set-valued proximity operator can be computed for every
$\widebar{\bm{x}} \in \mathcal{X}$
and
$\gamma \in (0,\gamma_{\phi})$.

\begin{algorithm}[t]
  \caption{Proximal variable smoothing with nonmonotone linesearch under Assumption~\ref{assumption:regular}}
  \label{alg:proposed}
  {
    \begin{algorithmic}[1]
      \Require
      $\bm{x}_{1}\in \dom{\phi}$,
      $c\in(0,2^{-1})$,
      $\rho \in (0,1)$,
      $\gamma_{\rm init} \in (0,\gamma_{\phi})$,
      $\ratio  \in (0,1)$,
      $\{\cost^{\langle \mu \rangle}\}_{\mu \in (0,\widetilde{\mu})}$: smoothing function of
      $\cost$ (see Definition~\ref{definition:smoothing})
      \For{$n=1,2,\ldots$}
      \State
      Set
      $\cost_{n} \coloneqq \cost^{\langle \mu_{n} \rangle}$
      with
      $\mu_{n} \in (0,\widetilde{\mu})$
      satisfying~\eqref{eq:smoothing_parameter}.
      \State
      Choose
      $\theta_{n} \in [\ratio ,1]$
      arbitrarily, and
      set
      $\NMV_{n}\in\mathbb{R}$,
      with a constant
      $\kappa\in\mathbb{R}_{++}$
      depending on
      $\{\cost^{\langle \mu \rangle}\}_{\mu \in (0,\widetilde{\mu})}$
      (see Definition~\ref{definition:smoothing}~\ref{enum:consistency:uniform}), as
      \begin{equation}
        \NMV_{n} \coloneqq
        \begin{cases}
          (\cost_{n}+\phi)(\bm{x}_{n}),                                                                                        & (n=1); \\
          \theta_{n}(\cost_{n}+\phi)(\bm{x}_{n}) + (1-\theta_{n})\left(\NMV_{n-1}+\kappa\left(\mu_{n-1}-\mu_{n}\right)\right), & (n>1).
        \end{cases}
        \label{eq:NMV}
      \end{equation}
      \label{line:p_n}
      \For{$l=0,1,2,\ldots$}\label{line:backtracking_start}
      \Comment{Backtracking step (see also Lemma~\ref{lemma:stepsize})}
      \State
      Pick
      $\bm{p} \in \prox{\rho^{l}\gamma_{\rm init}\phi}(\bm{x}_{n} - \rho^{l}\gamma_{\rm init} \nabla \cost_{n}(\bm{x}_{n})) \subset \dom{\phi}$ \label{line:projection}
      \If{$(\cost_{n}+\phi)(\bm{p}) \leq \NMV_{n} - c\rho^{l}\gamma_{\rm init}\norm{\frac{\bm{x}_{n}-\bm{p}}{\rho^{l}\gamma_{\rm init}}}^{2}$ holds} \label{line:Armijo}
      \State
      Set
      $(\bm{x}_{n+1},\gamma_{n}) \leftarrow (\bm{p},\rho^{l}\gamma_{\rm init})$, and \textbf{break}
      \EndIf
      \EndFor \label{line:backtracking_end}
      \EndFor
    \end{algorithmic}
  }
\end{algorithm}

At the
$n$-th
iteration of Algorithm~\ref{alg:proposed},
we choose
$\mu_{n} \in (0,\widetilde{\mu})$
satisfying~\eqref{eq:smoothing_parameter},
and set
$\cost_{n}\coloneqq \cost^{\langle \mu_{n} \rangle}$.
Then,
we estimate
$\bm{x}_{n+1} \in \dom{\phi}$
and an appropriate stepsize
$\gamma_{n} \in (0,\gamma_{\phi})$
such that
$(\bm{x},\gamma) \coloneqq (\bm{x}_{n+1},\gamma_{n})$
satisfies the following (nonmonotone) Armijo-type condition with
$c \in (0,2^{-1})$
and a current estimate
$\bm{x}_{n} \in \dom{\phi}$:
\begin{equation}
  \thickmuskip=0.0\thickmuskip
  \medmuskip=0.0\medmuskip
  \thinmuskip=0.0\thinmuskip
  \arraycolsep=0.0\arraycolsep
  \mathrm{Armijo\mathchar`-type\ condition\ }(n): \
  (\cost_{n}+\phi)(\bm{x})
  \leq \NMV_{n} - c\gamma \norm{\frac{\bm{x}_{n}-\bm{x}}{\gamma}}^{2}
  \left(=\NMV_{n} - \frac{c}{\gamma} \norm{\bm{x}_{n}-\bm{x}}^{2}\right),
  \label{eq:Armijo}
\end{equation}
where the reference value
$\NMV_{n} \in \mathbb{R}$
is given by~\eqref{eq:NMV} in Algorithm~\ref{alg:proposed}.

By setting
$\theta_{n} \coloneqq 1\ (n\in\mathbb{N})$
in line~\ref{line:p_n} of Algorithm~\ref{alg:proposed},
each reference value is given by
$\NMV_{n} = (\cost_{n}+\phi)(\bm{x}_{n})$
(see~\eqref{eq:NMV}),
and thus the Armijo-type condition
$(n)$
in~\eqref{eq:Armijo} coincides with the standard monotone Armijo condition for
$\cost_{n} + \phi$
(see, e.g.,~\eqref{eq:Armijo_general} below; see also, e.g.,~\cite{Beck17,Bolte-Sabach-Teboulle14,Themelis-Stella-Patrinos18}).
In this case, the Armijo-type condition $(n)$
in~\eqref{eq:Armijo}
for
$\bm{x} \neq \bm{x}_{n}$
implies
$(\cost_{n}+\phi)(\bm{x}) < (\cost_{n}+\phi)(\bm{x}_{n})$,
i.e., a monotonic decrease
in the value of
$\cost_{n}+\phi$.
For this reason,
we call the condition~\eqref{eq:Armijo} the monotone Armijo condition when
$\theta_{n} \coloneqq 1$.

In contrast, by setting
$\theta_{n} \in (0,1)\ (n > 1)$,
we have
$\NMV_{n} =\theta_{n}(\cost_{n}+\phi)(\bm{x}_{n}) + (1-\theta_{n})\left(\NMV_{n-1}+\kappa\left(\mu_{n-1}-\mu_{n}\right)\right)$,
where
$\NMV_{n-1} \in \mathbb{R}$
is the previous reference value
and
$\kappa \in \mathbb{R}_{++}$
is given in Definition~\ref{definition:smoothing}~\ref{enum:consistency:uniform}.
Then, the Armijo-type condition $(n)$ in~\eqref{eq:Armijo} for
$\bm{x} \neq \bm{x}_{n}$
does not imply
$(\cost_{n}+\phi)(\bm{x}) < (\cost_{n}+\phi)(\bm{x}_{n})$,
and is essentially identical to the (average-type) nonmonotone Armijo condition, e.g.,~\cite{Zhang-Hager04,Themelis-Stella-Patrinos18,Kanzow-Mehlitz22,DeMarchi23,Jia-Kanwoz-Mehlitz23,Kanzow-Lehmann25,Yagishita-Ito25A}.
We note that the additional term
$\kappa(\mu_{n-1}-\mu_{n})$
in~\eqref{eq:NMV} is not required in~\cite{Zhang-Hager04,Themelis-Stella-Patrinos18,Kanzow-Mehlitz22,DeMarchi23,Jia-Kanwoz-Mehlitz23,Kanzow-Lehmann25,Yagishita-Ito25A}, but is required for the proposed algorithm in order to compensate for the drift due to the time-varying function
$\cost_{n}$ (see also~\eqref{eq:function_inequality} below).
It is reported, e.g., in
\cite{Zhang-Hager04,Ahookhosh-Ghaderi17},
that the use of a nonmonotone Armijo condition can improve the likelihood of finding a global minimizer because the generated sequence may escape valleys around local minimizers, compared with the monotone Armijo condition.

In order to find
$(\bm{x},\gamma)$
satisfying the Armijo-type condition $(n)$ in~\eqref{eq:Armijo},
we employ a {\em backtracking step} in
lines~\ref{line:backtracking_start}-\ref{line:backtracking_end} of Algorithm~\ref{alg:proposed}.
The following lemma ensures that
the loop in lines~\ref{line:backtracking_start}-\ref{line:backtracking_end} of Algorithm~\ref{alg:proposed} is terminated within a finite number of iterations.

\begin{lemma}[Finite termination of the backtracking step]
  \label{lemma:stepsize}
  Let
  $\cost + \phi$
  satisfy Assumptions~\ref{assumption:basic} and~\ref{assumption:regular}, and
  $\{\cost^{\langle \mu \rangle}\}_{\mu \in (0,\widetilde{\mu})}$
  with some
  $\widetilde{\mu} \in \exRp$
  be a smoothing function of
  $\cost$.
  Choose arbitrarily
  $\bm{x}_{1} \in \dom{\phi}$,
  $c \in (0,2^{-1})$,
  $\rho \in (0,1)$,
  $\gamma_{\rm init} \in (0,\gamma_{\phi})$,
  and
  $\ratio \in (0,1)$.
  Then,
  for
  $\cost_{n}\coloneqq \cost^{\langle \mu_{n}\rangle}\ (n\in \mathbb{N})$
  with
  $(\mu_{n})_{n=1}^{\infty} \subset (0,\widetilde{\mu})$
  satisfying~\eqref{eq:smoothing_parameter},
  and
  for
  $(\bm{x}_{n})_{n=1}^{\infty} \subset \dom{\phi}$
  generated by Algorithm~\ref{alg:proposed}, the following hold:
  \begin{enumerate}[label=(\alph*)]
    \item
          \label{enum:Armijo_implies_inequality}
          For
          $n$-th estimate
          $\bm{x}_{n} \in \dom{\phi}$
          in Algorithm~\ref{alg:proposed},
          assume that the Armijo-type condition
          $(n)$
          in~\eqref{eq:Armijo}
          is achieved by
          $(\bm{x},\gamma) \coloneqq (\bm{x}_{n+1},\gamma_{n})$,
          i.e.,
          \begin{equation}
            (\cost_{n}+\phi)(\bm{x}_{n+1})
            \leq \NMV_{n} - c\gamma_{n} \norm{\frac{\bm{x}_{n}-\bm{x}_{n+1}}{\gamma_{n}}}^{2}
            =\NMV_{n} - \frac{c}{\gamma_{n}} \norm{\bm{x}_{n}-\bm{x}_{n+1}}^{2}, \label{eq:Armijo_assumption}
          \end{equation}
          where
          $\NMV_{n}$
          is given in~\eqref{eq:NMV}.
          Then, with
          $\kappa\in\mathbb{R}_{++}$
          in Def.~\ref{definition:smoothing}~\ref{enum:consistency:uniform},
          we have
          \begin{equation}
            \left(\cost_{n+1}+\phi\right)(\bm{x}_{n+1})
            \leq \NMV_{n+1}
            \leq \NMV_{n} - \ratio c\gamma_{n}\left(\mathcal{M}_{\gamma_{n}}^{\cost_{n},\phi}\left(\bm{x}_{n}\right)\right)^{2} + \kappa\left(\mu_{n}-\mu_{n+1}\right). \label{eq:function_measure_inequality_algorithm}
          \end{equation}
    \item
          \label{enum:lemma:stepsize:backtracking}
          Let
          $\Delta_{n} \coloneqq \min\{\gamma_{\rm init},(1-2c)L_{\nabla \cost_{n}}^{-1}\} \in (0,\gamma_{\rm init}]\ (n\in\mathbb{N})$
          with the Lipschitz constant
          $L_{\nabla \cost_{n}}>0$
          of
          $\nabla \cost_{n}$
          (see Definition~\ref{definition:smoothing}~\ref{enum:consistency:gradient_Lipschitz} and~\ref{enum:consistency:gradient_Lipschitz_constant}).
          Then, at each iteration
          $n$,
          the loop in lines~\ref{line:backtracking_start}-\ref{line:backtracking_end} of Algorithm~\ref{alg:proposed} is terminated within at most
          $l_{n}+1$
          iterations
          with
          $l_{n}\coloneqq \lceil \log_{\rho}(\Delta_{n}\gamma_{\rm init}^{-1})\rceil$.
          Hence,
          for every
          $n\in\mathbb{N}$,
          the Armijo-type condition
          $(n)$
          in~\eqref{eq:Armijo}
          is satisfied with
          $(\bm{x},\gamma) \coloneqq (\bm{x}_{n+1},\gamma_{n}) \in \dom{\phi}\times (0,\gamma_{\rm init}]$,
          i.e., the inequality~\eqref{eq:Armijo_assumption} holds,
          because of line~\ref{line:Armijo} of Algorithm~\ref{alg:proposed}.
    \item
          \label{enum:lemma:stepsize:lower_bound}
          For
          $(\gamma_{n})_{n=1}^{\infty}$
          generated by Algorithm~\ref{alg:proposed}
          and
          $\beta \coloneqq\min\left\{\gamma_{\rm init} L_{\nabla \cost_{1}}, \rho(1-2c)\right\} > 0$,
          we have
          $\gamma_{n} \geq \beta L_{\nabla \cost_{n}}^{-1}\ (n\in\mathbb{N})$.
  \end{enumerate}
\end{lemma}

The proof of Lemma~\ref{lemma:stepsize} relies on Fact~\ref{fact:Armijo} (see~\cite[Lemma 2]{Bolte-Sabach-Teboulle14}\footnote{
  Strictly speaking, \cite[Lemma 2]{Bolte-Sabach-Teboulle14} assumes
  $\mathrm{int}(\dom{J}) = \mathcal{X}$
  and
  $\inf_{\bm{x} \in \mathcal{X}} \phi(\bm{x}) > -\infty$.
  Nevertheless, the same discussion in the proof of~\cite[Lemma 2]{Bolte-Sabach-Teboulle14} can apply to a relaxed case under the setting in Fact~\ref{fact:Armijo}.
},~\cite[Lemma 2.1]{Themelis-Stella-Patrinos18}).

\begin{fact}[Sufficient decrease property  with smooth function $J$]
  \label{fact:Armijo}
  Let
  $\phi:\mathcal{X} \to \exR$
  satisfy Assumption~\ref{assumption:basic}~\ref{enum:problem:origin:phi}
  and let
  $J:\mathcal{X}\to\exR$
  be a continuously differentiable function over
  $\mathrm{int}(\dom{J})$
  such that
  $\mathrm{int}(\dom{J})$
  is convex,
  $\mathrm{int}(\dom{J})\supset \dom{\phi}$
  and
  $\nabla J:\mathrm{int}(\dom{J})\to\mathcal{X}$
  is Lipschitz continuous with a Lipschitz constant
  $L_{\nabla J} >0$
  over
  $\mathrm{int}(\dom{J})$.
  Then,
  for
  $c \in (0,2^{-1})$,
  $\gamma_{\rm init} \in (0,\gamma_{\phi})$,
  and
  $\Delta \coloneqq \min\{\gamma_{\rm init},(1-2c)L_{\nabla J}^{-1}\}$,
  the following Armijo condition holds:
  \begin{align}
     & \left(\forall \widebar{\bm{x}}\in \dom{\phi},\; \forall \gamma \in (0,\Delta],\; \forall \bm{p} \in \prox{\gamma\phi}(\widebar{\bm{x}}-\gamma\nabla J(\widebar{\bm{x}}))\right) \\
     & \hspace{15em}
    (J+\phi)(\bm{p})
    \leq (J+\phi)(\widebar{\bm{x}}) - c\gamma \norm{\frac{\widebar{\bm{x}}-\bm{p}}{\gamma}}^{2}.
    \label{eq:Armijo_general}
  \end{align}
\end{fact}

\begin{proof}[Proof of Lemma~\ref{lemma:stepsize}]
  \ref{enum:Armijo_implies_inequality}
  By the condition~\ref{enum:consistency:uniform} in Definition~\ref{definition:smoothing},
  we have
  \begin{equation}
    (n\in \mathbb{N}) \quad
    \left(\cost_{n+1}+\phi\right)\left(\bm{x}_{n+1}\right)
    \leq
    \left(\cost_{n}+\phi\right)\left(\bm{x}_{n+1}\right) + \kappa\left(\mu_{n}-\mu_{n+1}\right) \label{eq:function_inequality}
  \end{equation}
  with
  $\cost_{n} = \cost^{\langle \mu_{n} \rangle}$.
  With
  $\theta_{n+1} \in [\ratio ,1]$,
  the first inequality in~\eqref{eq:function_measure_inequality_algorithm} follows from
    {
      \thickmuskip=0.0\thickmuskip
      \medmuskip=0.0\medmuskip
      \thinmuskip=0.0\thinmuskip
      \arraycolsep=0.0\arraycolsep
      \begin{align*}
         & \NMV_{n+1}
        \overset{\eqref{eq:NMV}}{=}\theta_{n+1}\left(\cost_{n+1}+\phi\right)\left(\bm{x}_{n+1}\right) + \left(1-\theta_{n+1}\right)\left(\NMV_{n}+\kappa\left(\mu_{n}-\mu_{n+1}\right)\right) \\
         & \overset{\eqref{eq:Armijo_assumption}}{\geq}\theta_{n+1}\left(\cost_{n+1}+\phi\right)(\bm{x}_{n+1})
        + (1-\theta_{n+1})\left(\hspace{-0.25em}(\cost_{n}+\phi)(\bm{x}_{n+1}) + \hspace{-0.1em}\frac{c}{\gamma_{n}}\norm{\bm{x}_{n}-\bm{x}_{n+1}}^{2}+\kappa(\mu_{n}-\mu_{n+1})\hspace{-0.25em}\right) \\
         & \overset{\eqref{eq:function_inequality}}{\geq}\theta_{n+1}\left(\cost_{n+1}+\phi\right)\left(\bm{x}_{n+1}\right) + (1-\theta_{n+1})\left((\cost_{n+1}+\phi)\left(\bm{x}_{n+1}\right) + \frac{c}{\gamma_{n}}\norm{\bm{x}_{n}-\bm{x}_{n+1}}^{2}\right) \\
         & = \left(\cost_{n+1}+\phi\right)\left(\bm{x}_{n+1}\right) + \frac{c(1-\theta_{n+1})}{\gamma_{n}}\norm{\bm{x}_{n}-\bm{x}_{n+1}}^{2}
        \geq \left(\cost_{n+1}+\phi\right)\left(\bm{x}_{n+1}\right).
      \end{align*}
    }%
  The second inequality in~\eqref{eq:function_measure_inequality_algorithm} follows from
  \begin{align*}
     & \NMV_{n+1}
    \overset{\eqref{eq:NMV}}{=}\theta_{n+1}\left(\cost_{n+1}+\phi\right)\left(\bm{x}_{n+1}\right) + (1-\theta_{n+1})\left(\NMV_{n}+\kappa\left(\mu_{n}-\mu_{n+1}\right)\right) \\
     & \overset{\eqref{eq:function_inequality}}{\leq}\theta_{n+1}\left(\left(\cost_{n}+\phi\right)\left(\bm{x}_{n+1}\right) + \kappa\left(\mu_{n}-\mu_{n+1}\right)\right) + (1-\theta_{n+1})\left(\NMV_{n}+ \kappa\left(\mu_{n}-\mu_{n+1}\right)\right) \\
     & \overset{\hphantom{\eqref{eq:function_inequality}}}{=}\theta_{n+1}\left(\cost_{n}+\phi\right)\left(\bm{x}_{n+1}\right) + (1-\theta_{n+1})\NMV_{n}+\kappa\left(\mu_{n}-\mu_{n+1}\right) \\
     & \overset{\eqref{eq:Armijo_assumption}}{\leq}\theta_{n+1}\left(\NMV_{n} -c\gamma_{n}\norm{\frac{\bm{x}_{n}-\bm{x}_{n+1}}{\gamma_{n}}}^{2}\right) + (1-\theta_{n+1})\NMV_{n}+\kappa\left(\mu_{n}-\mu_{n+1}\right) \\
     & \overset{\hphantom{\eqref{eq:Armijo}}}{=} \NMV_{n} -\theta_{n+1}c\gamma_{n}\norm{\frac{\bm{x}_{n}-\bm{x}_{n+1}}{\gamma_{n}}}^{2}  + \kappa\left(\mu_{n}-\mu_{n+1}\right) \\
     & \overset{\eqref{eq:measure_limiting}}{\leq} \NMV_{n} -\ratio c\gamma_{n}\left(\mathcal{M}_{\gamma_{n}}^{\cost_{n},\phi}\left(\bm{x}_{n}\right)\right)^{2}  + \kappa\left(\mu_{n}-\mu_{n+1}\right). \quad
    (\because \ratio \leq\theta_{n+1})
  \end{align*}

  \ref{enum:lemma:stepsize:backtracking}
  By Definition~\ref{definition:smoothing}~\ref{enum:consistency:gradient_Lipschitz},
  $\nabla \cost_{n}=\nabla \cost^{\langle \mu_{n} \rangle}$
  is
  $L_{\nabla \cost_{n}}$-Lipschitz continuous over the convex set
  $\mathrm{int}(\dom{\cost_{n}}) = \mathrm{int}(\dom{\cost})$
  with some
  $L_{\nabla \cost_{n}} > 0$.
  Hence, by applying Fact~\ref{fact:Armijo} with
  $J=\cost_{n}$
  and
  $\widebar{\bm{x}} = \bm{x}_{n} \in \dom{\phi}$,
  we get
  \begin{align}
     & (\forall\gamma \in (0,\Delta_{n}],\ \forall\bm{p} \in \prox{\gamma\phi}(\bm{x}_{n} - \gamma \nabla \cost_{n}(\bm{x}_{n}))) \\
     & \hspace{10em}
    (\cost_{n}+\phi)(\bm{p}) \leq  (\cost_{n}+\phi)(\bm{x}_{n}) - c\gamma \norm{\frac{\bm{x}_{n}-\bm{p}}{\gamma}}^{2}.
    \label{eq:Armijo_standard}
  \end{align}

  We show the statement by induction.

  Let
  $n=1$.
  Then,
  we have
  $\NMV_{1} = (\cost_{1}+\phi)(\bm{x}_{1})$
  from~\eqref{eq:NMV}.
  By~\eqref{eq:Armijo_standard},
  $(\cost_{1}+\phi)(\bm{p}) \leq (\cost_{1}+\phi)(\bm{x}_{1})- c\gamma \norm{(\bm{x}_{1}-\bm{p})/\gamma}^{2}= \NMV_{1} - c\gamma \norm{(\bm{x}_{1}-\bm{p})/\gamma}^{2}$
  holds for
  $\gamma\coloneqq \rho^{l_{1}}\gamma_{\rm init} \in (0,\Delta_{1}]$
  with
  $l_{1}= \lceil \log_{\rho}(\Delta_{1}\gamma_{\rm init}^{-1})\rceil$,
  and for every
  $\bm{p} \in \prox{\gamma\phi}(\bm{x}_{1} - \gamma \nabla \cost_{1}(\bm{x}_{1}))$.
  Hence, for
  $n = 1$,
  the loop in lines~\ref{line:backtracking_start}-\ref{line:backtracking_end} of Algorithm~\ref{alg:proposed} is terminated within at most
  $l_{1}+1$
  iterations.

  Let
  $n>1$,
  and assume the inductive hypothesis:
  the Armijo-type condition
  $(n-1)$
  in~\eqref{eq:Armijo} is achieved by
  $(\bm{x},\gamma)\coloneqq (\bm{x}_{n},\gamma_{n-1})$.
  Then, the first inequality of~\eqref{eq:function_measure_inequality_algorithm} in
  Lemma~\ref{lemma:stepsize}~\ref{enum:Armijo_implies_inequality} yields
  $\left(\cost_{n}+\phi\right)(\bm{x}_{n}) \leq \NMV_{n}$.
  Together with~\eqref{eq:Armijo_standard},
  we get
  $(\cost_{n}+\phi)(\bm{p}) \leq \left(\cost_{n}+\phi\right)(\bm{x}_{n}) - c\gamma \norm{(\bm{x}_{n}-\bm{p})/\gamma}^{2} \leq \NMV_{n} - c\gamma \norm{(\bm{x}_{n}-\bm{p})/\gamma}^{2}$
  for
  $\gamma \coloneqq \rho^{l_{n}}\gamma_{\rm init} \in (0,\Delta_{n}]$
  with
  $l_{n} = \lceil \log_{\rho}(\Delta_{n}\gamma_{\rm init}^{-1})\rceil$
  and for every
  $\bm{p} \in \prox{\gamma\phi}(\bm{x}_{n} - \gamma \nabla \cost_{n}(\bm{x}_{n}))$.
  Hence,
  for
  $n > 1$,
  the loop in lines~\ref{line:backtracking_start}-\ref{line:backtracking_end} of Algorithm~\ref{alg:proposed} is terminated within at most
  $l_{n}+1$
  iterations.

  \ref{enum:lemma:stepsize:lower_bound}
  For every
  $n \in \mathbb{N}$,
  consider two cases:
  $\gamma_{n} = \gamma_{\rm init}$
  and
  $\gamma_{n} \leq \rho \gamma_{\rm init}$.
  For the case
  $\gamma_{n} = \gamma_{\rm init}$,
  we have
  $\gamma_{n}=\gamma_{\rm init}L_{\nabla \cost_{n}}L_{\nabla \cost_{n}}^{-1} \geq \gamma_{\rm init}L_{\nabla \cost_{1}}L_{\nabla \cost_{n}}^{-1} \geq \beta L_{\nabla \cost_{n}}^{-1}$,
  where
  $L_{\nabla \cost_{n}} \geq L_{\nabla \cost_{1}}$
  follows from
  $L_{\nabla \cost_{n}} = \varpi_{1} + \varpi_{2}\mu_{n}^{-1}$
  (see Definition~\ref{definition:smoothing}~\ref{enum:consistency:gradient_Lipschitz_constant})
  and from
  $\mu_{n} \leq \mu_{1}$
  (see the condition (iii) in~\eqref{eq:smoothing_parameter}).
  Consider the other case
  $\gamma_{n} \leq \rho \gamma_{\rm init}$.
  By the procedure in lines~\ref{line:backtracking_start}-\ref{line:backtracking_end} of Algorithm~\ref{alg:proposed},
  the Armijo-type condition $(n)$ in~\eqref{eq:Armijo} does not hold for
  $\gamma \coloneqq \rho^{-1}\gamma_{n}$
  and for some
  $\bm{x} \in \prox{\gamma\phi}(\bm{x}_{n}-\gamma\nabla\cost_{n}(\bm{x}_{n}))$.
  Hence,
  $\Delta_{n} < \rho^{-1}\gamma_{n}$
  must hold by~\eqref{eq:Armijo_standard}.
  Together with
  $\gamma_{n} \leq \rho \gamma_{\rm init}$,
  we get
  $\Delta_{n} < \gamma_{\rm init}$,
  and thus
  $\Delta_{n} (=\min\{\gamma_{\rm init},(1-2c)L_{\nabla \cost_{n}}^{-1}\}) = (1-2c)L_{\nabla \cost_{n}}^{-1}$.
  Consequently, we get
  $\gamma_{n} > \rho\Delta_{n} = \rho(1-2c) L_{\nabla \cost_{n}}^{-1} \geq \beta L_{\nabla \cost_{n}}^{-1}$.
\end{proof}

Finally, we present an asymptotic convergence analysis of Algorithm~\ref{alg:proposed} below.

\begin{theorem}[Convergence analysis of Algorithm~\ref{alg:proposed}]
  \label{theorem:convergence_extension}
  Let
  $\cost + \phi$
  satisfy Assumptions~\ref{assumption:basic} and~\ref{assumption:regular},
  and
  $\{\cost^{\langle \mu \rangle}\}_{\mu \in (0,\widetilde{\mu})}$
  with some
  $\widetilde{\mu} \in \exRp$
  be a smoothing function of
  $\cost$.
  Choose arbitrarily
  $\bm{x}_{1} \in \dom{\phi}$,
  $c \in (0,2^{-1})$,
  $\rho \in (0,1)$,
  $\gamma_{\rm init} \in (0,\gamma_{\phi})$,
  and
  $\ratio \in (0,1)$.
  Then,
  for
  $\cost_{n}\coloneqq \cost^{\langle \mu_{n}\rangle}\ (n\in \mathbb{N})$
  with
  $(\mu_{n})_{n=1}^{\infty} \subset (0,\widetilde{\mu})$
  satisfying~\eqref{eq:smoothing_parameter},
  and
  for
  $(\bm{x}_{n})_{n=1}^{\infty} \subset \dom{\phi}$
  generated by Algorithm~\ref{alg:proposed}, the following hold:
  \begin{enumerate}[label=(\alph*)]
    \item
          \label{enum:theorem:convergence_extension:NMV}
          For
          $\widehat{\NMV}_{n} \coloneqq \NMV_{n} + \kappa\mu_{n}\ (n\in\mathbb{N})$
          with
          $\NMV_{n} \in \mathbb{R}$
          in~\eqref{eq:NMV}, and
          $\kappa \in \mathbb{R}_{++}$
          given in Definition~\ref{definition:smoothing}~\ref{enum:consistency:uniform},
          we have
            {
              \thickmuskip=0.3\thickmuskip
              \medmuskip=0.3\medmuskip
              \thinmuskip=0.3\thinmuskip
              \arraycolsep=0.3\arraycolsep
              \begin{align}
                (n\in\mathbb{N}) \quad &
                \widehat{\NMV}_{n} - \widehat{\NMV}_{n+1} \geq
                \ratio c\gamma_{n}\left(\mathcal{M}_{\gamma_{n}}^{\cost_{n},\phi}\left(\bm{x}_{n}\right)\right)^{2} \geq
                0, \label{eq:measure_NMV} \\
                (n\in\mathbb{N}) \quad &
                \widehat{\NMV}_{n} \geq (\cost + \phi)(\bm{x}_{n}) \geq \inf_{\bm{x}\in\mathcal{X}} (\cost+\phi)(\bm{x}) > -\infty,
                \label{eq:measure_NMV_inf} \\
                (n\in\mathbb{N}) \quad &
                \bm{x}_{n} \in \lev{(\cost+\phi)}{\widehat{\NMV}_{1}}\coloneqq \{\bm{x} \in \mathcal{X} \mid (\cost + \phi)(\bm{x}) \leq \widehat{\NMV}_{1}\} \subset \dom{\phi}.
                \label{eq:level_set}
              \end{align}
            }%
    \item
          \label{enum:theorem:convergence_extension:rate}
          For any pair
          $(\underline{k}, \overline{k}) \in \mathbb{N}^{2}$
          satisfying
          $\underline{k} \leq \overline{k}$,
          we have
          \begin{equation}
            \sum_{n=\underline{k}}^{\overline{k}} \mu_{n}
            \left(\mathcal{M}_{\gamma_{n}}^{\cost_{n},\phi}(\bm{x}_{n})\right)^{2} \leq
            \chi, \quad
            \min_{\underline{k}\leq n \leq \overline{k}}
            \mathcal{M}_{\gamma_{n}}^{\cost_{n},\phi}(\bm{x}_{n})
            \leq
            \sqrt{\frac{\chi}{\sum_{n=\underline{k}}^{\overline{k}}\mu_{n}}},
            \label{eq:convergence_rate}
          \end{equation}
          where
          $\chi\coloneqq \frac{\left((\cost+\phi)(\bm{x}_{1}) + \kappa\mu_{1} - \inf_{\bm{x}\in\mathcal{X}} (\cost +\phi)(\bm{x})\right)(\varpi_{1} + \mu_{1}^{-1}\varpi_{2})}{\ratio c\beta\mu_{1}^{-1}} \in \mathbb{R}_{++}$
          is given with
          $\varpi_{1} \in \mathbb{R}_{+}$ and
          $\varpi_{2} \in \mathbb{R}_{++}$
          in Definition~\ref{definition:smoothing}~\ref{enum:consistency:gradient_Lipschitz_constant}, and
          $\beta \in \mathbb{R}_{++}$
          in Lemma~\ref{lemma:stepsize}~\ref{enum:lemma:stepsize:lower_bound}.
          In particular, by the second inequality of~\eqref{eq:convergence_rate}, we get a convergence rate:
          \begin{equation*}
            \min\limits_{1 \leq n \leq k} \mathcal{M}_{\gamma_{n}}^{\cost_{n},\phi}(\bm{x}_{n}) = \mathcal{O}\left(k^{-\frac{\alpha-1}{2\alpha}}\right)\
            \mathrm{as}\ k\to\infty
          \end{equation*}
          if
          $(\mu_{n})_{n=1}^{\infty} \coloneqq (\tau n^{-1/\alpha})_{n=1}^{\infty}\ (\tau \in (0,\widetilde{\mu}),\ \alpha > 1)$
          is employed because we have
          $\sum_{n=1}^{k} \mu_{n} \geq \frac{\tau }{1-\alpha^{-1}}((k+1)^{\frac{\alpha-1}{\alpha}} - 1)$
          (see, e.g.,~\cite[(27) in Exm. 4.6]{Kume-Yamada24A}), where
          $\mathcal{O}(\cdot)$
          stands for the Landau's big-O notation.
    \item
          \label{enum:liminf}
          ${\liminf_{n\to\infty} \mathcal{M}_{\gamma_{n}}^{\cost_{n},\phi}(\bm{x}_{n}) = 0}$
          holds.
          Moreover, choose a subsequence
          $(\bm{x}_{m(l)})_{l=1}^{\infty}$
          of
          $(\bm{x}_{n})_{n=1}^{\infty}\subset \dom{\phi}$
          such that
          $\lim_{l\to\infty} \mathcal{M}_{\gamma_{m(l)}}^{\cost_{m(l)},\phi}(\bm{x}_{m(l)}) = 0$,
          where
          $m:\mathbb{N}\to\mathbb{N}$
          is monotonically increasing.
          Then,
          every cluster point
          $\widebar{\bm{x}} \in \mathcal{X}$
          of
          $(\bm{x}_{m(l)})_{l=1}^{\infty}$
          belongs to
          $\dom{\phi}$.
          Moreover,
          such a cluster point
          $\widebar{\bm{x}} \in \mathcal{X}$
          is an {\rm F}-stationary point of
          $\cost+\phi$
          if Assumption~\ref{assumption:phi} is additionally satisfied.

  \end{enumerate}
\end{theorem}
\begin{proof}

  \ref{enum:theorem:convergence_extension:NMV}
  Recall that the inequality~\eqref{eq:Armijo_assumption} holds for all
  $n\in \mathbb{N}$
  from Lemma~\ref{lemma:stepsize}~\ref{enum:lemma:stepsize:backtracking}.
  Then, the second inequality of~\eqref{eq:function_measure_inequality_algorithm} in
  Lemma~\ref{lemma:stepsize}~\ref{enum:Armijo_implies_inequality} together with
  $\widehat{\NMV}_{n} = \NMV_{n} + \kappa\mu_{n}$
  yields
  $\widehat{\NMV}_{n+1} \leq \widehat{\NMV}_{n} - \ratio c\gamma_{n}(\mathcal{M}_{\gamma_{n}}^{\cost_{n},\phi}(\bm{x}_{n}))^{2}$,
  i.e.,
  the inequality in~\eqref{eq:measure_NMV} holds.
  From
  $(\cost_{1}+\phi)(\bm{x}_{1}) = \NMV_{1}$
  by~\eqref{eq:NMV}
  and
  $(\cost_{n+1}+\phi)(\bm{x}_{n+1}) \leq \NMV_{n+1}\ (n \in \mathbb{N})$
  by~\eqref{eq:function_measure_inequality_algorithm},
  we have
  $(\cost + \phi)(\bm{x}_{n}) \overset{\eqref{eq:uniform_bounded_smoothing_original}}{\leq} (\cost_{n}+\phi)(\bm{x}_{n}) + \kappa\mu_{n} \leq \NMV_{n} + \kappa\mu_{n} = \widehat{\NMV}_{n}\ (n\in\mathbb{N})$.
  By combining this inequality with
  $\inf_{\bm{x}\in \mathcal{X}} (\cost+\phi)(\bm{x}) > -\infty$
  in Assumption~\ref{assumption:basic}~\ref{enum:problem:origin:minimizer},
  we get the inequality in~\eqref{eq:measure_NMV_inf}.
  Together with~\eqref{eq:measure_NMV},
  we have
  $(\cost + \phi)(\bm{x}_{n}) \leq \widehat{\NMV}_{n}
    \leq \widehat{\NMV}_{1}\ (n\in\mathbb{N})$,
  i.e.,~\eqref{eq:level_set} follows.

  \ref{enum:theorem:convergence_extension:rate}
  By summing the inequality~\eqref{eq:measure_NMV} up from
  $n=\underline{k}$
  to
  $\overline{k}$,
  we have
  \begin{equation}
    \ratio c\sum_{n=\underline{k}}^{\overline{k}}  \gamma_{n}\left(\mathcal{M}_{\gamma_{n}}^{\cost_{n},\phi}(\bm{x}_{n})\right)^{2}  \leq \widehat{\NMV}_{\underline{k}} - \widehat{\NMV}_{\overline{k}+1}  \leq (\cost+\phi)(\bm{x}_{1}) + \kappa\mu_{1}
    - \inf_{\bm{x}\in \mathcal{X}}(\cost+\phi)(\bm{x}) <+\infty, \label{eq:extension:sum_gradient_bound}
  \end{equation}
  where
  we used
  $\widehat{\NMV}_{\underline{k}} \overset{\eqref{eq:measure_NMV}}{\leq} \widehat{\NMV}_{1} = \NMV_{1} + \kappa\mu_{1} \overset{\eqref{eq:NMV}}{=} (\cost_{1} + \phi)(\bm{x}_{1}) + \kappa\mu_{1} \overset{\eqref{eq:uniform_bounded_smoothing_original}}{\leq}(\cost+\phi)(\bm{x}_{1}) + \kappa\mu_{1}$,
  and
  $\widehat{\NMV}_{\overline{k}+1} \overset{\eqref{eq:measure_NMV_inf}}{\geq} \inf_{\bm{x}\in\mathcal{X}} (\cost+\phi)(\bm{x}) > -\infty$.
  By Lemma~\ref{lemma:stepsize}~\ref{enum:lemma:stepsize:lower_bound}, we have a lower bound
  $\gamma_{n} \geq \beta L_{\nabla \cost_{n}}^{-1}$
  with some
  $\beta \in \mathbb{R}_{++}$.
  Since we have
  $L_{\nabla \cost_{n}} = \varpi_{1} + \varpi_{2}\mu_{n}^{-1}$
  with
  $\varpi_{1} \in \mathbb{R}_{+},\ \varpi_{2} \in\mathbb{R}_{++}$
  in Definition~\ref{definition:smoothing}~\ref{enum:consistency:gradient_Lipschitz_constant},
  we get
  $\gamma_{n} \geq \beta L_{\nabla \cost_{n}}^{-1} = \frac{\beta\mu_{n}}{\varpi_{1}\mu_{n} + \varpi_{2}} =\frac{\beta\mu_{1}^{-1}\mu_{n}}{\varpi_{1}\mu_{1}^{-1}\mu_{n} + \mu_{1}^{-1}\varpi_{2}} \geq \frac{\beta\mu_{1}^{-1}}{\varpi_{1}+\mu_{1}^{-1}\varpi_{2}}\mu_{n}$
  for
  $n\in\mathbb{N}$,
  where the last inequality follows by
  $\mu_{1}^{-1}\mu_{n} \in (0,1]$
  from
  $\mu_{n} \leq \mu_{1}$
  (see the condition (iii) in~\eqref{eq:smoothing_parameter}).
  Then, the LHS in~\eqref{eq:extension:sum_gradient_bound} is bounded below as
  \begin{align}
    \ratio c\sum_{n=\underline{k}}^{\overline{k}} \gamma_{n}
    \left(\mathcal{M}_{\gamma_{n}}^{\cost_{n},\phi}(\bm{x}_{n})\right)^{2}
     & \geq \frac{\ratio c\beta\mu_{1}^{-1}}{\varpi_{1} + \mu_{1}^{-1}\varpi_{2}}\sum_{n=\underline{k}}^{\overline{k}} \mu_{n}
    \left(\mathcal{M}_{\gamma_{n}}^{\cost_{n},\phi}(\bm{x}_{n})\right)^{2} \\
     & \geq \frac{\ratio c\beta\mu_{1}^{-1}}{\varpi_{1} + \mu_{1}^{-1}\varpi_{2}}\min_{\underline{k}\leq n \leq \overline{k}}
    \left(\mathcal{M}_{\gamma_{n}}^{\cost_{n},\phi}(\bm{x}_{n})\right)^{2}
    \sum_{n=\underline{k}}^{\overline{k}} \mu_{n}. \label{eq:tmp_convergence}
  \end{align}
  By combining~\eqref{eq:extension:sum_gradient_bound} with~\eqref{eq:tmp_convergence}, we deduce the desired inequalities in~\eqref{eq:convergence_rate}.

  \ref{enum:liminf}
  Assume contrarily that
  $\liminf_{n\to\infty} \mathcal{M}_{\gamma_{n}}^{\cost_{n},\phi}(\bm{x}_{n}) > 0$,
  i.e., there exist
  $\epsilon > 0$
  and
  $n_{0} \in \mathbb{N}$
  such that
  $\mathcal{M}_{\gamma_{n}}^{\cost_{n},\phi}(\bm{x}_{n}) > \epsilon$
  for all
  $n\geq n_{0}$.
  By the first inequality in~\eqref{eq:convergence_rate}, we have
  $\chi \geq \sum_{n=n_{0}}^{\infty}\mu_{n} \left(\mathcal{M}_{\gamma_{n}}^{\cost_{n},\phi}(\bm{x}_{n})\right)^{2} > \sum_{n=n_{0}}^{\infty}\mu_{n}\epsilon^{2} > 0$.
  In contrast,
  by
  $\sum_{n=1}^{\infty}\mu_{n} = +\infty$
  (see the condition (ii) in~\eqref{eq:smoothing_parameter}),
  we have
  $\sum_{n=n_{0}}^{\infty}\mu_{n}\epsilon^{2} = +\infty$,
  which is absurd.

  By
  $\liminf_{n\to\infty} \mathcal{M}_{\gamma_{n}}^{\cost_{n},\phi}(\bm{x}_{n}) = 0$,
  there exists a subsequence
  $(\bm{x}_{m(l)})_{l=1}^{\infty}$
  of
  $(\bm{x}_{n})_{n=1}^{\infty}$
  satisfying
  $\lim_{l\to\infty} \mathcal{M}_{\gamma_{m(l)}}^{\cost_{m(l)},\phi}(\bm{x}_{m(l)}) = 0$.
  Let
  $\widebar{\bm{x}} \in \mathcal{X}$
  be a cluster point of
  $(\bm{x}_{m(l)})_{l=1}^{\infty}$.
  Without loss of generality,
  we can assume
  $\lim_{l\to\infty}\bm{x}_{m(l)} = \widebar{\bm{x}}$
  (by passing through further subsequence if necessary).
  Then,
  we get
  $\widebar{\bm{x}} \in\lev{(\cost+\phi)}{\widehat{\NMV}_{1}}\subset \dom{\phi}$
  because
  $\lev{(\cost+\phi)}{\widehat{\NMV}_{1}}$
  is closed by the lower semicontinuity of
  $\cost + \phi$,
  and
  $(\bm{x}_{m(l)})_{l=1}^{\infty} \subset \lev{(\cost+\phi)}{\widehat{\NMV}_{1}}$
  by~\eqref{eq:level_set}.
  We can verify that
  $\widebar{\bm{x}}$
  is an {\rm F}-stationary point of
  $\cost+\phi$
  under Assumption~\ref{assumption:phi} by applying Theorem~\ref{theorem:characterization_stationarity_asymptotic} with
  $(\gamma_{m(l)})_{l=1}^{\infty} \subset (0,\gamma_{\rm init}]$,
  $\lim_{l\to\infty} \bm{x}_{m(l)} = \widebar{\bm{x}} \in \dom{\phi}$,
  and
  $\lim_{l\to\infty} \mathcal{M}_{\gamma_{m(l)}}^{\cost_{m(l)},\phi}(\bm{x}_{m(l)}) = 0$.
\end{proof}

\section{Concluding remarks} \label{sec:conclusion}
In this paper, for minimization of the sum of two nonconvex nonsmooth functions, we presented asymptotic properties of the gradient mapping-type stationarity measure
$\mathcal{M}_{\gamma}^{\cost,\phi}$,
and characterizations between the stationarities and
$\mathcal{M}_{\gamma}^{\cost,\phi}$.
By exploiting these properties, we revisited and refined convergence results for conventional proximal gradient algorithms.
More specifically, we provided an affirmative answer to an open question, raised by~\cite{Olikier-Waldspurger25}, on whether every cluster point of a sequence generated by the proximal gradient method under Assumptions~\ref{assumption:basic} and~\ref{assumption:smooth} is a {\rm P}-stationary point of
$\cost + \phi$.
Moreover, we presented a proximal variable smoothing algorithm with a nonmonotone linesearch for minimization of a nonconvex nonsmooth function in~\eqref{eq:problem} under Assumptions~\ref{assumption:basic},~\ref{assumption:regular} and~\ref{assumption:phi}.
We established an asymptotic convergence analysis of the proposed algorithm in terms of an {\rm F}-stationary point with the aid of an asymptotic property of
$\mathcal{M}_{\gamma}^{\cost,\phi}$.

Another promising direction of this work is to apply the proposed analysis, e.g., in Theorems~\ref{theorem:asymptotic_approximation} and~\ref{theorem:F_stationarity}, of the gradient mapping-type stationarity measure to proximal subgradient-type methods.
Indeed, the update
$\bm{x}_{n+1} \in \prox{\gamma_{n}\phi}(\bm{x}_{n} - \gamma_{n}\bm{v}_{n})$
with a subgradient
$\bm{v}_{n} \in \Lsubdiff \cost(\bm{x}_{n})$
is naturally associated with
$\mathcal{M}_{\gamma_{n}}^{\cost,\phi}(\bm{x}_{n})$,
because the residual
$\norm{(\bm{x}_{n}-\bm{x}_{n+1})/\gamma_{n}}$
provides an upper bound on this measure.
Proximal subgradient-type methods have been studied in, e.g.,~\cite{Bello17} for the sum of convex functions,~\cite{Davis-Drusvyatskiy19,Chen-Garcia-Shahrampour22,Pougkakiotis-Kalogerias23} for the sum of weakly convex functions,~\cite{Davis-Drusvyatskiy-Shi25} for the sum of a weakly convex function and the indicator function with a proximally smooth set, which implies the prox-regularity of that set, and~\cite{Solodov25} for the sum of locally Lipschitz continuous and lower regular functions, and the indicator function with a convex set.
It would be interesting to investigate whether the asymptotic analysis developed in this paper can yield stationarity guarantees for proximal subgradient-type methods
beyond convexity and weak convexity, especially under prox-regularity and
lower regularity assumptions.

Finally, for readers interested in numerical experiments,
we refer to our recent papers~\cite{Kume-Yamada25B,Kume-Yamada25C}:
\cite{Kume-Yamada25B} reports a fast convergence speed of the proposed algorithm under the convexity of
$\phi$
in numerical experiments;
\cite{Kume-Yamada25C} reports an effective estimation performance for a robust low-rank matrix recovery problem formulated into the problem~\eqref{eq:problem} under Assumptions~\ref{assumption:basic},~\ref{assumption:regular} and~\ref{assumption:phi}.

\section*{Declarations}
\noindent
{\bf Data availability.}
No datasets were generated or analyzed during the current study.

\noindent
{\bf Funding Information.}
This work was supported partially by the Japan Society for the Promotion of Science (JSPS) KAKENHI (24K23885, 26K21332).

\noindent
{\bf Competing Interests.}
The authors declare that they have no competing interests relevant to this article.

\appendix
\def\thesection{Appendix \Alph{section}}

\newcounter{appnum}
\setcounter{appnum}{0}

\setcounter{theorem}{0}
\renewcommand{\thetheorem}{\Alph{appnum}.\arabic{theorem}}

\stepcounter{theorem}
\stepcounter{appnum}

\section{Basic facts on set-valued analysis} \label{sec:appendix:set_valued}
We summarize basic facts on asymptotic analysis for set-valued mappings, including the proximity operator, together with notions of boundedness.
We also use
the {\em graphical outer limit}~\cite[Prop. 5.33]{Rockafellar-Wets98} of a sequence
$(\mathcal{S}_{n})_{n=1}^{\infty}$
of set-valued mappings
$\mathcal{S}_{n}:\mathcal{X}\rightrightarrows\mathcal{X}$:
\begin{align}
  (\widebar{\bm{x}}\in\mathcal{X}) \quad
  \gLimsup_{n\to\infty} \mathcal{S}_{n}(\widebar{\bm{x}})\coloneqq \bigcup_{\mathcal{X}\ni\bm{x}_{n}\to\widebar{\bm{x}}} \Limsup_{n\to\infty} \mathcal{S}_{n}(\bm{x}_{n}). \label{eq:gLimsup}
\end{align}

\begin{definition}[Boundedness; see a recent paper~{\cite{Lopez-Sama21}} for thorough review]
  \label{definition:bounded}
  \
  \begin{enumerate}[label=(\alph*)]
    \item
          (Eventually bounded) \label{enum:definition:bounded:eventually_bounded}
          A set sequence
          $(\mathcal{E}_{n})_{n=1}^{\infty}$
          of
          $\mathcal{E}_{n} \subset \mathcal{X}$
          is said to be {\em eventually bounded} if
          $\bigcup_{n=n_{0}}^{\infty} \mathcal{E}_{n}$
          is bounded with some
          $n_{0}\in\mathbb{N}$~\cite[Exm. 4.22]{Rockafellar-Wets98}.
    \item
          (Locally bounded) \label{enum:definition:bounded:locally_bounded}
          A set-valued mapping
          $\mathcal{S}:\mathcal{X} \rightrightarrows \mathcal{X}$
          is said to be {\em locally bounded} at
          $\widebar{\bm{x}} \in \mathcal{X}$
          if
          $\mathcal{S}(\mathcal{N}_{\widebar{\bm{x}}}) \subset \mathcal{X}$
          is bounded with some open neighborhood
          $\mathcal{N}_{\widebar{\bm{x}}} \subset \mathcal{X}$
          of
          $\widebar{\bm{x}}$~\cite[Def. 5.14]{Rockafellar-Wets98}.
          $\mathcal{S}$
          is called locally bounded if
          $\mathcal{S}$
          is locally bounded at every
          $\widebar{\bm{x}} \in \mathcal{X}$,
          or equivalently if
          $\mathcal{S}(\mathcal{E}) \subset \mathcal{X}$
          is bounded for every bounded set
          $\mathcal{E}\subset\mathcal{X}$~\cite[Prop. 2.3 (a)]{Lopez-Sama21}.
  \end{enumerate}
\end{definition}

For
$\cost$
and
$\phi$
in Assumption~\ref{assumption:basic},
$\Lsubdiff\cost$
and
$\prox{\gamma\phi}$
are typical examples of locally bounded set-valued mappings as follows.

\begin{fact}[Properties of $\Lsubdiff \cost$]
  \label{fact:subdifferential_nonempty}
  For
  $\cost:\mathcal{X} \to \exR$
  satisfying Assumption~\ref{assumption:basic}~\ref{enum:problem:origin:cost}
  and for
  $\widebar{\bm{x}} \in \mathrm{int}(\dom{\cost})$,
  $\Lsubdiff \cost (\widebar{\bm{x}})$
  is nonempty and compact, and
  $\Lsubdiff \cost$
  is locally bounded
  and outer semicontinuous at
  $\widebar{\bm{x}}$.
  See~\cite[Prop. 8.7 and Thm. 9.13]{Rockafellar-Wets98} with the local Lipschitz continuity
  (or equivalently, strict continuity) of
  $\cost$
  at
  $\widebar{\bm{x}}$.
\end{fact}

\begin{fact}[Properties of the proximity operator]
  \label{fact:prox_set_mapping}
  For
  $\phi:\mathcal{X} \to \exR$
  satisfying Assumption~\ref{assumption:basic}~\ref{enum:problem:origin:phi},
  the following hold:
  \begin{enumerate}[label=(\alph*)]
    \item
          \label{enum:fact:prox_set_mapping:nonempty_resolvent}
          For every
          $\gamma \in (0,\gamma_{\phi})$
          and
          $\widebar{\bm{x}} \in \mathcal{X}$,
          $\prox{\gamma \phi}(\widebar{\bm{x}})$
          is nonempty and compact, and
          $\prox{\gamma \phi}(\widebar{\bm{x}}) \subset (\Id+\gamma\Lsubdiff \phi)^{-1}(\widebar{\bm{x}})$
          holds~\cite[Thm. 1.25 and Exm. 10.2]{Rockafellar-Wets98},
          where the latter ensures
          $\frac{\widebar{\bm{x}} - \bm{p}}{\gamma} \in \Lsubdiff \phi(\bm{p})$
          for
          $\bm{p} \in \prox{\gamma\phi}(\widebar{\bm{x}})$.
    \item
          \label{enum:fact:prox_set_mapping:locally_bounded}
          For every
          $\gamma \in (0,\gamma_{\phi})$
          and
          $\widebar{\bm{x}} \in \mathcal{X}$,
          $\prox{\gamma\phi}$
          is locally bounded and outer semicontinuous at
          $\widebar{\bm{x}}$~\cite[Exm. 5.23 (b)]{Rockafellar-Wets98}.
    \item
          \label{enum:fact:prox_set_mapping:eventually_locall_bounded}
          Let
          $(\gamma_{n})_{n=1}^{\infty} \subset (0,\gamma_{\phi})$
          converge to some
          $\widebar{\gamma} \in [0,\gamma_{\phi})$.
          Then, we have
          \begin{equation}
            (\widebar{\gamma} (= \lim_{n\to\infty}\gamma_{n}) \in (0,\gamma_{\phi}),\  \widebar{\bm{x}}\in\mathcal{X})\quad
            \gLimsup_{n\to\infty}\prox{\gamma_{n}\phi}(\widebar{\bm{x}})
            \subset \prox{\widebar{\gamma}\phi}(\widebar{\bm{x}}) \label{eq:gLimsup_prox_positive_gamma}
          \end{equation}
          from~\cite[Exe. 7.38]{Rockafellar-Wets98}.
          Moreover, if
          $(\gamma_{n})_{n=1}^{\infty} \subset (0,\gamma_{\phi})$
          converges to
          $0$
          and
          $(\bm{x}_{n})_{n=1}^{\infty}\subset \mathcal{X}$
          converges to some
          $\widebar{\bm{x}} \in \dom{\phi}$,
          then every sequence
          $\bm{p}_{n} \in \prox{\gamma_{n}\phi}(\bm{x}_{n}) \ (n\in\mathbb{N})$
          converges to
          $\widebar{\bm{x}}$~\cite[Lemma 2.3]{Perez-Aroz-Torregrosa-Belen25}, from which we obtain
          $\Limsup_{n\to\infty}\prox{\gamma_{n}\phi}(\bm{x}_{n}) = \{\widebar{\bm{x}}\}$,
          and thus
          \begin{equation}
            (\widebar{\gamma} (=\lim_{n\to\infty}\gamma_{n}) = 0,\  \widebar{\bm{x}}\in\dom{\phi})\quad
            \gLimsup_{n\to\infty}\prox{\gamma_{n}\phi}(\widebar{\bm{x}})
            = \{\widebar{\bm{x}}\}. \label{eq:gLimsup_prox_zero_gamma}
          \end{equation}
    \item
          \label{enum:fact:prox_set_mapping:prox_single}
          Assume that
          $\phi$
          is prox-regular at
          $\widebar{\bm{x}} \in \dom{\phi}$
          for
          $\widebar{\bm{v}} \in \Lsubdiff \phi(\widebar{\bm{x}})$.
          Then,
          there exists
          $\gamma_{\widebar{\bm{x}}} \in (0,\gamma_{\phi})$
          such that, for each
          $\gamma \in (0,\gamma_{\widebar{\bm{x}}}]$,
          $\prox{\gamma \phi}$
          is single-valued and Lipschitz continuous over
          $B(\widebar{\bm{x}}+\gamma\widebar{\bm{v}}, \tau_{\gamma})$
          with some
          $\tau_{\gamma} > 0$~\cite[Thm. 4.4]{Poliquin-Rockafellar96}~\cite[Thm. 2.4]{Hare-Poliquin05}.
          In particular, we have
          $\prox{\gamma \phi}(\widebar{\bm{x}} + \gamma \widebar{\bm{v}}) = \{\widebar{\bm{x}}\}\ (\forall \gamma \in (0,\gamma_{\widebar{\bm{x}}}])$~\cite[Thm. 2.3]{Hare-Poliquin05}.
  \end{enumerate}
\end{fact}

\begin{fact}[{\cite[Exe. 4.8]{Rockafellar-Wets98}}]
  \label{fact:outer}
  For any
  $\widebar{\bm{x}} \in \mathcal{X}$,
  and any sequence
  $(\mathcal{E}_{n})_{n=1}^{\infty}$
  of subsets
  $\mathcal{E}_{n} \subset \mathcal{X}$,
  $\liminf_{n\to\infty} \dist(\widebar{\bm{x}},\mathcal{E}_{n}) = \dist\left(\widebar{\bm{x}},\Limsup_{n\to\infty}\mathcal{E}_{n}\right)$
  holds.
\end{fact}

The following two lemmas are elementary consequences of outer limits and graphical outer limits, where they are used only to make the proof of Theorem~\ref{theorem:asymptotic_approximation}, in particular of Claim~\ref{claim:set_valued_elementary}.
Similar results to Lemma~\ref{lemma:sum_rule_outer_limit} and~\ref{lemma:convergence_mapping_set} are found respectively in~\cite[Exe.
  4.29 (c)]{Rockafellar-Wets98} and~\cite[Prop. 4.15 (b)]{Lopez-Sama21}.
For self-containedness, we present their proofs because these lemmas are used in our analysis repeatedly.

\begin{lemma}[Sum rule for outer limit under the eventual boundedness]
  \label{lemma:sum_rule_outer_limit}
  For set sequences
  $\mathcal{D}_{n} \subset \mathcal{X}$
  and
  $\mathcal{E}_{n} \subset \mathcal{X}\ (n\in \mathbb{N})$,
  we have
  $\Limsup_{n\to\infty}(\mathcal{D}_{n} + \mathcal{E}_{n}) \subset \Limsup_{n\to\infty}\mathcal{D}_{n} + \Limsup_{n\to\infty} \mathcal{E}_{n}$
  if
  $(\mathcal{D}_{n})_{n=1}^{\infty}$
  is eventually bounded.
\end{lemma}
\begin{proof}
  Let
  $\bm{v} \in \Limsup_{n\to\infty} (\mathcal{D}_{n}+\mathcal{E}_{n})$,
  i.e., there exist
  $\mathcal{N} \in \Ninfty$,
  $\bm{d}_{n} \in \mathcal{D}_{n}$,
  and
  $\bm{e}_{n} \in \mathcal{E}_{n}\ (n\in\mathcal{N})$
  satisfying
  $\lim_{\mathcal{N}\ni n\to\infty} (\bm{d}_{n}+\bm{e}_{n}) = \bm{v}$.
  Since
  $(\mathcal{D}_{n})_{n=1}^{\infty}$
  is eventually bounded,
  $(\bm{d}_{n})_{n\in\mathcal{N}}$
  is bounded.
  Without loss of generality, we can assume that
  $(\bm{d}_{n})_{n\in\mathcal{N}}$
  converges to some
  $\bm{d} \in \Limsup_{n\to\infty} \mathcal{D}_{n}$
  (by passing through further subsequence of
  $(\bm{d}_{n})_{n\in\mathcal{N}}$
    if necessary).
    By using this subsequence,
    we get
  $\bm{v}-\bm{d} = \lim_{\mathcal{N}\ni n\to\infty} (\bm{d}_{n}+\bm{e}_{n}) - \lim_{\mathcal{N} \ni n\to\infty} \bm{d}_{n} = \lim_{\mathcal{N}\ni n\to\infty}\bm{e}_{n} \in \Limsup_{n\to\infty} \mathcal{E}_{n}$.
    Hence,
  $\bm{v} = \bm{d} + (\bm{v}-\bm{d}) \in \Limsup_{n\to\infty} \mathcal{D}_{n} + \Limsup_{n\to\infty} \mathcal{E}_{n}$
  holds.
\end{proof}

\begin{lemma}[Outer limit of image sequence of set-valued mappings]
  \label{lemma:convergence_mapping_set}
  Let
  $\mathcal{S}:\mathcal{X} \rightrightarrows \mathcal{X}$
  and
  $\mathcal{S}_{n}:\mathcal{X} \rightrightarrows \mathcal{X}\ (n\in \mathbb{N})$.
  For subsets
  $\mathcal{E} \subset \mathcal{X}$
  and
  $\mathcal{E}_{n} \subset \mathcal{X}\ (n\in \mathbb{N})$,
  assume
  (i)
  $\Limsup_{n\to\infty} \mathcal{E}_{n} \subset \mathcal{E}$,
  (ii)
  $(\mathcal{E}_{n})_{n=1}^{\infty}$
  is eventually bounded,
  and
  (iii)
  $\gLimsup_{n\to\infty} \mathcal{S}_{n}(\widebar{\bm{x}}) \subset \mathcal{S}(\widebar{\bm{x}})\ (\widebar{\bm{x}} \in \mathcal{E})$,
  where the graphical outer limit
  $\gLimsup\mathcal{S}_{n}$
  of
  $(\mathcal{S}_{n})_{n=1}^{\infty}$
  is defined in~\eqref{eq:gLimsup}.
  Then, we have
  $\Limsup_{n\to\infty} \mathcal{S}_{n}(\mathcal{E}_{n}) \subset \mathcal{S}(\mathcal{E})$.
\end{lemma}
\begin{proof}
  Let
  $\bm{v} \in \Limsup_{n\to\infty} \mathcal{S}_{n}(\mathcal{E}_{n})$,
  i.e., there exist
  $\mathcal{N} \in \Ninfty$
  and
  $\bm{v}_{n} \in \mathcal{S}_{n}(\mathcal{E}_{n})$
  $(n\in\mathcal{N})$
  such that
  $\lim_{\mathcal{N}\ni n\to \infty} \bm{v}_{n} = \bm{v}$.
  For
  $n\in\mathcal{N}$,
  we have
  $\bm{v}_{n} \in \mathcal{S}_{n}(\bm{u}_{n})$
  with some
  $\bm{u}_{n} \in \mathcal{E}_{n}$,
  i.e.,
  $\bm{v} \in \Limsup_{\mathcal{N}\ni n\to\infty} \mathcal{S}_{n}(\bm{u}_{n})$.
  Since
  $(\mathcal{E}_{n})_{n=1}^{\infty}$
  is eventually bounded by the condition (ii),
  the sequence
  $(\bm{u}_{n})_{n\in\mathcal{N}}$
  is bounded.
  Hence, we can assume that
  $(\bm{u}_{n})_{n\in\mathcal{N}}$
  converges to some
  $\widebar{\bm{u}} \in \Limsup_{n\to\infty} \mathcal{E}_{n} \subset \mathcal{E}$
  (by passing through further subsequence of
  $(\bm{u}_{n})_{n\in\mathcal{N}}$
    if necessary),
    where the last inclusion holds by the condition (i).
    Then,
    the condition (iii)
    ensures
  $\Limsup_{\mathcal{N}\ni n\to\infty}\mathcal{S}_{n}(\bm{u}_{n}) \subset \gLimsup_{\mathcal{N}\ni n\to\infty}\mathcal{S}_{n}(\widebar{\bm{u}}) \subset \gLimsup_{n\to\infty}\mathcal{S}_{n}(\widebar{\bm{u}})\subset \mathcal{S}(\widebar{\bm{u}}) \subset \mathcal{S}(\mathcal{E})$.
    From
  $\bm{v} \in \Limsup_{\mathcal{N}\ni n\to\infty} \mathcal{S}_{n}(\bm{u}_{n})$,
    we have
  $\bm{v} \in \mathcal{S}(\mathcal{E})$.
\end{proof}

\stepcounter{theorem}
\stepcounter{appnum}

  {
    \bibliography{main}}

\end{document}

%% file: preamble.tex
\usepackage{graphicx}%
\usepackage{multirow}%

\usepackage{amsmath,amssymb,amsfonts}%
\usepackage{amssymb,amsfonts}%
\usepackage{amsthm}%
\usepackage{mathrsfs}%
\usepackage{xcolor}%
\usepackage{textcomp}%
\usepackage{manyfoot}%
\usepackage{booktabs}%
\usepackage{algorithm}%
\usepackage{algorithmicx}%
\usepackage{algpseudocode}%
\renewcommand{\figurename}{Fig.}

\usepackage{threeparttable}

\usepackage{epstopdf}%

\usepackage{here}
\usepackage{bm}
\usepackage{algorithm}
\usepackage{algpseudocode}
\usepackage{comment}
\usepackage{cases}
\usepackage{mathtools}
\usepackage{enumitem}
\usepackage{mathabx}
\usepackage{stmaryrd}
\usepackage{tikz}
\usetikzlibrary{arrows.meta}
\usetikzlibrary{positioning}
\usetikzlibrary{decorations.pathreplacing,calc}
\usepackage{csvsimple}
\usepackage{lscape}

\usepackage{etoolbox}

\usepackage{tabularx}
\usepackage{booktabs}
\newcommand{\argmin}{\mathop{\mathrm{argmin}}}

\newcommand{\inprod}[2]{{\left\langle #1,#2 \right\rangle}}

\newcommand{\TT}{\mathsf{T}}

\newcommand{\prox}[1]{\mathrm{prox}_{#1}}

\newcommand{\moreau}[2]{{}^{#2}#1}

\newcommand{\lev}[2]{\mathrm{lev}_{\leq #2}#1}

\newcommand{\exR}{\mathbb{R}\cup \{+\infty\}}
\newcommand{\exRp}{(0,+\infty]}
\newcommand{\dom}[1]{\mathrm{dom}(#1)}

\newcommand{\Limsup}{\mathop{\mathrm{Lim~sup}}}
\newcommand{\gLimsup}{\mathop{\mathrm{g\mathchar`-Lim~sup}}}

\newcommand{\norm}[1]{\left\lVert #1 \right\rVert}
\newcommand{\Id}{\mathrm{Id}}

\newcommand{\Psubdiff}{\partial_{\rm P}}
\newcommand{\Fsubdiff}{\partial_{\rm F}}
\newcommand{\Lsubdiff}{\partial_{\rm L}}

\newcommand{\Conv}{\mathop{\mathrm{Conv}}}
\newcommand{\dist}{\mathop{\mathrm{dist}}}

\makeatletter
\newcommand{\doublewidetilde}[1]{{%
  \mathpalette\double@widetilde{#1}%
}}
\newcommand{\double@widetilde}[2]{%
  \sbox\z@{$\m@th#1\widetilde{#2}$}%
  \ht\z@=.9\ht\z@
  \widetilde{\box\z@}%
}
\makeatother

\setlist[trivlist]{topsep=0.3em plus 2pt minus 2pt} %

\theoremstyle{plain}%
\newtheorem{theorem}{Theorem}[section]
\newtheorem{lemma}[theorem]{Lemma}
\newtheorem{corollary}[theorem]{Corollary}
\newtheorem{proposition}[theorem]{Proposition}
\newtheorem{claim}[theorem]{Claim}

\theoremstyle{definition}
\newtheorem{definition}[theorem]{Definition}
\newtheorem{fact}[theorem]{Fact}
\newtheorem{example}[theorem]{Example}
\newtheorem{assumption}[theorem]{Assumption}

\newtheorem{remark}[theorem]{Remark}

\theoremstyle{remark}

\DeclarePairedDelimiter{\abs}{\lvert}{\rvert}